\documentclass[reqno]{amsart}

\usepackage{setspace}

\usepackage{hyperref,todonotes,comment}
\usepackage[color,matrix,arrow]{xy}
\usepackage{graphics,amssymb}
\usepackage{pagecolor,lipsum}
\usepackage{latexsym}
\usepackage{mathrsfs}
\xyoption{curve}
\usepackage{hyperref}
\hypersetup{ 
colorlinks,
citecolor=cyan,
filecolor=cyan,
linkcolor=cyan,
urlcolor=cyan
}

\newtheorem{lemma}{Lemma}[section]
\newtheorem{proposition}[lemma]{Proposition}
\newtheorem{theorem}[lemma]{Theorem}
\newtheorem{corollary}[lemma]{Corollary}

\newtheorem{conjecture}[lemma]{Conjecture}
\newtheorem*{theoremA}{Theorem}
\newtheorem*{corollaryA}{Corollary}

\theoremstyle{definition}
\newtheorem{example}[lemma]{Example}
\newtheorem{definition}[lemma]{Definition}
\newtheorem{remark}[lemma]{Remark}

\newtheorem{notation}[lemma]{Notation}
\newtheorem{observation}[lemma]{Observation}

\newcommand{\mfk}[1]{\mathfrak{#1}}
\newcommand{\mbb}[1]{\mathbb{#1}}
\newcommand{\mcl}[1]{\mathcal{#1}}

\newcommand{\msc}[1]{\mathscr{#1}}
\newcommand{\mbf}[1]{\mathbf{#1}}

\DeclareMathOperator{\Hom}{Hom}
\DeclareMathOperator{\End}{End}
\DeclareMathOperator{\RHom}{RHom}
\DeclareMathOperator{\REnd}{REnd}
\DeclareMathOperator{\Ext}{Ext}
\DeclareMathOperator{\Aut}{Aut}
\DeclareMathOperator{\Ad}{Ad}
\DeclareMathOperator{\rep}{rep}
\DeclareMathOperator{\Rep}{Rep}
\DeclareMathOperator{\Spec}{Spec}
\DeclareMathOperator{\Proj}{Proj}
\DeclareMathOperator{\Enh}{Enh}
\DeclareMathOperator{\Supp}{Supp}
\DeclareMathOperator{\Fr}{Fr}
\DeclareMathOperator{\stab}{stab}
\DeclareMathOperator{\Stab}{Stab}
\DeclareMathOperator{\ord}{ord}
\DeclareMathOperator{\Sym}{Sym}
\DeclareMathOperator{\FK}{FK}
\DeclareMathOperator{\res}{res}
\DeclareMathOperator{\supp}{supp}

\DeclareMathOperator{\coh}{coh}
\DeclareMathOperator{\Coh}{Coh}
\DeclareMathOperator{\QCoh}{QCoh}

\DeclareMathOperator{\Kempf}{Kempf}
\DeclareMathOperator{\Surf}{Surf}

\newcommand{\Sp}{\mathtt{Spr}}
\newcommand{\dg}{\mathrm{dg}}
\newcommand{\de}{\operatorname{d\acute{e}}}
\newcommand{\sEnd}{\mathscr{E}\! nd}
\newcommand{\sHom}{\mathscr{H}\! om}
\newcommand{\sExt}{\mathscr{E}xt}
\newcommand{\sRHom}{\mathscr{R\!H}\!om}

\newcommand{\ot}{\otimes}
\newcommand{\tN}{\tilde{\mathcal{N}}}

\newcommand{\uqG}{u(G_q)}
\newcommand{\uqB}{u(B_q)}
\newcommand{\su}{u}
\newcommand{\dG}{\check{G}}
\newcommand{\dB}{\check{B}}
\renewcommand{\1}{\mathbf{1}}
\renewcommand{\O}{\mathscr{O}}
\renewcommand{\tilde}{\widetilde}

\definecolor{page_color}{HTML}{000000}
\definecolor{text_color}{HTML}{fffff0}

\title[]{Support theory for the small quantum group and the Springer resolution}
\date{\today}

\author{Cris Negron}
\address{Department of Mathematics, University of Southern California, Los Angeles, CA 90007}
\email{cnegron@usc.edu}

\author{Julia Pevtsova}
\address{Department of Mathematics, University of Washington, Seattle, WA 98195}
\email{julia@math.washington.edu}

\begin{document}

\maketitle

\begin{abstract}
We consider the small quantum group $u(G_q)$, for an almost-simple algebraic group $G$ over the complex numbers and a root of unity $q$ of sufficiently large order.  We show that the Balmer spectrum for the small quantum group in type $A$ admits a continuous surjection $\mbb{P}(\tN)\to \Spec(\stab u(G_q))$ from the (projectivized) Springer resolution.  This surjection is shown to be a homeomorphism over a dense open subset in the spectrum.  In type $A_1$ we calculate the Balmer spectrum precisely, where it is shown to be the projectivized nilpotent cone.  Our results extend to arbitrary Dynkin type provided certain conjectures hold for the small quantum Borel.  At the conclusion of the paper we touch on relations with geometric representation theory and logarithmic TQFTs, as represented in works of Arkhipov-Bezrukavnikov-Ginzburg and Schweigert-Woike respectively.
\end{abstract}

\tableofcontents

\section{Introduction}

Let $G$ be an almost-simple algebraic group over the complex numbers and let $q\in \mbb{C}$ be a root of unity of sufficiently large odd order.\footnote{This restriction to odd order $q$ is inessential.  See \ref{sect:even_orderq}.}  We consider representations $\Rep G_q$ of Lusztig's divided power quantum algebra, and the associated small quantum group $\uqG$ for $G$ at $q$.
\par

A primary goal of this work is to provide a rigorous analysis of support theory for the small quantum group.  Such support theoretic investigations include, but are not limited to, a calculation of the Balmer spectrum for the stable category of finite-dimensional representations
\[
\stab \uqG=D^b(\rep \uqG)/\langle\text{bounded complexes of projectives}\rangle.
\]
As introduced in \cite{balmer05}, the Balmer spectrum $\Spec(\stab \uqG)$ is the spectrum of thick, prime tensor ideals in the stable category.  Hence a calculation of the spectrum provides a global picture of the ``geometry" of $\stab\uqG$, as a monoidal category.  Such calculations represent just one point within a much broader engagement between geometry and tensor categories of representations, however, both in the classical and quantum settings.  One might think of support theory liberally as this broad engagement between these two subjects.\footnote{At least, this representes one \emph{branch} of what one might call support theory.}
\par

In the present paper, informed by principles coming from geometric representation theory, we argue that a natural setting for support theory in the quantum group context is on the Springer resolution $\tN \to \mathcal{N}$  of the nullcone $\mathcal{N}$ for $G$. We address support for $\uqG$ by introducing sheaves on the Springer resolution into an analysis of quantum group representations, at the derived level.  This echos results of Arkhipov-Bezrukavnikov-Ginzburg and Bezrukavnikov-Lachowska \cite{arkhipovbezrukavnikovginzburg04,bezrukavnikovlachowska07} which show that the principal block in the derived category of representations for $\uqG$ is actually equivalent to the derived category of (dg) sheaves on the Springer resolution.  A main difference here is that we incorporate the monoidal structure into our study. We elaborate on the proposed melding of these two topics in Section \ref{sect:intro_grt} below.
\par

At the heart of this work is a certain \emph{approach} to the quantum group.  Namely, we propose a means of studying the quantum group via a global object called the \emph{half-quantum flag variety}.  The Springer resolution is, for us, a derived descendant of this half-quantum flag variety.  
\par

In the introduction, we first discuss our explicit results for support theory in Sections \ref{sect:intro_support} and \ref{sect:more!}, then return to a discussion of the half-quantum flag variety in Section \ref{sect:231}.  We describe connections to other areas of mathematics in Sections \ref{sect:intro_grt} and \ref{sect:intro_logtqft}, where we consider geometric representation theory as well as logarithmic TQFTs.  These various portions of the introduction represent the contents of Parts 2, 1, and 3 of the text respectively.

\subsection{Results in support theory}
\label{sect:intro_support}

Calculating the Balmer spectrum, for a given braided tensor category, is quite a difficult problem in general.  The first examples came from geometry and topology, where the topic still provides an active area of interest.  For an incomplete sampling one can see \cite{thomason97,hopkinssmith98,stevenson14,hall16,hirano19,gallauer19}.  Most representation theoretic cases where this spectrum has been calculated come from finite groups and finite group schemes \cite{bensoncarlsonrickard97,balmer05,friedlanderpevtsova07}, although there are examples which lie slightly outside of this domain as well \cite{boekujawanakano17,drupieskikujawa19,bensoniyengarkrausepevtsova,friedlandernegron}.
\par

In the present text we show that the Balmer spectrum for the small quantum group is essentially resolved by the Springer resolution, at least in type $A$.

\begin{theoremA}[{\ref{thm:bspec2}/\ref{thm:bspec2.5}}]
Let $G$ be an almost-simple algebraic group in type $A$.  There is a continuous surjective map of topological spaces
\[
f:\mbb{P}(\tN)\to \Spec\big(\stab \uqG\big)
\]
which restricts to a homeomorphism $\mbb{P}(\tN)_{\rm reg}\overset{\sim}\to \Spec(\stab \uqG)_{\rm reg}$ over a dense open subset in $\Spec(\stab\uqG)$.  Furthermore, the map $f$ factors the moment map for $\tN$,
\[
\mbb{P}(\tN)\overset{f}\longrightarrow \Spec\big(\stab \uqG\big)\overset{\rho}\longrightarrow \mbb{P}(\mcl{N}).
\]
The same calculations hold in arbitrary Dynkin type, provided some specific conjectures hold for the quantum Borel.
\end{theoremA}

Here $\mbb{P}(\tN)$ and $\mbb{P}(\mcl{N})$ denote the projectivizations of the conical varieties $\tN$ and $\mcl{N}$, respectively.  Also, by factoring the moment map we mean that the composite $\rho\circ f$ is equal to the projectivization of the moment map $\kappa:\tN\to \mcl{N}$.
\par

For the quantum Borel outside of type $A$, one simply wants to know that cohomological support for $\uqB$ classifies thick ideals in the associated stable category.  More directly, these ``specific conjectures" claim that cohomological support for the small quantum Borel is a lavish support theory, in the language of \cite{negronpevtsova3}.  Such a result was proved for the quantum Borel in type $A$ in \cite{negronpevtsova3}, and is expected to hold in general (see Section \ref{sect:the_others}).
\par

We note that Theorem \ref{thm:bspec2} specializes in type $A_1$ to provide a complete calculation of the Balmer spectrum.

\begin{corollaryA}[\ref{cor:bspecA1}]
For small quantum $\operatorname{SL}(2)$, the Balmer spectrum is precisely the projectivized nilpotent cone
\[
\Spec\big(\stab u(\operatorname{SL}(2)_q)\big)\overset{\cong}\to \mbb{P}(\mcl{N}).
\]
\end{corollaryA}

We do expect that, in general, the Balmer spectrum for $\uqG$ is just the projectivized nilpotent cone.  This is a point of significant interest, and one \emph{can} provide precise conditions under which $\tN$-support for the quantum group collapses to identify the Balmer spectrum with $\mbb{P}(\mcl{N})$.  The interested reader can see Section \ref{sect:3286} for a detailed discussion.
\par

In regards to the proofs of Theorem \ref{thm:bspec2} and Corollary \ref{cor:bspecA1}, there are two important constructions which we employ.  Both of these constructions are as functional in type $A$ as they are in arbitrary Dynkin type.  So one needn't concern themselves with Dynkin types for the remainder of the introduction.

\subsection{More support theory!}
\label{sect:more!}
As we have just stated, there are two main constructions which we employ in the proofs of the above theorems.  First, we produce a natural support theory for the quantum group which takes values in the Springer resolution
\[
\supp_{\tN}:\big\{\text{finite-dimensional $\uqG$-representations}\big\}\to \big\{\text{closed subsets in }\mbb{P}(\tN)\big\}.
\]
This $\tN$-support is defined via the supports of certain $\mbb{G}_m$-equivariant cohomology sheaves over the Springer resolution, which one associates to $\uqG$-representations (see Section \ref{sect:supp_tN}).  The construction of this sheaf-valued cohomology follows from considerations of the half-quantum flag variety, as described in Section \ref{sect:231} below.
\par

Our second construction is that of a complete family of small quantum Borels for the small quantum group.  These quantum Borels appear as a family of $\Rep \uqG$-central tensor categories
\[
\res_\lambda:\Rep\uqG\to \msc{B}_\lambda
\]
which are parametrized by points on the flag variety $G/B$.
\par

At the identity $\lambda=1$, the category $\msc{B}_1$ recovers the standard small quantum Borel $\Rep \uqB$ as a $\Rep \uqG$-central tensor category.  Furthermore, the collection of quantum Borels $\{\msc{B}_\lambda:\text{points $\lambda$ for }G/B\}$ admits a natural transitive $G$-action by tensor equivalences under which each $\msc{B}_\lambda$ is stabilized precisely by its corresponding Borel subgroup $B^{\lambda}$ in $G$.   In particular, each $\msc{B}_\lambda$ is isomorphic to the standard positive quantum Borel as a $k$-linear tensor category, and hence is a finite tensor category of (Frobenius-Perron) dimension $\sim l^{\operatorname{rank}(\mfk{g})+|\Phi^+|}$, where $l=\operatorname{ord}(q)$.  The categories $\msc{B}_\lambda$ provide the necessary quantum replacements for the family of Borels $\mbb{B}_{\lambda,(1)}\subset \mbb{G}_{(1)}$ in the corresponding first Frobenius kernel in finite characteristic (see Section \ref{sect:borels}).

\begin{remark}
Many of the $G$'s and $B$'s above should really be dual groups $\dG$ and $\dB$.  See Section \ref{sect:Gq}.  We ignore this subtlety for the introduction.
\end{remark}

These two constructions, $\tN$-support and our $G/B$-family of Borels for the small quantum group, intertwine in interesting ways.  For example, $\tN$-support satisfies a naturality property with respect to the restriction functors $\res_\lambda:\Rep\uqG\to \msc{B}_\lambda$.  In the following statement $i_\lambda:\mfk{n}_\lambda\to \tN$ denotes the nilpotent subalgebra over a given point $\lambda:\Spec(K)\to G/B$, and $\supp^{\operatorname{chom}}_{\mfk{n}_\lambda}$ denotes cohomological support for $\msc{B}_\lambda$.

\begin{theoremA}[{\ref{thm:natrl}}]
Let $G$ be an almost-simple algebraic group (of arbitrary Dynkin type), and let $V$ be any finite-dimensional $\uqG$-representation.  Pulling back $\tN$-support along the map $i_\lambda$ recovers cohomological support for $\msc{B}_\lambda$,
\[
\mbb{P}(i_\lambda)^{-1}\big(\supp_{\tN}(V)\big)=\supp^{\operatorname{chom}}_{\mfk{n}_\lambda}(\res_\lambda V).
\]
\end{theoremA}

We recall that cohomological support for a given finite tensor category $\msc{C}$ is defined via the tensor actions of $\Ext_\msc{C}(k,k)$ on $\Ext_\msc{C}(V,V)$, for varying $V$ in $\msc{C}$, and takes values in the corresponding projective spectrum $\Proj\Ext_\msc{C}(k,k)$.  For the small quantum group and small quantum Borels $\msc{B}_\lambda$, these projective spectra are the projectivized nilpotent cone $\mcl{N}$ and nilpotent subalgebras $\mfk{n}_\lambda$ respectively.
\par

The point of Theorem \ref{thm:natrl} is that $\tN$-support for the small quantum group can be reconstructed from cohomological support over the Borels.  Hence we can reduce a study of support for $\uqG$ to a corresponding study for the categories $\msc{B}_\lambda$.
\par

In addition to the naturality properties with respect to the restriction, recalled above, $\tN$-support has strong connections to cohomological support for the small quantum group as well.  At Theorem \ref{thm:tN_N} we show that $\tN$-support localizes cohomological support for $\uqG$, in the sense that pushing forward along the moment map $\kappa:\tN\to \mcl{N}$ provides an equality
\[
\kappa\left(\supp_{\tN}(V)\right)=\supp^{\operatorname{chom}}_{\mathcal{N}}(V).
\]
Here, as expected, $\supp^{\operatorname{chom}}_{\mathcal{N}}$ denotes cohomological support for the small quantum group.
\par

In total, we achieve the approximate calculation of the Balmer spectrum $\mbb{P}(\tN)\to \Spec(\stab\uqG)$ appearing in Theorem \ref{thm:bspec2} via an application of the naturality result of Theorem \ref{thm:natrl}, the analysis of the quantum Borel provided in \cite{negronpevtsova3} (where the type $A$ bottleneck occurs), and a projectivity test for infinite-dimensional quantum group representations via the Borels, which we now recall.

\begin{theoremA}[{\ref{thm:projectivity}}]
Let $G$ be an almost-simple algebraic group. A (possibly infinite-dimensional) $\uqG$-representation $V$ is projective if and only if its restriction $\res_\lambda V$ at each geometric point $\lambda:\Spec(K)\to G/B$ is projective in $\msc{B}_\lambda$.
\end{theoremA}

For the finite characteristic analog of this result one can see \cite{pevtsova02}.  The proof given here is conceptually different from the one given in \cite{pevtsova02}.
\par

Having relayed our support theoretic results, we now turn our attention to $\tN$-support itself, and a preceding localization of $\Rep\uqG$ over the flag variety.

\subsection{Localizing $\Rep G_q$ via the half-quantum flag variety}
\label{sect:231}

Our construction of $\tN$-support relies fundamentally on a monoidal localization of the category $\Rep\uqG$ of small quantum group representations, over the flag variety.  Formally, this localization appears as a monoidal enhancement for $\Rep\uqG$ in the category of quasi-coherent sheaves over $G/B$ (see Section \ref{sect:enrich}).  We explain here how such an enhancement for the small quantum group comes into being.  Our implicit claim is that such a localization arises naturally when one studies quantum group representations via the quantum Borels $\msc{B}_\lambda$.
\par

As a starting point, we consider the category $\QCoh(G/B_q)$ of $B_q$-equivariant sheaves on $G$.  Here one might consider $B_q$ as a noncommutatively ringed space, and we act on $G$ via the quantum Frobenius $\Fr:B_q\to B$ composed with the translation action of $B$ on $G$.  The category $\QCoh(G/B_q)$ is the category of sheaves for our \emph{half-quantum flag variety} (cf.\ \cite{backelinkremnizer08}).
\par

The category $\QCoh(G/B_q)$ is monoidal under the product $\ot=\ot_{\O_{G}}$, and has an additional $\QCoh(G/B)$-linear structure provided by a central tensor embedding $\zeta^\ast:\QCoh(G/B)\to \QCoh(G/B_q)$ from the category of sheaves on the classical flag variety.  This central embedding is obtained directly from the quantum Frobenius functor.  We think of $\QCoh(G/B_q)$ as a flat sheaf of tensor categories over $G/B$ whose fibers are the small quantum Borels $\msc{B}_\lambda$ introduced in Section~\ref{sect:more!} (see Sections \ref{sect:borels} \& \ref{sect:fibers_G/B}).
\par

We show in Section \ref{sect:Kempf} that the restriction functors $\res_\lambda:\Rep\uqG\to \msc{B}_\lambda$ all ``glue" to form a fully faithful, exact monoidal embedding $\Rep\uqG\to \QCoh(G/B_q)$ into the category of sheaves for the half-quantum flag variety.  We show furthermore that the induced functor on unbounded derived categories remains monoidal and fully faithful
\[
\Kempf:D(\uqG)\hookrightarrow D(G/B_q).
\]
We refer to the above embedding as the Kempf embedding.
\par

Now, the subcategory $\Rep\uqG\subset \QCoh(G/B_q)$ is \emph{not} stable under the ambient $\QCoh(G/B)$-action, and so does not inherit a sheaf structure over the flag variety.  However, taking inner morphisms for the $\QCoh(G/B)$-action on $\QCoh(G/B_q)$ provides a tensor categorical enhancement $\sHom_{G/B_q}$ for $\QCoh(G/B_q)$ over $G/B$ (cf.\ \cite[\S 3.2]{etingofostrik04}).  This enhancement restricts to an enhancement for the full monoidal subcategory $\Rep\uqG$.  Hence we obtain sheaf-valued morphisms $\sHom_{\uqG}$ for the small quantum group which are a posteriori compatible with the tensor structure on $\Rep\uqG$.  Specifically, these sheaf-morphisms $\sHom_{\uqG}$ inherit natural composition and monoidal structure maps which localize those of the linear category of small quantum group representations (see Section \ref{sect:shom1}).
\par

Similarly, at the derived level, the action of $D(G/B)$ on $D(G/B_q)$ provides sheaf-valued derived morphisms for the quantum group $\sRHom_{\uqG}$ which again carry a localized monoidal structure (see Section \ref{sect:Enh_derived}).  In this way we obtain monoidal enhancements for both the abelian and derived categories of quantum group representations, over the flag variety.
\par

At the derived level, the Springer resolution manifests naturally via the sheafy derived endomorphisms of the unit.  Namely, via work of Ginzburg and Kumar we obtain an identification
\[
H^\ast(\sRHom_{\uqG}(\1,\1))=p_\ast\O_{\tN},
\]
where $p$ is the affine bundle map $p:\tN\to G/B$.  The natural actions of the sheaf-morphisms $\sRHom_{\uqG}(\1,\1)$ on each $\sRHom_{\uqG}(V,V)$, provided by the localized tensor structure, can then be used to define a support theory
\[
\supp_{\tN}(V):=\Supp_{\tN}H^\ast(\sRHom_{\uqG}(V,V))/\mbb{G}_m
\]
which takes values in the projectivized Springer resolution.  This is our $\tN$-support.  We invite the reader to see Section \ref{sect:AH!} for more details.

\subsection{Relations with geometric representation theory}
\label{sect:intro_grt}

Given the comments of the previous subsections, one might ask about the relationship between our approach to support for the small quantum group and the canonical studies of Arkhipov-Bezrukavnikov-Ginzburg and Bezrukavnikov-Lachowska \cite{arkhipovbezrukavnikovginzburg04,bezrukavnikovlachowska07} in geometric representation theory.  We would conjecture that the results of \cite{arkhipovbezrukavnikovginzburg04,bezrukavnikovlachowska07}, which relate quantum group representations to sheaves on the Springer resolution, can be approached naturally from the perspective of the half-quantum flag variety.
\par

One can look at the situation as follows: the right adjoint to the monoidal embedding $\zeta^\ast:\QCoh_{\dg}(G/B)\to \QCoh_{\dg}(G/B_q)$, which we now consider at an infinity categorical level, provides a pushforward functor
\[
\zeta_\ast:=\sRHom_{G/B_q}(\1,-):\QCoh_{\dg}(G/B_q)\to \QCoh_{\dg}(G/B).
\]
This functor is no longer monoidal, but is a map of stable infinity categories.  Modulo a certain formality claim (Conjecture \ref{conj:formality}), the global action of the endomorphisms of the unit endow the values $\zeta_\ast(V)$ with a natural action of $p_\ast\O_{\tN}$.  By taking account of this action we obtain a functor $\xi_\ast:\QCoh_{\dg}(G/B_q)\to \QCoh_{\dg}(\tN)$.
\par

Via the Kempf embedding we embed the principal block $\mathbf{DG}(\uqG)_0$ of dg representations for the small quantum group into the category of sheaves $\QCoh_{\dg}(G/B_q)$, and so obtain a functor
\begin{equation}\label{eq:325}
\xi_\ast|_{\operatorname{block}_0}:\mathbf{DG}(\uqG)_0\to \QCoh_{\dg}(\tN),
\end{equation}
which again is non-monoidal but is a map between stable infinity categories.  We have conjectured that the map \eqref{eq:325} recovers the equivalence of Bezrukavnikov-Lachowska \cite{bezrukavnikovlachowska07}, and also that of Arkhipov-Bezrukavnikov-Ginzburg \cite{arkhipovbezrukavnikovginzburg04} when we restrict to $G$-equivariant objects.  Further elaborations on this topic are given in Section \ref{sect:GRT} below.

\subsection{Comments on logarithmic TQFT and modular functors}
\label{sect:intro_logtqft}

In the final section of the text, Section \ref{sect:log}, we explain how one might view our $G/B$-enhancement in the light of derived approaches to logarithmic, or non-semisimple, topological quantum field theory.  We consider in particular a derived modular surface invariant
\begin{equation}\label{eq:353}
F_q:\Surf^c(q)\to Vect_{\dg}
\end{equation}
constructed by Schweigert and Woike \cite{schweigertwoike21}.  This functor takes as inputs ($2$-dimensional real) surfaces with certain labels from $\Rep \uqG$, and outputs cochain complexes.  The value $F_q(\Sigma)$ on a given surface is determined by cohomology $F_q(\Sigma)\cong\RHom_{\uqG}(V_{\Sigma},W_{\Sigma})$, and the value on $S^2$ in particular is the derived endomorphism algebra of the unit $\RHom_{\uqG}(k,k)\cong \O(\mcl{N})$.
\par

A main point of consideration here is the ``singular" nature of the functor $F_q$.  In particular, the values $F_q(\Sigma)$ are non-dualizable not only as dg vector spaces, but also as dg modules over the state space for the sphere $F_q(S^2)=\O(\mcl{N})$.  This failure of dualizability follows from the fact that the nilpotent cone itself is singular, and these singularities obstruct ones ability to extend the functor $F_q$ up from $2$-dimensions to produce meaningful values $F^{\rm Extended}_q(M)$ on $3$-manifolds.  (See Observation \ref{obs:sing} below.)
\par

In Section \ref{sect:log} we explain how one might use our $\QCoh(G/B)$-enhancement for the quantum group (discussed above) to replace the target category $Vect_{\dg}$ in \eqref{eq:353} with $\QCoh_{\dg}(G/B)$, and also to replace $\mcl{N}$ with its desingularization $\tN$.  We argue that, in principal, such a global desingularization provides a means of overcoming \emph{some} obstructions to dualizability present in the functor $F_q$.  The interested reader can see Section \ref{sect:log} for more details.

\subsection{Organization}

The paper has three main parts, which are preceded by the introductory sections \ref{sect:notation}--\ref{sect:borel}.  These introductory sections cover the necessary background materials for quantum groups.  In Section \ref{sect:borels} in particular we introduce the $G/B$-family of small quantum Borels $\msc{B}_\lambda$ referenced above.
\par

In Part I of the text produce monoidal localizations for $\Rep \uqG$ and $D(\uqG)$ over the flag variety.  One can see Theorem \ref{thm:enhA} and Proposition \ref{prop:D_Enh} respectively.  These localizations appear as consequences of a monoidal embedding $\Rep \uqG\to \QCoh(G/B_q)$, which is discussed in Section \ref{sect:Kempf}.  At Theorem \ref{thm:fiber_l} we calculate the fibers of the localized derived category for $\uqG$ over $G/B$, and relate these fibers to the small quantum Borels $\msc{B}_\lambda$.
\par

In Part II of the paper we explain how the localizations of Part I imply the existence of an $\tN$-valued support theory for the small quantum group, and a subsequent resolution of the Balmer spectrum for stable $\uqG$-representations.  In Part III we discuss relationships between the materials of Part I, fundamental results from geometric representations theory, and emergent studies of logarithmic TQFTs.

\subsection{Acknowledgements}

Thanks to David Ben-Zvi, Roman Bezrukavnikov, Ken Brown, Eric Friedlander, Sergei Gukov, David Jordan, Simon Lentner, Dan Nakano, Sunghyuk Park, Noah Snyder, and Lukas Woike for useful discussions which influenced the trajectory of this work.  C.\ Negron is supported by NSF grant DMS-2001608.  J.\ Pevtsova is supported by NSF grant DMS-1901854 and the Brian
and Tiffinie Pang faculty fellowship. This material is based upon work supported
by the National Science Foundation under Grant No. DMS-1440140, while the first
author was in residence at the Mathematical Sciences Research Institute in Berkeley,
California, and the second author was in digital
residence.


\section{Notation guide and categorical generalities}
\label{sect:notation}

Throughout $k$ is an algebraically closed field of characteristic $0$, and $G$ is an almost-simple algebraic group over $k$.  We take $h$ to be the Coxeter number for $G$, and $B\subset G$ is a fixed choice of Borel, which we recognize as the positive Borel in $G$.
\begin{itemize}
\item $q\in k$ is a root of unity of odd order $\ord(q)>h$.  When the Dynkin diagram for $G$ has a component of type $G_2$ we also assume $3\nmid \ord(q)$.\vspace{2mm}

\item $\pi:G\to G/B$ is the quotient map.  (Often we employ a dual group $\dG$ with Borel $\dB$ and have the corresponding quotient $\pi:\dG\to \dG/\dB$.)\vspace{2mm}

\item $\kappa:\tN\to \mcl{N}$ is the moment map for the Springer resolution and $p:\tN\to G/B$ is the projection to the flag variety.\vspace{2mm}

\item A geometric point $x:\Spec(K)\to Y$ in a ($k$-)scheme $Y$ is a map from the spectrum of an algebraically closed field extension $K$ of the base $k$.\vspace{2mm}

\item The symbol $\ot_k$ denotes a vector space tensor product, and the generic symbol $\ot$ indicates the product operation in a given monoidal category $\msc{C}$.  So this product can be a product of sheaves, or a product of group representations, etc.  We also let $\ot_k$ denote the induced action of $Vect_k$ on a linear category.\vspace{2mm}

\item $\1$ is the unit object in a given tensor category $\msc{C}$.
\end{itemize}

\subsection{Finite-dimensional vs.\ infinite-dimensional representations}
\label{sect:finite}

For our study it has been convenient to work with cocomplete categories, where one generally has enough injectives and can freely use representability theorems.  For this reason we employ categories of arbitrary (possibly infinite-dimensional) representations $\Rep A$ for a given Hopf algebra $A$.  This is, by definition, the category of locally finite representations, or equivalently representations which are the unions of their finite-dimensional subrepresentations.
\par

One recovers the category $\rep A$ of finite-dimensional representations as the subcategory of objects in $\Rep A$ which admit left and right duals.  Since tensor functors preserve dualizable objects \cite[Ex. 2.10.6]{egno15} restricting to the subcategory of dualizable objects $\rep A\subset \Rep A$ is a natural operation with respect to tensor functors.  In this way one moves freely between the ``small" and ``big" representation categories for $A$.
\par

When we work with derived categories, we take
\[
D(A):=\left\{\begin{array}{c}\text{the unbounded derived category of (generally}\\
\text{infinite-dimensional) $A$-representations}
\end{array}\right\}.
\]
We have the distinguished subcategories $D^b(A)$, $D^+(A)$, etc.\ of appropriately bounded complexes of (generally infinite-dimensional) representations, and a distinguished subcategory of bounded complexes of finite-dimensional representations, which one might write as $D^b(\rep A)$ or $D_{\rm fin}(A)$.  In following the philosophy proposed above, one can view $D_{\rm fin}(A)$ as the subcategory of dualizable objects in $D(A)$.

\subsection{Relative Hopf modules}

For a Hopf algebra $A$, and an $A$-comodule algebra $\O$, we let $_\O\mbf{M}^A$ denote the category of relative $(\O,A)$-Hopf modules, with no finiteness assumptions.  This is the category of simultaneous (left) $\O$-modules and (right) $A$-comodules $M$, for which the coaction $M\to M\ot_k A$ is a map of $\O$-modules.  Here $\O$ acts diagonally on the product $M\ot_k A$, $a\cdot(v\ot b)=a_1v\ot a_2b$.
\par

For a basic example, we consider an affine algebraic group $H$ acting on (the right of) an affine scheme $Y$.  This gives the algebra of functions $\O=\O(Y)$ a comodule structure over $A=\O(H)$.  The fact that the action map $Y\times H\to Y$ is a scheme map says that $\O$ is an $A$-comodule algebra under this comodule structure.  Relative Hopf modules are then identified with equivariant sheaves on $Y$ via the global sections functor,
\[
\Gamma(Y,-):\QCoh(Y)^H\overset{\sim}\to {_{\O(Y)}\mbf{M}^{\O(H)}}.
\]

\subsection{Descent along $H$-torsors}
\label{sect:descent}

Suppose an algebraic group $H$ acts on a scheme $Y$, that the quotient $Y/H$ exists, and the quotient map $\pi:Y\to Y/H$ is a (faithfully flat) $H$-torsor.  For example, we can consider an algebraic group $G$ and a closed subgroup $H\subset G$ acting on $G$ by translation.  In this case the quotient $G\to G/H$ exists, is faithfully flat, and realizes $G$ as an $H$-torsor over the quotient \cite[Theorem B.37]{milne17}.
\par

For $Y$ as prescribed, pulling back along the quotient map $\pi:Y\to Y/H$ defines an equivalence of categories
\begin{equation}\label{eq:420}
\pi^\ast:\QCoh(Y/H)\overset{\sim}\to \QCoh(Y)^H
\end{equation}
from sheaves on the quotient to $H$-equivariant sheaves on $Y$.  The inverse to this equivalence if provided by faithfully flat descent, or simply ``descent"
\[
\operatorname{desc}:\QCoh(Y)^H\to \QCoh(Y/H).
\]
For details on faithfully flat descent one can see \cite[Expos\'{e} VIII]{sga1} or \cite[\href{https://stacks.math.columbia.edu/tag/03O6}{Tag 03O6}]{stacks2}.  One may understand descent simply as the inverse to the equivalence \eqref{eq:420}, which we are claiming exists under the precise conditions outlined above.
\par

For \emph{any} scheme $Y$ with an $H$-action, we have a pair of adjoint functors 
\[
E_{-}:  \Rep H \to \QCoh(Y)^H\ \ \text{and}\ \ -|_{H}: \QCoh(Y)^H \to \Rep H,
\]
where for $V \in \Rep H$ we set $E_V$ to be the vector bundle on $Y$ corresponding to the $H$-equivariant $\O(Y)$-module $\O(Y) \otimes_k V$, with $\O(Y)$ acting on the left factor and $H$ acting diagonally. By construction, $E_V$ is locally free; it is coherent if and only if $V$ is finite-dimensional. The right adjoint $-|_{H}$ is given by taking  global section $\Gamma(Y,-)$ and then forgetting the action of $\O(Y)$.
\par

For an affine algebraic group $Y = G$ and a closed subgroup $H \subset G$, the descent of the equivariant vector bundle $E_V$, defined as above, is the familiar bundle from \cite[I.5.8]{jantzen03}, \cite[3.3]{suslinfriedlanderbendel97b}. 

\subsection{Compact and perfect objects}
\label{sect:perf_comp}

A compact object $V$ in a cocomplete triangulated (resp.\ abelian) category $T$ is an object for which $\Hom_T(V,-)$ commutes with set indexed sums (resp.\ filtered colimits).  Given a monoidal triangulated category $T=(T,\ot)$, we reserve the term \emph{perfect} to refer to objects in $T$ which are dualizable for the product $\ot$--although we won't use this term with any frequency in this text.  (By a dualizable object we mean an object which admits left and right duals, in the sense of \cite[Definition 2.1.1]{bakalovkirillov01}.)  This language is adapted from geometry, where we might view the category of $H$-representations for a group $H$, for example, as sheaves on the classifying stack $BH$, see for example \cite{benzvifrancisnadler10,hallrydh17,lurieSAG}.

\subsection{Enriched categories and enhancements}
\label{sect:enrich}

By a category $\underline{T}$ enriched in a given monoidal category $\mcl{Q}$ we mean a collection of objects $\operatorname{obj}\underline{T}$ for $\underline{T}$ and morphisms $\mcl{H}om(M,N)$, which are objects in $\mcl{Q}$, for each pair of objects $M$ and $N$ in $\underline{T}$.  We suppose, additionally, the existence of associative composition morphisms
\[
\circ:\mcl{H}om(M,N)\ot\mcl{H}om(L,M)\to\mcl{H}om(L,N)
\]
for each triple of objects in $\underline{T}$.  Basics on enriched categories can be found in \cite[Ch 3]{riehl14}, though we recall some essential points here.
\par

Given any lax monoidal functor $\xi:\mcl{Q}\to \mcl{Q}'$ we can push forward the morphisms in $\underline{T}$ along $\xi$ to get a new category $\xi(\underline{T})$ which is enriched in $\mcl{Q}'$ \cite[Lemma 3.4.3]{riehl14}.  As one expects, the objects in $\xi(\underline{T})$ are the same as those in $\underline{T}$, and the morphisms in $\xi(\underline{T})$ are given as $\xi\big(\mcl{H}om(M,N)\big)$.  The composition maps for $\xi(\underline{T})$ are induced by those of $\underline{T}$ and the (lax) monoidal structure on the functor $\xi$.
\par

Of particular interest is the ``global sections" functor $\Gamma(\mcl{Q},-)=\Hom_{\mcl{Q}}(\1,-):\mcl{Q}\to \operatorname{Set}$, with lax monoidal structure on $\Gamma(\mcl{Q},-)$ provided by the monoidal and unit structure maps for $\mcl{Q}$,
\[
\Gamma(\mcl{Q},A)\times \Gamma(\mcl{Q},B)\to \Gamma(\mcl{Q},A\ot B),\ \ (f,g)\mapsto f\ot g.
\]
For any enriched category $\underline{T}$ over $\mcl{Q}$ the global sections $\Gamma(\mcl{Q},\underline{T})$ are then an ordinary category \cite[Definition 3.4.1]{riehl14}.  Note that the category of global sections $\Gamma(\mcl{Q},\underline{T})$ also acts on $\underline{T}$, in the sense that any morphism $f\in \Gamma(\mcl{Q},\mcl{H}om(M,N))$ specifies composition and precomposition maps
\[
f_\ast:\mcl{H}om(L,M)\to \mcl{H}om(L,N)\ \ \text{and}\ \ f^\ast:\mcl{H}om(N,L)\to \mcl{H}om(M,L)
\]
via the unit structure on $\mcl{Q}$ and composition in $\underline{T}$.
\par

Suppose now that $\mcl{Q}$ is symmetric.  By a \emph{monoidal} category enriched in $\mcl{Q}$ we mean an enriched category $\underline{T}$ with a product structure on objects, unit object, and associator, where the unit and associator structure maps appear as global isomorphisms for $\underline{T}$.  We require the existence of tensor maps
\[
tens^{\underline{T}}:\mcl{H}om(M,N)\ot \mcl{H}om(M',N')\to \mcl{H}om(M\ot M',N\ot N')
\]
which are associative relative to the associators on $\mcl{Q}$ and $\underline{T}$, and appropriately compatible with composition.  This compatibility between the tensor and composition morphisms appears as an equality
\[
\begin{array}{l}
(g_1\ot g_2)\circ (f_1\ot f_2)=(g_1\circ f_1)\ot (g_2\circ f_2):\vspace{1mm}\\
\hspace{2cm}\mcl{G}_1\ot \mcl{F}_1\ot \mcl{G}_2\ot\mcl{F}_2\to \mcl{H}om(L_1\ot L_2,N_1\ot N_2)
\end{array}
\]
for maps $f_i:\mcl{F}_i\to \mcl{H}om(L_i,M_i)$ and $g_i:\mcl{G}_i\to \mcl{H}om(M_i,N_i)$.  Such a monoidal structure on $\underline{T}$ implies a monoidal structure on its category of global sections.
\par

By an \emph{enhancement} of a (monoidal) category $T$ in $\mcl{Q}$, we mean a choice of enriched (monoidal) category $\underline{T}$ and a choice of (monoidal) equivalence $T\cong \Gamma(\mcl{Q},\underline{T})$.

\section{Quantum groups}
\label{sect:quantumgroups}

We recall basic constructions and results for quantum groups.  We also recall a geometric (re)construction of the small quantum group via de-equivariantization along the quantum Frobenius functor.

\subsection{Lusztig's divided power algebra}

We briefly review Lusztig's divided power algebra, leaving the details to the original works \cite{lusztig90II,lusztig93}.  Let $\mfk{g}$ be a semisimple Lie algebra over the complex numbers, and fix some choice of simple roots for $\mfk{g}$.
\par

To begin, one considers the generic quantum universal enveloping algebra
\[
U^{\operatorname{gen}}_v(\mfk{g})=\frac{\mbb{Q}(v)\langle E_\alpha, F_\alpha, K^{\pm 1}_\alpha:\alpha\text{ is a simple root for }\mfk{g}\rangle}{(v\text{-analogs of Serre relations})},
\]
where the $v$-Serre relations are as described in \cite[\S 1.1]{lusztig90II}.  These relations can be written compactly as
\[
\begin{array}{c}
\operatorname{ad}_v(E_\alpha)^{1-\langle \alpha,\beta\rangle}(E_\beta)=0,\ \ \operatorname{ad}_v(F_\alpha)^{1-\langle \alpha,\beta\rangle}(F_\beta)=0,\ \ K_\beta E_\alpha K_\beta^{-1}=v^{(\alpha,\beta)}E_\alpha,\vspace{2mm}\\
K_\beta F_\alpha K_\beta^{-1}=v^{-(\alpha,\beta)}F_\alpha,\ \ [E_\alpha,F_\beta]=\delta_{\alpha\beta}(K_\alpha+K_\alpha^{-1})/(v-v^{-1}).
\end{array}
\]
The algebra $U^{\operatorname{gen}}_v(\mfk{g})$ admits a Hopf algebra structure over $\mbb{Q}(v)$ with coproduct
\begin{equation}\label{eq:424}
\Delta(E_\alpha)=E_\alpha\ot K_\alpha +1\ot E_\alpha,\ \ \Delta(F_\alpha)=F_\alpha\ot 1+K_\alpha^{-1}\ot F_\alpha,\ \ \Delta(K_\alpha)=K_\alpha\ot K_\alpha.
\end{equation}
We consider the distinguished subalgebra of ``Cartan elements"
\[
U^{\operatorname{gen}}_v(\mfk{g})^0=\mbb{Q}(v)\langle K_\alpha,K_\alpha^{-1}:\alpha\text{ simple}\rangle\subset U^{\operatorname{gen}}_v(\mfk{g}),
\]
which we refer to as the toral subalgebra in $U^{\operatorname{gen}}_v(\mfk{g})$.
\par

In $U^{\operatorname{gen}}_v(\mfk{g})$ one has the $\mbb{Z}[v,v^{-1}]$-subalgebra $U_v(\mfk{g})$ generated by the divided powers
\[
E_\alpha^{(i)}=E^i_\alpha/[i]_{d_\alpha}!,\ \ F_\alpha^{(i)}=F^i_\alpha/[i]_{d_\alpha}!,\ \ \text{and }K_\alpha^{\pm 1}.
\]
Here $d_\alpha=|\alpha|^2/|\text{short\ root}|^2$ and $[v]_{d_\alpha}!$ is the $v^{d_i}$-factorial (see \cite[\S 1 and Theorem 6.7]{lusztig90II}).  The subalgebra $U_v(\mfk{g})$ furthermore has a Hopf structure induced by that of $U^{\operatorname{gen}}_v(\mfk{g})$, which is given by the same formulas \eqref{eq:424}.  We define the toral subalgebra as the intersection $U_v(\mfk{g})\cap U^{\operatorname{gen}}_v(\mfk{g})^0$.  This subalgebra not only contains powers of the $K_\alpha$, but also certain divided powers in the $K_\alpha$.
\par

For any value $q\in \mbb{C}^\times$ we have the associated $\mbb{Z}$-algebra map $\phi_q:\mbb{Z}[v,v^{-1}]\to \mbb{C}$ which sends $v$ to $q$, and we change base along $\phi_q$ to obtain Lusztig's divided power algebra
\[
U_q(\mfk{g})=\mbb{C}\ot_{\mbb{Z}[v,v^{-1}]}U_v(\mfk{g})
\]
at parameter $q$.  An important aspect of this algebra is that, at $q$ of order $l$, we have
\[
E_\alpha^l=[l]_{d_\alpha}!E_\alpha^{(l)}=0\ \ \text{and}\ \ F_\alpha^l=[l]_{d_\alpha}!F_\alpha^{(l)}=0.
\]
Hence the elements $E_\alpha$ and $F_\alpha$ become nilpotent in $U_q(\mfk{g})$.  One should compare with the distribution algebra associated to an algebraic group in finite characteristic.
\par

In Section \ref{sect:borel} we also consider the divided power algebra $U_q(\mfk{b})$ for the Borel.  This is the subalgebra $U_q(\mfk{b})\subset U_q(\mfk{g})$ generated by the toral elements as well as the divided powers $E_\alpha^{(i)}$ for the positive simple roots.

\subsection{Big quantum groups}
\label{sect:Gq}

We follow \cite{lusztig93} \cite[\S 3]{andersenparadowski95}.  Fix $G$ almost-simple with associated character lattice $X$.  Let $q$ be of finite odd order $l$.  We consider the category of representations for the ``big" quantum group associated to $G$,
\[
\rep G_q:=\left\{\begin{array}{c}
\text{finite-dimensional representations for $U_q(\mfk{g})$}\\
\text{which are compatibly graded by the character lattice }X
\end{array}\right\}.
\]
Here when we say a representation $V$ is graded by the character lattice we mean that $V$ decomposes into eigenspaces $V=\oplus_{\lambda\in X} V_\lambda$ for the action of the toral subalgebra in $U_q(\mfk{g})$, with each $V_\lambda$ of the expected eigenvalue.  For example one requires $K_\alpha\cdot v=q^{d_\alpha\langle \alpha,\lambda\rangle}v$ for homogeneous $v\in V_\lambda$ (see \cite[\S 3.6]{andersenparadowski95}).  The category $\rep G_q$ is the same as the category of representations of Lusztig's modified algebra $\dot{U}_q(G)$ \cite{lusztig93} \cite[\S 1.2]{kashiwara94}.
\par

We let $\Rep G_q$ denote the category of generally infinite-dimensional, integrable, $G_q$-representations.  Equivalently, $\Rep G_q$ is the category of $U_q(\mfk{g})$-modules which are the unions of finite-dimensional $X$-graded submodules.  Note that all objects in $\Rep G_q$ remain $X$-graded.

Let $P$ and $Q$ denote the weight and root lattices for $G$ respectively, $\Phi$ denote the roots in $Q$, and let $(-,-):P\times P\to \mbb{Q}$ denote the normalized Killing form which takes value $2=(\alpha,\alpha)$ on short roots $\alpha$.  For $r=\exp(P/Q)$ we choose a $r$-th root $\sqrt[r]{q}$ of $q$ and define $q^{(\lambda,\mu)}=(\sqrt[r]{q})^{r(\lambda,\mu)}$.  We then have the associated $R$-matrix for $\Rep G_q$ which appears as
\begin{equation}\label{eq:R}
\begin{array}{rl}
R&=(\sum_{n:\Phi^+\to \mbb{Z}_{\geq 0}} \operatorname{coeff}(q,n)E_{\gamma_1}^{n_1}\dots E_{\gamma_r}^{n_r}\ot F_{\gamma_1}^{n_1}\dots F_{\gamma_r}^{n_r})\Omega^{-1}\vspace{1mm}\\
&=\Omega^{-1}+\text{higher order terms}.
\end{array}
\end{equation}
Here $\Omega$ is the global semisimple operator
\[
\Omega=\sum_{\lambda,\mu\in X}q^{(\lambda,\mu)}1_\lambda\ot 1_\mu
\]
with each $1_\lambda$ the natural projection $1_\lambda:V\to V_\lambda$ \cite[Ch.\ 32]{lusztig93} \cite[\S 1]{sawin06}.  We note that the sum in \eqref{eq:R} has only finitely many non-vanishing terms, since the $E_\gamma$ and $F_\gamma$ are nilpotent in $U_q(\mfk{g})$.  The operator $R$ endows $\Rep G_q$ with its standard braiding
\[
\begin{array}{l}
c_{W,V}:W\ot V\to V\ot W,\vspace{1mm}\\
c_{V,W}(w,v)=q^{(\deg(w),\deg(v))}(v\ot w+\sum_{n>0}\operatorname{coeff}(q,n) F_{\gamma_1}^{n_1}\dots F_{\gamma_r}^{n_r}v\ot E_{\gamma_1}^{n_1}\dots E_{\gamma_r}^{n_r}w).
\end{array}
\]

We consider the dual group $\dG$, which (at $q$ of the given order) is of the same Dynkin type as $G$, but has an alternate choice of character lattice.  Specifically, $\dG$ has roots $l\cdot \Phi$ with form $(-,-)^{\vee}=\frac{1}{l^2}(-,-)$ and character lattice $X^M$, where
\[
X^M:=\{\mu\in X:(\mu,\lambda)\in l\mbb{Z}\ \text{for all }\lambda\in X\}.
\]
So, for example, when $G$ is simply-connected the dual $\dG$ is the adjoint form for $G$.
\par

Via Lusztig's quantum Frobenius \cite[Ch.\ 35]{lusztig93} we have a Hopf map $fr^\ast:\dot{U}_q(G)\to \dot{U}(\dG)$,
\[
fr^\ast=\left\{
\begin{array}{l}
E_\gamma,\ F_\gamma\mapsto 0\\
E_\gamma^{(l)}\mapsto e_\gamma\\
F_\gamma^{(l)}\mapsto f_\gamma\\
1_\lambda\mapsto 1_\lambda\text{ when $\lambda\in X^M$ and $0$ otherwise}
\end{array}\right.
\]
which defines a braided tensor embedding
\[
\Fr:\Rep \dG\to \Rep G_q
\]
whose image is the M\"uger center in $\Rep G_q$ \cite[Theorem 5.3]{negron21}, i.e.\ the full tensor subcategory of all $V$ in $\Rep G_q$ for which $c_{-,V}c_{V,-}:V\ot-\to V\ot-$ is the identity transformation.

\begin{remark}\label{rem:fr}
Our Frobenius map $fr^\ast:\dot{U}_q(G)\to \dot{U}_q(\dG)$ is induced by that of \cite{lusztig93}, but is not precisely the map of \cite{lusztig93}.  Similarly, our dual group $\dG$ is not precisely the dual group $G^\ast$ from \cite{lusztig93}.  Specifically, Lusztig's dual group $G^\ast$ is defined by taking the dual lattice $X^\ast\subset X$ to consist of all $\mu$ with restricted pairings $(\alpha,\mu)\in l\mbb{Z}$ at all simple roots $\alpha$.  This lattice $X^\ast$ contains $X^M$, so that we have an inclusion $\Rep \dG\to \Rep G^\ast$.  We then obtain our quantum Frobenius, functor say, by restricting the more expansive quantum Frobenius $\Rep G^\ast\to \Rep G_q$ from \cite{lusztig93} along the inclusion from $\Rep \dG$.  Very directly, our dual group is dictated by the $R$-matrix while Lusztig's dual group is dictated by representation theoretic considerations.
\end{remark}

\subsection{Quantum function algebras}

For us the quantum function algebra $\O(G_q)$ is a formal device which allows us to articulate certain categorical observations in a ring theoretic language.  The Hopf algebra $\O(G_q)$ is the unique Hopf algebra so that we have an equality
\[
\operatorname{Corep}\O(G_q)=\Rep G_q
\]
of non-full monoidal subcategories in $Vect$.  Via Tannakian reconstruction \cite{delignemilne82,schauenburg92}, one obtains $\O(G_q)$ (uniquely) as the coendomorphism algebra of the forgetful functor
\[
forget:\Rep G_q\to Vect,
\]
and the Hopf structure on $\O(G_q)$ is derived from the monoidal structure on $forget$.

\subsection{The small quantum group}

For $G$ and $q$ as above the small quantum group $\uqG$ is essentially the small quantum group from Lusztig's original work \cite{lusztig90,lusztig90II}, but with some slight variation in the grouplikes.  First, let $A$ denote the character group on the quotient $X/X^M$, $A=(X/X^M)^\ast$.  For any root $\gamma$ let $K_\gamma\in A$ denote the character $K_\gamma:X/X^M\to \mbb{C}^\ast$, $K_\gamma(\bar{\lambda})=q^{(\gamma,\lambda)}$.
\par

We now define
\[
u(G_q):=\frac{k\langle E_\alpha,\ F_\alpha,\ \xi:\alpha\text{ simple roots, }\xi\in A\rangle}
{\left(\begin{array}{c}\text{$q$-Serre relations \cite[(a3)--(a5)]{lusztig90II}, relations from $A$,}\vspace{2mm}\\
\xi\cdot E_\alpha\cdot \xi^{-1}=\xi(\alpha)E_\alpha,\ \xi\cdot F_\alpha\cdot \xi^{-1}=\xi(-\alpha)F_\alpha
\end{array}\right)}.
\]
So, a representation of $\uqG$ is just a representations of Lusztig's usual small quantum group which admits an additional grading by $X/X^M$ for which the $K_\alpha$ act as the appropriate semisimple endomorphisms $K_\alpha \cdot v=q^{(\alpha,\deg(v))}v$.  The algebra $\uqG$ admits the expected Hopf structure, just as in \cite{lusztig90II}, with the $\xi\in A$ grouplike and the $E_\alpha$ and $F_\alpha$ skew primitive.
\par

Now, the simple representations for $\uqG$ are labeled by highest weights $L(\bar{\lambda})$, for $\bar{\lambda}\in X/X^M$, and $L(\bar{\lambda})$ is dimension $1$--and hence invertible with respect to the tensor product--precisely for those $\bar{\lambda}$ with $q^{(\alpha,\lambda)}=1$ at all simple $\alpha$.  So if we let $X^\ast\subset X$ denote the sublattice of weights $\lambda$ with $(\lambda,\alpha)\in l\mbb{Z}$ for all simple $\alpha$, then we have $X^M\subset X^\ast$ and the subgroup $X^\ast/X^M\subset X/X^M$ labels these $1$-dimensional simples.  These simples form a fusion subcategory in $\Rep \uqG$ so that we have a tensor embedding
\[
Vect(X^\ast/X^M)\to \Rep \uqG,
\]
where $Vect(X^\ast/X^M)$ denotes the category of $X^\ast/X^M$ graded vector spaces, and we have the corresponding Hopf quotient $\uqG\to \O(X^\ast/X^M)$.

\begin{example}
When $G$ is of adjoint type $X^\ast=X^M$, so that $Vect(X^\ast/X^M)$ is trivial.  When $G$ is simply-connected $X^\ast=lP$, $X^M=lQ$, and $Vect(X^\ast/X^M)$ is isomorphic to representations of the center $Z(G)$.
\end{example}

We note that the $R$-matrix for $G_q$ provides a well-defined global operator on products of $\uqG$-representations, so that we have the braiding on $\Rep \uqG$ given by the same formula
\[
c_{V,W}:V\ot W\to W\ot V,\ \ c_{V,W}(v,w)=R_{21}(w\ot v).
\]
One can see that this braiding on $\Rep \uqG$ is non-degenerate, in the sense that the M\"uger center vanishes, since the induced form $\bar{\Omega}$ on $X/X^M$ is non-degenerate \cite[Theorem 5.3]{negron21}.  We have the restriction functor
\[
\res:\Rep G_q\to \Rep \uqG
\]
which is braided monoidal.  The following result is essentially covered in works of Andersen and coauthors \cite{andersen03,andersenwen92,andersenparadowski95,andersenpolowen91}.

\begin{proposition}\label{prop:andersen}
\begin{enumerate}
\item For any simple object $L$ in $\Rep \uqG$, there is a simple $G_q$-representation $L'$ so that $L$ is a summand of $\res(L')$.
\item Any projective object in $\Rep G_q$ restricts to a projective in $\Rep \uqG$.
\item For any projective object $P$ in $\Rep \uqG$, there is a projective $G_q$-representation $P'$ so that $P$ is a summand of $\res(P')$.
\end{enumerate}
\end{proposition}

\begin{proof}
For $\lambda\in X$ let $\msc{L}(\lambda)$ denote the corresponding simple representation in $\Rep G_q$.  Let $A'=(X/X^\ast)^\ast$ and consider $u'\subset \uqG$ the Hopf subalgebra generated by the $E_\alpha$, $F_\alpha$, and elements in $A'$.  We have the exact sequence of Hopf algebras $1\to u'\to \uqG\to \O(X^\ast/X^M)\to 1$.  From the corresponding spectral sequence on cohomology we find that an object in $\Rep \uqG$ is projective if and only if its restriction to $u'$ is projective, and obviously any object with simple restriction to $u'$ is simple over $\uqG$.
\par

The representation category of $u'$ is the category $\msc{C}_{G_1}$ from \cite{andersenparadowski95}.  So by \cite[Theorem 3.12]{andersenparadowski95} we see that all simple representations for $u'$ are restricted from representations over $G_q$, and by the formula \cite[Theorem 1.10]{andersenwen92} one can see that all simples for $\Rep\uqG$ are summands of a simple from $G_q$.  (The representation ``$\bar{L}_k(\lambda_0)$" from \cite{andersenwen92} will split into $1$-dimensional simples over $\uqG$ and there will be no twists ``$\bar{L}_k(\lambda_i)^{(i)}$" since we assume $\operatorname{char}(k)=0$.)  So we obtain (i).  Statements (i) and (ii) follow from the fact that the Steinberg representation is simple and projective over $G_q$ and restricts to a simple and projective representation over $\uqG$ \cite[Proposition 2.2]{andersenwen92} \cite[Theorem 9.8]{andersenpolowen91} \cite[Corollary 9.7]{andersen03}.
\end{proof}

\subsection{A remark on grouplikes}

In the literature there are basically two choices of grouplikes for the small quantum group which are of interest.  In the first case, we take the small quantum group with grouplikes given by characters on the quotient $X/X^\ast$, where $X^\ast=lP\cap X$.  This is a choice which is relevant for many representations theoretic purposes, and which reproduces Lusztig's original small quantum group \cite{lusztig90,lusztig90II} at simply-connected $G$.  In the second case, one proceeds as we have here and considers grouplikes given by characters on the quotient $X/X^M$.  This is a choice relevant for physical applications, as one preserves the $R$-matrix and hence allows for the small quantum group to be employed in constructions and analyses of both topological and conformal field theories, see for example \cite{derenzigeerpatureau18,schweigertwoike21,brochierjordansafronovsnyder,creutziggannon17,feigintipunin,gannonnegron}.
\par

This movement of the grouplikes for the small quantum group corresponds precisely to the choice of dual group to $G_q$ for the quantum Frobenius (discussed above).  One has the maximal choice $G^\ast$, or the choice $\dG$ dictated by the $R$-matrix.

\subsection{De-equivariantization and the small quantum group}
\label{sect:de_equiv}

We have the quantum Frobenius $\Fr:\Rep \dG\to \Rep G_q$ as above, and define the de-equivariantization in the standard way
\[
(\Rep G_q)_{\dG}:=\left\{\begin{array}{c}\text{the category of arbitrary}\\
\text{$\O(\dG)$-modules in }\Rep G_q\end{array}\right\}={_{\O(\dG)}\mbf{M}^{\O(G_q)}}
\]
\cite{dgno10}.  Here we abuse notation to write the image of the object $\O(\dG)$ in $\Rep \dG$ under quantum Frobenius simply as $\O(\dG)\in \Rep G_q$.  The category $(\Rep G_q)_{\dG}$ is monoidal under the product $\ot=\ot_{\O(\dG)}$.  We have the de-equivariantization map $\de:\Rep G_q\to (\rep G_q)_{\dG}$, which is a free module functor $\de(V):=\O(\dG)\ot_k V$, and the category $(\Rep G_q)_{\dG}$ admits a unique braided monoidal structure so that this de-equivariantization map is a functor between braided monoidal categories.
\par

Via the monoidal equivalence $(-)^{\sim}:\O(\dG)\text{-Mod}\overset{\sim}\to \QCoh(\dG)$ the de-equivariantization is identified with a certain non-full monoidal subcategory in $\QCoh(\dG)$, which one might refer to as the category of $G_q$-equivariant sheaves over $\dG$.  We let $\QCoh(\dG)^{G_q}$ denote this category of $G_q$-equivariant sheaves on $\dG$, so that we have a monoidal equivalence
\[
\Gamma(\dG,-):\QCoh(\dG)^{G_q}\overset{\sim}\to (\Rep G_q)_{\dG}.
\]
This equivalence furthermore provides a braiding on the monoidal category $\QCoh(\dG)^{G_q}$ which is induced directly by the $R$-matrix,
\[
\operatorname{braid}_{M,N}:M\ot N\to N\ot M,\ \ m\ot n\mapsto R_{21}(n\ot m).
\]

\begin{definition}
The quantum Frobenius kernel for $G$ at $q$ is the braided monoidal category of $G_q$-equivariant quasi-coherent sheaves over $\dG$,
\[
\FK{G}_q:=\QCoh(\dG)^{G_q}.
\]
\end{definition}

The compact objects in $\FK{G}_q$ are precisely those equivariant sheaves which are coherent over $\dG$ \cite[Lemma 8.4]{negron21}.  As a consequence of Theorem \ref{thm:ag} below, all coherent sheaves are furthermore dualizable.
\par

\begin{remark}
Our use of sheaves over $\dG$ rather that $\O(\dG)$-modules in the definition of the quantum Frobenius kernel is stylistic.  The reader will not be harmed in thinking of the category $\FK{G}_q$ simply as the category of $\O(\dG)$-modules in $\Rep G_q$, or equivalently as the category of $(\O(\dG),\O(G_q))$-relative Hopf modules.
\end{remark}

We have the following observation of Arkhipov and Gaitsgory \cite{arkhipovgaitsgory03}, which is premeditated by works of Takeuchi and Schneider \cite{takeuchi79,schneider93}.  

\begin{theorem}[\cite{arkhipovgaitsgory03}]\label{thm:ag}
Taking the fiber at the identity in $\dG$ provides an equivalence of (abelian) braided monoidal categories
\[
1^\ast:\FK{G}_q\overset{\sim}\to \Rep \uqG.
\]
\end{theorem}

\begin{remark}
Theorem \ref{thm:ag} is alternately deduced from \cite[Theorem 2]{takeuchi79} and \cite[Remark 2.5]{schneider93}.
\end{remark}

The above theorem tells us that the dualizable objects in $\FK{G}_q$ are precisely the compact objects, i.e.\ coherent equivariant sheaves, as claimed above.  This subcategory of coherent sheaves is a finite tensor category which is equivalent to the category of finite-dimensional $\uqG$-representations, via the above equivalence.
\par

Since monoidal functors preserve duals \cite[Exercise 2.10.6]{egno15} we see that the image of any dualizable sheaf under the forgetful functor $\FK{G}_q\to \QCoh(\dG)$ is dualizable.  Since the dualizable objects in $\QCoh(\dG)$ are precisely finite rank vector bundles we see that all compact/dualizable objects in $\FK{G}_q$ are finite rank vector bundles over $\dG$, and all objects in $\FK{G}_q$ are therefore flat over $\dG$.
\par

From the above geometric perspective the de-equivariantization map for $(\Rep G_q)_{\dG}$ becomes an equivariant vector bundle map $E_-:\Rep G_q\to \FK{G}_q$, $E_V=\O_{\dG}\ot_k V$, which we still refer to as the de-equivariantization functor.  One sees immediately that the equivalence of Theorem \ref{thm:ag} fits into a diagram
\[
\xymatrix{
 & \Rep G_q\ar[dr]^{\res}\ar[dl]_{E_-}\\
\FK{G}_q\ar[rr]^\sim_{1^\ast} & & \Rep \uqG.
}
\]
From this point on we essentially forget about the Hopf algebra $\uqG$, and work strictly with its geometric incarnation $\FK{G}_q$.

\subsection{The $\dG$-action on the quantum Frobenius kernel}

As explained in \cite{arkhipovgaitsgory03,dgno10} we have a translation action of $\dG$ on the category $\FK{G}_q=\QCoh(\dG)^{G_q}$ of $G_q$-equivariant sheaves.  This gives an action of $\dG$ on $\FK{G}_q$ by braided tensor automorphism.  This action is algebraic, in the precise sense of \cite[Appendix A]{negron21}, and we have the corresponding group map $\dG\to \underline{\Aut}_{\ot}^{br}(\FK{G}_q)$.  In terms of the translation action of $\dG$, the de-equivariantization map from the big quantum group restricts to an equivalence
\[
E_-:\Rep G_q\overset{\sim}\longrightarrow (\FK{G}_q)^{\dG}
\]
onto the monoidal category of $\dG$-equivariant objects in $\FK{G}_q$ \cite[Proposition 4.4]{arkhipovgaitsgory03}.
\par

One can translate much of the analysis in this text from the small quantum group to the big quantum group by restricting to $\dG$-equivariant objects in $\FK{G}_q$.  One can compare, for example, with \cite{arkhipovbezrukavnikovginzburg04,boekujawanakano}.

\section{The quantum Borels}
\label{sect:borel}

We give a presentation of the small quantum Borel which is in line with the presentation of Section \ref{sect:quantumgroups} for the small quantum group.  At the conclusion of the section we extend the construction of the usual (positive) quantum Borel to provide a family of small quantum Borels which are parametrized by the flag variety $\dG/\dB$.
\par

While much of the material of this section is known, or at least deducible from known results in the literature, the construction of the small quantum Borel at an arbitrary geometric point $\lambda:\Spec(K)\to \dG/\dB$ in the flag variety is new.

\subsection{Quantum Frobenius for the Borel, and de-equivariantization}

As with the (big) quantum group, we let $\Rep B_q$ denote the category of integrable $U_q(\mfk{b})$-representations which are appropriately graded by the character lattice $X$.  The quantum Frobenius for the quantum group induces a quantum Frobenius for the quantum Borel, $\Fr:\Rep \dB\to \Rep B_q$ \cite{lusztig93}.  This quantum Frobenius identifies $\Rep \dB$ with the full subcategory of $B_q$-representations whose $X$-grading is supported on the sublattice $X^M$ (see Section \ref{sect:Gq}).
\par

The functor $\Fr$ is a fully faithful tensor embedding, in the sense that its image is closed under taking subquotients.  This implies that the corresponding Hopf algebra map $fr:\O(\dB)\to \O(B_q)$, which one can obtain directly by Tannakian reconstruction, is an inclusion \cite[Lemma 2.2.13]{schauenburg92}.
\par

We now consider the restriction functor $\Rep B_q\to \Rep \uqB$.  Since this functor is surjective, the corresponding Hopf map $\O(B_q)\to (\uqB)^\ast$ is surjective as well \cite[Lemma 2.2.13]{schauenburg92}.  Furthermore, an object in $\Rep B_q$ is in the image of quantum Frobenius if and only if that object has trivial restriction to $\uqB$.  Since $\uqB$ is normal in the big quantum Borel, it follows that $\O(\dB)$ is identified with the $\uqB$-invariants, or $(\uqB)^\ast$-coinvariants, in the quantum function algebra $\O(\dB)=\O(B_q)^{\uqB}$ via the map $fr:\O(\dB)\to \O(B_q)$.  It also follows that the Hopf algebra map $\O(B_q)\to (\uqB)^\ast$ induces an isomorphism from the fiber $k\ot_{\O(\dB)}\O(B_q)\cong (\uqB)^\ast$ \cite[proof of Proposition 3.11]{arkhipovgaitsgory03}.
\par

To rephrase what we have just said; we observe an exact sequence of Hopf algebras
\[
k\to \O(\dB)\overset{fr}\to \O(B_q)\to (\uqB)^\ast\to k
\]
via quantum Frobenius which corresponds to, and is (re)constructed from, the exact sequence of tensor categories
\[
Vect\to \Rep \dB\overset{\Fr}\longrightarrow \Rep B_q\to \Rep(\uqB)\to Vect
\]
\cite[Definition 3.7]{bruguieresnatale11} (cf.\ \cite{etingofgelaki17}).

\begin{proposition}\label{prop:129}
The quantum function algebra $\O(B_q)$ is faithfully flat over $\O(\dB)$, and injective over $\uqB$.
\end{proposition}

\begin{proof}
Faithful flatness of $\O(B_q)$ over $\O(\dB)$ follows by Schneider \cite[Remark 2.5]{schneider93}.  Now by Takeuchi we see that $\O(B_q)$ is coflat as a $(\uqB)^\ast$-comodule \cite[Theorem 1]{takeuchi79}.  Equivalently, $\O(B_q)$ is injective over $\uqB$.
\end{proof}

As with the quantum group, we define the quantum Frobenius kernel $\FK{B}_q$ for the quantum Borel as the category of $B_q$-equivariant sheaves on $\dB$,
\[
\FK{B}_q:=\QCoh(\dB)^{B_q}.
\]
Via the global sections functor, this category of equivariant sheaves is identified with the category of relative Hopf modules $\QCoh(\dB)^{B_q}\overset{\sim}\to {_{\O(\dB)}\mbf{M}^{\O(B_q)}}$.  The following is an immediate application of Proposition \ref{prop:129} and \cite[Theorem 1]{takeuchi79}.

\begin{corollary}\label{cor:438}
Taking the fiber at the identity $1:\Spec(k)\to \dB$ provides an equivalence of (abelian) monoidal categories
\[
1^\ast:\FK{B}_q\overset{\sim}\to \Rep\uqB.
\]
\end{corollary}

As with the quantum group, discussed at Theorem \ref{thm:ag}, the compact/dualizable objects in $\FK{B}_q$ are precisely those equivariant sheaves which are coherent over $\dB$, and all objects in $\FK{B}_q$ are flat over $\dB$.

\subsection{A spectral sequences for $B_q$-extensions}

\begin{lemma}[{\cite{doi81}}]\label{lem:doi}
An object $V$ in $\Rep B_q$ is injective if and only if $V$ is a summand of some additive power $\oplus_{i\in I} \O(B_q)$.
\end{lemma}

\begin{proof}
Comultiplication provides an injective comodule map $V\to \underline{V}\ot \O(B_q)$, where $\underline{V}$ is the vector space associated to the representation $V$.  Since any cofree comodule is injective \cite{doi81}, this inclusion is split.
\end{proof}

Since the $\uqB$-invariants in $\O(B_q)$ are precisely the classical algebra $\O(\dB)$, we observe the following.

\begin{corollary}\label{cor:143}
If $W$ is injective over $B_q$, and $V$ is a finite-dimensional $B_q$-representations, then $\Hom_{\uqB}(V,W)$ is an injective $\dB$-representation.
\end{corollary}

\begin{proof}
We have $\Hom_{\uqB}(V,W)=\Hom_{\uqB}(k,{^\ast V}\ot W)$, and ${^\ast V}\ot W$ is injective over $B_q$ in this case.  So it suffices to assume $V=k$, in which case the result follows by Lemma \ref{lem:doi} and the calculation $\O(\dB)=\O(B_q)^{\uqB}$.
\end{proof}

\begin{proposition}\label{prop:151}
Let $V$ and $W$ be in $\Rep(B_q)$, and assume that $V$ is finite-dimensional.  There is a natural isomorphism
\[
\RHom_B(k,\RHom_{\uqB}(V,W))\cong \RHom_{B_q}(V,W),
\]
and subsequent spectral sequence
\[
E_2^{i,j}=\Ext^i_B(k,\Ext^j_{\uqB}(V,W))\ \Rightarrow\ \Ext^{i+j}_{B_q}(V,W).
\]
\end{proposition}

\begin{proof}
The quantum function algebra $\O(B_q)$ is an injective $\uqB$-module, by Proposition \ref{prop:129}.  It follows by Lemma \ref{lem:doi} that any injective $B_q$-representation restricts to an injective $\uqB$-representation.  So the result follows by Corollary \ref{cor:143}.
\end{proof}

\subsection{Kempf vanishing and a transfer theorem}

We recall some essential relations between quantum group representations and representations for the quantum Borel.  The following vanishing result, which first appears in works of Andersen, Polo, and Wen \cite{andersenpolowen91,andersenwen92} with some restrictions on the order of $q$, appears in complete generality in works of Woodock and Ryom-Hasen \cite[Theorem 8.7]{woodock97} \cite[Lemma 4.3, Theorem 5.5]{ryom03}.

\begin{theorem}\label{thm:kempf_vanish}
Let $\mathsf{I}^0$ denote induction from the quantum Borel, $\mathsf{I}^0:\Rep B_q\to \Rep G_q$.
\begin{enumerate}
\item $\mathsf{I}^0(\1)=\1$.
\item The higher derived functors $\mathsf{I}^{>0}(\1)$ vanish.
\end{enumerate}
\end{theorem}

We can now employ the information of Theorem~\ref{thm:kempf_vanish} and follow exactly the proof of~\cite[Theorem 2.1]{clineparshallscottkallen77} to observe the following transfer theorem.

\begin{theorem}[\cite{clineparshallscottkallen77}]\label{thm:cpsk}
For arbitrary $V$ and $W$ in $\Rep G_q$, and $i\geq 0$, the restriction functor $\Rep G_q\to \Rep B_q$ induces an isomorphism on cohomology
\[
\Ext^i_{G_q}(V,W)\overset{\cong}\to \Ext^i_{B_q}(V,W).
\]
\end{theorem}

\subsection{Quantum Borels indexed by the flag variety}
\label{sect:borels}

For any $k$-point $\lambda:\Spec(k)\to \dG/\dB$ we have the corresponding $\dB$-coset $\iota_\lambda:\dB_\lambda\to \dG$, which is the fiber of $\lambda$ along the quotient map $\pi:\dG\to \dG/\dB$.  This coset is a $\dB$-torsor under the right translation action of $\dB$ and we have the corresponding algebra object $\O(\dB_\lambda)$ in $\Rep \dB\subset \Rep B_q$.  We then consider the monoidal category
\[
\msc{B}_\lambda:=\QCoh(\dB_\lambda)^{B_q}\cong {_{\O(\dB_\lambda)}\mbf{M}^{\O(B_q)}}
\]
of $B_q$-equivariant sheaves on $\dB_\lambda$.  The equivalence with relative Hopf modules here is given by taking global section, and the product on $\msc{B}_\lambda$ is the expected one $\ot=\ot_{\O_{\dB_\lambda}}$.
\par

Restriction along the inclusion $\iota_\lambda:\dB_\lambda\to \dG$ provides a central monoidal functor
\[
\res_\lambda:=\iota_\lambda^\ast:\FK{G}_q\to \msc{B}_\lambda
\]
with central structure given by the $R$-matrix
\[
b_{V,W}:W\ot \res_\lambda(V)\to \res_\lambda(V)\ot W,\ \ b_{W,V}(w\ot v)=R_{21}(v\ot w)
\]
\cite[Definition 4.15]{dgno10}.  We have $\msc{B}_1=\FK{B}_q$ and the functor $\res_1:\FK{G}_q\to \FK{B}_q$ is identified with the standard restriction functor for the small quantum group, in the sense that the diagram
\[
\xymatrix{
\FK{G}_q\ar[rr]^{\res_1}\ar[d]_{1^\ast}^\sim & & \FK{B}_q\ar[d]_{1^\ast}^\sim\\
\Rep \uqG\ar[rr]^{\res} & & \Rep \uqB
}
\]
commutes.
\par

At a general closed point $\lambda$, any choice of a point $x:\Spec(k)\to \dB_\lambda$ provides a $\dB$-equivariant isomorphism $x:\dB\to \dB_\lambda$ given by left translation.  Taking global sections then provides an isomorphism $x:\O(\dB_\lambda)\to \O(\dB)$ of algebra objects in $\Rep B_q$.  So we see that pushing forward along $x$ provides an equivalence of tensor categories $x:\msc{B}_1\to \msc{B}_\lambda$ which fits into a diagram
\begin{equation}\label{eq:502}
\xymatrix{
\FK{G}_q\ar[rr]^x_\sim\ar[d]_{\res} & & \FK{G}_q\ar[d]_{\res_\lambda}\\
\msc{B}_1\ar[rr]^x_\sim & & \msc{B}_\lambda.
}
\end{equation}

Now, let us consider and arbitrary \emph{geometric} point $\lambda:\Spec(K)\to \dG/\dB$.  At $\lambda$ we again have the fiber $\iota_\lambda: \dB_\lambda\to G$, which now has the structure of a $K$-scheme, and which is a torsor over $\dB_K$.  We consider the monoidal category
\[
\msc{B}_\lambda:=\QCoh(\dB_\lambda)^{(B_K)_q}
\]
of equivariant sheaves relative to the base change $(B_K)_q$.  Pulling back along $\iota_\lambda$ again provides a central monoidal functor
\[
\res_\lambda:=\iota_\lambda^\ast:\FK{G}_q\to \msc{B}_\lambda
\]
which factors as a base change map composed with restriction along $\iota_{K,\lambda}$
\[
\res_\lambda=\big(\FK{G}_q\overset{(-)_K}\longrightarrow \FK(G_K)_q\overset{\res_{K,\lambda}}\longrightarrow \msc{B}_\lambda\big).
\]
Here $\iota_{K,\lambda}:\dB_\lambda\to \dG_K$ is the map implied by the universal property of the pullback $\dG_K=\Spec(K)\times \dG$.  All of this is to say that, after base change, the construction of $\msc{B}_\lambda$ at a geometric point for $\dG/\dB$ is no different from the construction at a closed point.

\begin{definition}
At any geometric point $\lambda:\Spec(K)\to \dG/\dB$, the category of sheaves for the associated small quantum Borel is the $\FK{G}_q$-central, monoidal category $\msc{B}_\lambda=\QCoh(\dB)^{(B_K)_q}$.
\end{definition}

The following Proposition is deduced from the equivalence $x:(\msc{B}_1)_K\overset{\sim}\to \msc{B}_\lambda$ provided by any choice of $K$-point $x:\Spec(K)\to \dB_\lambda$.

\begin{proposition}
At each geometric point $\lambda:\Spec(K)\to \dG/\dB$ the monoidal category $\msc{B}_\lambda$ has the following properties:
\begin{itemize}
\item $\msc{B}_\lambda$ has enough projectives and injectives, and an object is projective if and only if it is injective (cf.\ \cite{faithwalker67}).
\item $\msc{B}_\lambda$ admits a compact projective generator.
\item The compact objects in $\msc{B}_\lambda$ are precisely those $(B_K)_q$-equivariant sheaves which are coherent over $\dB_\lambda$, and all compact objects are dualizable.
\item Coherent sheaves in $\msc{B}_\lambda$ form a finite tensor subcategory which is of (Frobenius-Perron) dimension $\dim \uqB$.
\item All objects in $\msc{B}_\lambda$ are flat over $\dB_\lambda$.
\item The central tensor functor $\res_{K,\lambda}:\FK(G_K)_q\to \msc{B}_\lambda$ is surjective.
\end{itemize}
\end{proposition}

\begin{remark}
Of course, the flag variety $\dG/\dB$ is the same as $G/B$, since we are at an odd root of unity and hence do not change Dynkin types when taking the dual group.  However, when we consider quotients $\dG/H_q$ by various quantum groups one does want to explicitly employ the dual $\dG$.  The quotient $\dG/\dB$ will also be the correct object to consider at even order $q$, where Dynkin types do change.
\end{remark}

\subsection{A notational comment}

As mentioned in Section \ref{sect:quantumgroups}, the small quantum group is essentially never referenced in it linear form $\uqG$ in this text.  We will, however, make extensive use of the algebra $\uqB$ throughout.  Furthermore, the algebra $\uqB$ will often appear as a subscript in formulas.  For this reason \emph{we adopt the notation
\[
u:=\uqB
\]
globally throughout this document}.  An unlabeled algebra $u$ which appears anywhere in the text is always the small quantum enveloping algebra $\uqB$ for the positive Borel.

\part{Geometric Enhancements for Quantum Groups}

In Part I of the paper we construct a monoidal enhancement $D^{\Enh}(\FK{G}_q)$ of the derived category of sheaves for the quantum Frobenius kernel.  The morphisms $\sRHom_{\FK{G}_q}(M,N)$ in this category are quasi-coherent dg sheaves over the flag variety $\dG/\dB$--or more precisely are objects in the derived category $D(\dG/\dB)$ of quasi-coherent sheaves--and the monoidal structure on $D^{\Enh}(\FK{G}_q)$ is reflected in natural composition and tensor structure maps for these sheaf-morphisms.\footnote{For a discussion of $D^{\Enh}(\FK{G}_q)$ in relation to the Springer resolution, specifically, see Proposition \ref{prop:1749} and Conjecture \ref{conj:formality} below.}
\par

The construction of the enhancement $D^{\Enh}(\FK{G}_q)$ is facilitated by a certain \emph{half-quantum flag variety}, whose sheaves $\QCoh(\dG/B_q)$ form a sheaf of tensor categories over the classical flag variety.  Our fundamental approach, throughout this work, is to reduce analyses of the enhancement $D^{\Enh}(\FK{G}_q)$ to corresponding analyses of the half-quantum flag variety.  Such reductions are made possible by a strong embedding theorem, referred to as the Kempf embedding theorem, which is proved at Theorem \ref{thm:Kempf} below.

\section{The half-quantum flag variety}
\label{sect:G/Bq}

In Section \ref{sect:borels} we introduced a family of small quantum Borels $\msc{B}_\lambda$ which are parametrized by geometric points for the flag variety.  In this section we consider a universal small quantum Borel from this perspective, which is simply the monoidal category of $B_q$-equivariant sheaves over $\dG$.  This is the category of quasi-coherent sheaves $\QCoh(\dG/B_q)$ on the so-called \emph{half-quantum flag variety}.
\par

Just as one considers each $\msc{B}_\lambda$ as a $K$-linear monoidal category, one should consider the category of sheaves for the half-quantum flag variety as a $\dG/\dB$-linear monoidal category.  Such linearity can be expressed via an action of the category of sheaves over the flag variety on $\QCoh(\dG/B_q)$.

\subsection{The half-quantum flag variety}
\label{sect:G/Bq_intro}

\begin{definition}
The category of quasi-coherent sheaves for the half-quantum flag variety is the abelian monoidal category
\[
\QCoh(\dG/B_q):=\QCoh(\dG)^{B_q}\cong \{\text{Arbitrary $\O(\dG)$-modules in }\Rep B_q\},
\]
with product $\ot=\ot_{\O_{\dG}}$.  We let $\Coh(\dG/B_q)$ denote the full monoidal subcategory of sheaves which are coherent over $\dG$.
\end{definition}

We note that the category $\Coh(\dG/B_q)$ is precisely the category of compact objects in $\QCoh(\dG/B_q)$.  In our notation $\dG/B_q$ should be interpreted (informally) as a stack quotient, so that sheaves on this ``noncommutative space" are $B_q$-equivariant sheaves on $\dG$.
\par

We say an object $M$ in $\QCoh(\dG/B_q)$ is flat if the operation $M\ot-$ is exact.  Since the product for $\QCoh(\dG/B_q)$ is simply the product over $\O_{\dG}$, one sees that an object $M$ in $\QCoh(\dG/B_q)$ is flat whenever its image in $\QCoh(\dG)$ is flat.  Furthermore, one can check that dualizable objects in $\QCoh(\dG/B_q)$ are precisely those objects which are flat and coherent.
\par

As with the usual flag variety \cite[I.5.8]{jantzen03} \cite[3.3]{suslinfriedlanderbendel97b}, we have an equivariant vector bundle functor
\[
E_-:\Rep B_q\to \QCoh(\dG/B_q),\ \ E_V:=\O_G\ot_k V.
\]
This functor is exact and has right adjoint provided by the forgetful functor
\[
-|_{B_q}:\QCoh(\dG/B_q)\to \Rep B_q,
\]
which is defined explicitly by applying global sections $M|_{B_q}=\Gamma(\dG,M)$ and forgetting the $\O(\dG)$-action (cf.\ Section \ref{sect:descent}).

\begin{lemma}
The category $\QCoh(\dG/B_q)$ is complete, in the sense that it has all set indexed limits, and has enough injectives.
\end{lemma}

\begin{proof}
The point is that $\QCoh(\dG/B_q)$ is a Grothendieck abelian category.  That is to say, $\QCoh(\dG/B_q)$ is cocomplete, has exact filtered colimits, and admits a generator.  All Grothendieck categories are complete and have enough injectives \cite[Th\'eor\`eme 1.10.1]{grothendieck57}.
\par

The only controversial issue here is the existence of a generator.  However, since $\rep B_q$ is essentially small we can choose a set of representations $\{V_i\}_{i\in I}$ so that each object in $\rep B_q$ admits a surjection from some $V_i$.  The object $\oplus_{i\in I}E_{V_i}$ is then a generator for $\QCoh(\dG/B_q)$, where $E_V$ is the vector bundle associated to a given $B_q$-representation $V$.
\end{proof}

\begin{remark}
Presumably the category $\QCoh(\dG/B_q)$ admits an ample line bundle $\mcl{L}_{-\rho}$, so that the powers $\oplus_{n\geq 0}\mcl{L}_{\rho}^{\ot n}$ provide a generator for $\QCoh(\dG/B_q)$ (cf.\ \cite[Proposition 4.4]{jantzen03}).
\end{remark}

\subsection{$\QCoh(\dG/B_q)$ as a sheaf of categories}

The quantum Frobenius maps for $G_q$ and $B_q$ fit into a diagram of tensor functors
\[
\xymatrix{
\Rep \dG\ar[rr]^{\Fr}\ar[d]_{\res} & & \Rep G_q\ar[d]^{\res}\\
\Rep \dB\ar[rr]^{\Fr} & & \Rep B_q.
}
\]
This diagram then implies the existence of a fully faithful monoidal embedding
\begin{equation}\label{eq:zeta}
\QCoh(\dG/\dB)\overset{\pi^\ast}\to \QCoh(\dG)^{\dB}\to \QCoh(\dG)^{B_q}=\QCoh(\dG/B_q).
\end{equation}
We let $\zeta^\ast$ denote this embedding.  The embedding $\zeta^\ast:\QCoh(\dG/\dB)\to \QCoh(\dG/B_q)$ admits a canonical central structure, i.e.\ lift to the Drinfeld center,
\[
\QCoh(\dG/\dB)\to Z(\QCoh(\dG/B_q))
\]
provided by the trivial symmetry
\begin{equation}\label{eq:symm}
\operatorname{symm}_{M,\msc{F}}:M\ot \zeta^\ast(\msc{F})\to \zeta^\ast(F)\ot M,\ \ \operatorname{symm}_{M,\msc{F}}(m\ot s)=s\ot m.
\end{equation}
This trivial symmetry gives $\QCoh(\dG/B_q)$ the structure of a symmetric bimodule category over $\QCoh(\dG/\dB)$.
\par

For $\msc{F}$ in $\QCoh(\dG/\dB)$ we let
\[
\msc{F}\star-:\Coh(\dG/B_q)\to \Coh(\dG/B_q)
\]
denote the corresponding action map, $\msc{F}\star-=\zeta^\ast(\msc{F})\ot-$.  The operation $\msc{F}\star-$ is exact whenever $\msc{F}$ is flat over $\dG/\dB$, and $-\star M$ is exat whenever $M$ is flat over $\dG/B_q$.
\par

One might think of this action as providing the quantum flag variety with the structure of a sheaf of categories over the classical flag variety $\dG/\dB$, whose sections over any open embedding $U\to \dG/\dB$, for example, are given by the base change $\QCoh(U)\ot_{\QCoh(\dG/\dB)}\QCoh(\dG/B_q)$ (see \cite{gaitsgory15}).

\begin{notation}
The functor $\zeta^\ast:\QCoh(\dG/\dB)\to \QCoh(\dG/B_q)$ is specifically the exact monoidal functor of \eqref{eq:zeta} along with the central structure \eqref{eq:symm}.
\end{notation}

Although $\zeta^\ast$ is not precisely the pullback equivalence $\pi^\ast:\QCoh(\dG/\dB)\overset{\sim}\to \QCoh(\dG)^{\dB}$, due to the appearance of quantum Frobenius, we will often abuse notation and write simply $\pi^\ast(\msc{F})$ for the object $\zeta^\ast(\msc{F})$.

\subsection{Sheafy morphisms over $\dG/\dB$}

We have just seen that $\QCoh(\dG/B_q)$ admits a natural module category structure over $\QCoh(\dG/\dB)$, and so becomes a sheaf of tensor categories over the flag variety.  Inner morphisms for this module category/sheaf structure provide a sheaf-Hom functor for $\QCoh(\dG/B_q)$.

\begin{lemma}
For any $M$ in $\QCoh(\dG/B_q)$, the operation $-\star M:\QCoh(\dG/\dB)\to \QCoh(\dG/B_q)$ has a right adjoint $\sHom_{\dG/B_q}(M,-)$.  The functors $\sHom_{\dG/B_q}(M,-)$ are furthermore natural in $M$, so that we have a bifunctor
\[
\sHom_{\dG/B_q}:\QCoh(\dG/B_q)^{op}\times \QCoh(\dG/B_q)\to \QCoh(\dG/B_q).
\]
\end{lemma}

\begin{proof}
The functor $-\star M$ commutes with colimits and thus admits a right adjoint.  Naturality in $M$ follows by Yoneda's lemma.
\end{proof}

We note that the functor $\sHom_{\dG/B_q}$ is left exact in both coordinates, since the functor
\[
\Hom_{\dG/\dB}(\msc{F},\sHom_{\dG/B_q}(-,-))\cong \Hom_{\dG/\dB}(\msc{F}\star-,-)
\]
is left exact in each coordinate at arbitrary $\msc{F}$.
\par

In Sections \ref{sect:shom1}--\ref{sect:shom2} we observe that, via general nonsense with adjunctions (cf.\ \cite{ostrik03,etingofostrik04}), the bifunctor $\sHom_{\dG/B_q}$ admits natural composition and monoidal structure maps.  These structure maps localize the monoidal structure on $\QCoh(\dG/B_q)$, in the sense that they recover the composition and tensor structure maps for $\QCoh(\dG/B_q)$ after taking global sections.  In the language of Section \ref{sect:enrich}, we are claiming specifically that the pairing
\begin{equation}\label{eq:991}
(\!\ \operatorname{obj}\QCoh(\dG/B_q),\ \sHom_{\dG/B_q}\!\ )
\end{equation}
provides a monoidal enhancement of $\QCoh(\dG/B_q)$ in the category of sheaves over the flag variety.
\par

Before delving further into these issues, we explain the (essential) role of the half-quantum flag variety in our study of the small quantum group.

\section{The Kempf embedding $\FK{G}_q\to \QCoh(\dG/B_q)$}
\label{sect:Kempf}

Let us consider again the quantum Frobenius kernel $\FK{G}_q$.  We have the obvious forgetful functor
\begin{equation}\label{eq:funcTOR}
\FK{G}_q=\QCoh(\dG)^{G_q}\to \QCoh(\dG)^{B_q}=\QCoh(\dG/B_q)
\end{equation}
This functor is immediately seen to be monoidal, and the $R$-matrix for the quantum group provides it with a central structure (see Section \ref{sect:Kempf_Z} below).  In this section we prove that the functor \eqref{eq:funcTOR} is fully faithful, and induces a fully faithful functor on unbounded derived categories as well.
\par

In the statement of the following theorem, $D(\FK{G}_q)$ denotes the unbounded derived category of complexes in $\FK{G}_q$, and $D(\dG/B_q)$ is the unbounded derived category of complexes of quasi-coherent sheaves for the half-quantum flag variety.

\begin{theorem}\label{thm:Kempf}
The forgetful functor $\FK{G}_q\to \QCoh(\dG/B_q)$ is fully faithful and induces a fully faithful monoidal embedding
\[
\Kempf:D(\FK{G}_q)\to D(\dG/B_q)
\]
for the corresponding unbounded derived categories.
\end{theorem}

\begin{proof}
Follows immediately by Theorem \ref{thm:kempf_ext} below.
\end{proof}

We refer to this embedding as the Kempf embedding, since the result essentially follows by Kempf vanishing \cite{kempf76}, or rather the quantum analog of Kempf vanishing provided in \cite{andersenpolowen91,woodock97,ryom03}.  From the perspective of the small quantum Borels the embedding of Theorem \ref{thm:Kempf} is a kind of universal restriction functor.

\subsection{Compact sheaves over the half-quantum flag variety}

We begin with a little lemma, the proof of which is deferred to Section \ref{sect:relative_p} below.

\begin{lemma}\label{lem:1090}
Suppose that $V$ is a finite-dimensional $B_q$-representation and that the restriction of $V$ to the small quantum Borel $\uqB$ is projective.  Then the associated vector bundle $E_V$ over $\dG/B_q$ is compact in $D(\dG/B_q)$.  Similarly, any summand of $E_V$ is compact in $D(\dG/B_q)$.
\end{lemma}

Now, by Proposition \ref{prop:andersen}, all coherent projectives in $\FK{G}_q$ are summands of vector bundles $E_V$ with $V$ projective over $G_q$, and hence projective over $\uqG$ and $\uqB$ as well.  So Lemma \ref{lem:1090} implies the following.

\begin{corollary}\label{cor:1110}
The functor $D(\FK{G}_q)\to D(\dG/B_q)$ induced by the forgetful functor sends compact objects in $D(\FK{G}_q)$ to compact objects in $D(\dG/B_q)$.
\end{corollary}

\subsection{Kempf embedding via extensions}

\begin{theorem}\label{thm:kempf_ext}
The forgetful functor $\FK{G}_q\to \QCoh(\dG/B_q)$ induces an isomorphism on cohomology
\begin{equation}\label{eq:1093}
\Ext^i_{\FK{G}_q}(M,N)\overset{\cong}\longrightarrow \Ext^i_{\dG/B_q}(M,N),
\end{equation}
for all $i$ and all $M$ and $N$ in $D(\FK{G}_q)$.
\end{theorem}

\begin{proof}
The forgetful functor is exact and hence induces a map on unbounded derived categories.  We first claim that the map on extensions is an isomorphism whenever $M$ is in the image of the de-equivariantization/equivariant vector bundle map $E_-:\Rep G_q\to \FK{G}_q$, and $N$ is arbitrary in $\FK{G}_q$.  Recall that this vector bundle functor has an exact right adjoint $-|_{G_q}:\FK{G}_q\to \Rep G_q$ provided by taking global sections and forgetting the $\O(\dG)$-action, and we have the analogous adjunction for $\QCoh(\dG/B_q)$, as discussed in Section \ref{sect:G/Bq_intro}.  (We make no notational distinction between vector bundles in $\FK{G}_q$ and in $\QCoh(\dG/B_q)$, and rely on the context to distinguish the two classes of sheaves.)
\par

Since the equivariant vector bundle functors for $\FK{G}_q$ and $\QCoh(\dG/B_q)$ are exact, the adjoints $-|_{G_q}$ and $-|_{B_q}$ preserve injectives.  So we have identifications
\[
\Ext^\ast_{\FK{G}_q}(E_V,N)\overset{\cong}\to \Ext^\ast_{G_q}(V,N|_{G_q})\ \ \text{and}\ \ \Ext^\ast_{\dG/B_q}(E_V,N)\overset{\cong}\to \Ext^\ast_{B_q}(V,N|_{B_q})
\]
which fit into a diagram
\[
\xymatrix{
\Ext^\ast_{\FK{G}_q}(E_V,N)\ar[rr]^{forget}\ar[d]_\cong & & \Ext^\ast_{\dG/B_q}(E_V,N)\ar[d]^\cong\\
\Ext^\ast_{G_q}(V,N|_{G_q})\ar[rr]^{restrict}_{\cong} & & \Ext^\ast_{B_q}(V,N|_{B_q}).
}
\]
The bottom morphism here is an isomorphism by Kempf vanishing, or more precisely by the transfer theorem of Theorem \ref{thm:cpsk}.  So we conclude that the top map is an isomorphism.
\par

We now understand that the map \eqref{eq:1093} is an isomorphism whenever $M=E_V$ for some $V$ in $\Rep G_q$ and $N$ is arbitrary in $\FK{G}_q$.  It follows that \eqref{eq:1093} is an isomorphism whenever $M$ is a summand of a vector bundle $E_V$ and $N$ is arbitrary in $\FK{G}_q$, and hence whenever $M$ is simple or projective in $\FK{G}_q$ by Proposition \ref{prop:andersen}.
\par

Since coherent projectives in $\FK{G}_q$ are compact in $D(\FK{G}_q)$, and the functor $D(\FK{G}_q)\to D(\dG/B_q)$ preserves compact objects by Corollary \ref{cor:1110}, it follows that \eqref{eq:1093} is an isomorphism whenever $M$ is a coherent projective in $\FK{G}_q$ and $N$ is in the localizing subcategory
\begin{equation}\label{eq:1118}
\operatorname{Loc}(\FK{G}_q)=\operatorname{Loc}(D(\FK{G}_q)^{\heartsuit})\subset D(\FK{G}_q)
\end{equation}
generated by the heart of the derived category.  But, this localizing subcategory is all of $D(\FK{G}_q)$ \cite[\S 5.10]{krause10}, so that the map \eqref{eq:1093} is an isomorphism whenever $M$ is coherent projective and $N$ is arbitrary in $D(\FK{G}_q)$.  We use the more precise identification
\[
\operatorname{Loc}(\operatorname{proj}\FK{G}_q)=D(\FK{G}_q)
\]
\cite[\S 5.10]{krause10}, where $\operatorname{proj}\FK{G}_q$ denotes the category of coherent projectives in $\FK{G}_q$, to see now that \eqref{eq:1093} is an isomorphism at arbitrary $M$ and $N$.
\end{proof}

\subsection{A discussion of Kempf's embedding}

Let us take a moment to discuss the role that Theorem \ref{thm:Kempf} plays in our analysis of the small quantum group.
\par

By the Kempf embedding theorem we can consider the derived category of sheaves for the quantum Frobenius kernel $\FK{G}_q$ as a (full) monoidal subcategory in the derived category of sheaves over $\dG/B_q$.  At the derived level, we still have the action of $D(\dG/\dB)$ on $D(\dG/B_q)$, and again speak of $D(\dG/B_q)$ as a sheaf of categories over the flag variety (cf.\ Section \ref{sect:GRT}).
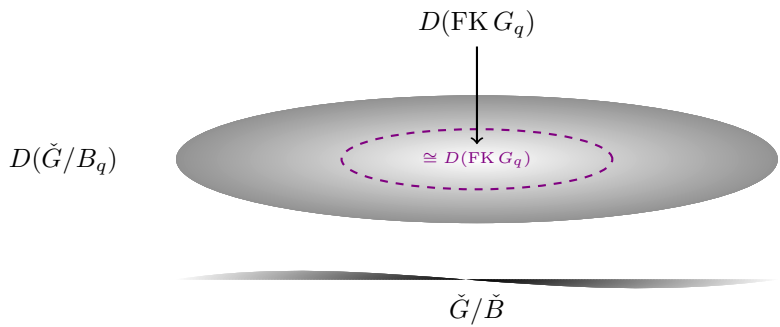
\[
\begin{tikzpicture}
\shade[inner color=white,outer color=black!43](1,1) ellipse (4cm and .85cm);
\draw[color=violet, thick, dashed] (1,1) ellipse (1.8cm and .4cm);
\draw (-4.5,1) node {$D(\dG/B_q)$};
\draw (1,2.8) node {$D(\FK{G}_q)$};
\shade[inner color=black,outer color=white] (-3,-.6) .. controls (-0.4,-.2) and (2,-1) .. (5,-.6);
\draw (1,-1) node {$\dG/\dB$};
\draw[thick, ->] (1,2.5) -- (1,1.2);
\draw[color=violet] (1,1) node {\tiny$\cong D(\FK{G}_q)$};
\end{tikzpicture}
\]

We note, however, that the monoidal subcategory $D(\FK{G}_q)\subset D(\dG/B_q)$ is \emph{not} stable under the action of $D(\dG/\dB)$ (one can check this explicitly).  Equivalently, if we think in the mode of sheaves of categories, $\FK{G}_q$ is not a sheaf of categories over $\dG/\dB$, or rather does not admit such a structure which is induced by the given inclusion into $\QCoh(\dG/B_q)$.  So, one can reasonably say that we are including only into the global sections
\[
D(\FK{G}_q)\hookrightarrow \text{global sections of $D(\dG/B_q)$ over }\dG/\dB.
\]
\par

In any case, although $D(\FK{G}_q)$ is not stable under the $D(\dG/\dB)$ action on $D(\dG/B_q)$, we can consider a ``dual" structure on $D(\dG/B_q)$ provided by this action, which is a corresponding enhancement in the symmetric monoidal category of sheaves on $\dG/\dB$.  This enhanced structure is, of course, provided by the inner-Homs $\sHom_{\dG/B_q}$ of Section \ref{sect:G/Bq} and is described further in Section \ref{sect:shom1}--\ref{sect:Enh_derived} below.  We record an obvious lemma.

\begin{lemma}
Suppose $T=(\operatorname{obj}T,\ot,\Hom_T)$ is a monoidal category which is equipped with an enhancement $(\operatorname{obj}T,\ot,\sHom_T)$ in a symmetric monoidal category $\mcl{Q}$.  Then any full monoidal subcategory $S\subset T$ inherits a monoidal $\mcl{Q}$-enhancement.  Namely, we have the enhancement $(\operatorname{obj}S,\ot,\sHom_S)$ where $\sHom_S(X,Y)=\sHom_{T}(X,Y)$ for each pair of objects $X$ and $Y$ in $S$.
\end{lemma}

So, the $\dG/\dB$-\emph{enhancement} for $D(\dG/B_q)$ will induce a monoidal enhancement for the full subcategory $D(\FK{G}_q)$ via the Kempf embedding.  In this sense the sheaf structure on $D(\dG/B_q)$ over $\dG/\dB$ still provides us with a means of employing the flag variety in an analysis of the category of quantum group representations, as a tensor category.  This entrance of the flag variety also leads naturally to an appearance of the Springer resolution in the derived category of sheaves over the half-quantum flag variety, and the derived category of representations over the quantum group as well (see Sections \ref{sect:RGq} \& \ref{sect:AH!}).
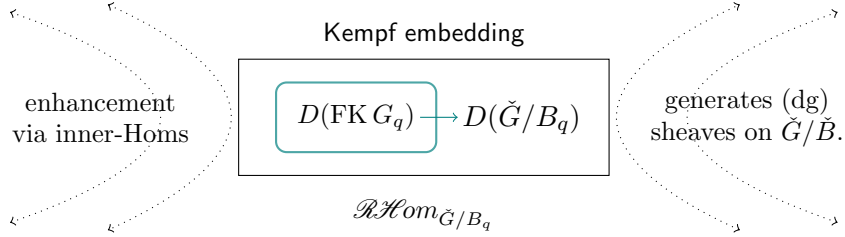
\[
\begin{tikzpicture}
\node at (0,1.1) {\sf Kempf embedding};
\node at (0,-1.3) {$\sRHom_{\dG/B_q}$};
\node[align=center,draw=teal!70, thick, rounded corners, inner sep=8pt] at  (-.9,0) {$D(\FK{G}_q)$};
\draw[->, color=teal] (-.05,0) -- (.45,0);
\node[align=center,draw=black, inner sep=16pt] at (0,0) {\hspace{3.8cm}\\ \hspace{3.8cm}\ };
\node[align=center] at (1.3,0) {\large$D(\dG/B_q)$};
\draw[dotted,<->,rotate=-90] (1.5,4.2) parabola bend (0,2.56) (-1.5,4.2);
\draw[dotted,<->,rotate=90] (1.5,4.2) parabola bend (0,2.56) (-1.5,4.2);
\draw[dotted,<->,rotate=-90] (1.4,5.5) parabola bend (0,3.5) (-1.4,5.5);
\draw[dotted,<->,rotate=90] (1.4,5.5) parabola bend (0,3.5) (-1.4,5.5);

\node[align=center] at (-4.3,0) {enhancement\\ via inner-Homs};
\node[align=center] at (4.3,0) {generates (dg)\\ sheaves on $\dG/\dB$.};
\end{tikzpicture}
\]

\subsection{Remarks on centrality of the Kempf embedding}
\label{sect:Kempf_Z}

Let us conclude with a remark on centrality of the Kempf embedding.
\par

We have argued above that one should consider $\QCoh(\dG/B_q)$ as a tensor category over $\QCoh(\dG/\dB)$, with its given symmetric bimodule structure.  Hence the appropriate ``Drinfeld center" for $\QCoh(\dG/B_q)$ should be the centralizer
\[
Z_{\dG/\dB}(\QCoh(\dG/B_q)):=\text{the centralizer of $\QCoh(\dG/\dB)$ in }Z(\QCoh(\dG/B_q)).
\]
We have the central structure $\Kempf':\FK{G}_q\to Z(\QCoh(\dG/B_q))$ for the Kempf embedding which is provided by the $R$-matrix in the expected way,
\begin{equation}\label{eq:1237}
\beta_{N,M}:N\ot \Kempf(M)\to \Kempf(M)\ot N,\ \ \beta_{N,M}(n,m)=R_{21}(m\ot n).
\end{equation}
One sees directly that for $N$ in the image of the embedding $\zeta:\QCoh(\dG/\dB)\to \QCoh(\dG/B_q)$, $\beta_{N,M}$ is the trivial symmetry $\operatorname{symm}_{N,M}$.  It follows that the image of $\Kempf'$ does in fact lie in the centralizer of $\QCoh(\dG/\dB)$, so that we restrict to obtain the desired ``$\QCoh(\dG/\dB)$-linear" central structure
\begin{equation}\label{eq:1241}
\Kempf':\FK{G}_q\to Z_{\dG/\dB}(\QCoh(\dG/B_q)).
\end{equation}
When we speak of the Kempf embedding as a central tensor functor we mean specifically the functor $\Kempf:\FK{G}_q\to \QCoh(\dG/B_q)$ along with the lift \eqref{eq:1241} provided by the half-braidings \eqref{eq:1237}.

\section{Structure of $\sHom$ I: Linearity, composition, and tensoring}
\label{sect:shom1}

In the next two sections we provide an analysis of the inner-Hom, or sheaf-Hom, functor for the action of $\QCoh(\dG/\dB)$ on $\QCoh(\dG/B_q)$.  Here we show that sheaf-Homs admit natural composition and monoidal structure maps, and so provide a monoidal enhancement $\QCoh^{\Enh}(\dG/B_q)$ for the monoidal category of sheaves on the half-quantum flag variety.  Many of the results of this section are completely general and completely formal.  So some proofs are sketched and/or delayed to the appendix.
\par

In the subsequent section, Section \ref{sect:shom2}, we provide an explicit description of the sheafy morphisms $\sHom_{\dG/B_q}$ and describe objects in $\QCoh(\dG/B_q)$ which are projective for this functor.

\begin{remark}
Sections \ref{sect:shom1}--\ref{sect:Enh_derived} are in some sense ``busywork", as we are just taking account of various structures for the inner-Hom functor, both at the abelian and derived levels.  One might therefore skim these contents on a first reading.  The next substantial result comes in Section \ref{sect:fibers_G/B}, where we calculate the fibers of derived sheaf-Hom $\operatorname{L}\lambda^\ast\sRHom_{\dG/B_q}$ over the flag variety in terms of derived maps over the small quantum Borels.
\end{remark}

\subsection{$\QCoh(\dG/\dB)$-linearity of sheaf-Hom}

The adjoint to the identity map $id:\sHom_{\dG/B_q}(M,N)\to \sHom_{\dG/B_q}(M,N)$ provides an evaluation morphism
\[
ev:\sHom_{\dG/B_q}(M,N)\star M\to N.
\]
The evaluation is just the counit for the $(\star,\sHom)$-adjunction.  For any $\msc{F}$ in $\QCoh(\dG/\dB)$ the map
\[
id\ot ev:\msc{F}\star(\sHom_{\dG/B_q}(M,N)\star M)\to \msc{F}\star N
\]
provides a natural morphism $\msc{F}\star\sHom(M,N)\to \sHom(M,\msc{F}\star N)$ in the category of sheaves over the flag variety.  In analyzing this natural morphism it is helpful to consider a notion of projectivity for the sheaf-Hom functor.

\begin{definition}\label{def:loc_inj/proj}
An object $M$ in $\QCoh(\dG/\dB)$ is called relatively projective (resp.\ relatively injective) if the functor $\sHom_{\dG/B_q}(M,-)$ (resp.\ $\sHom_{\dG/B_q}(-,M)$) is exact.
\end{definition}

Discussions of relatively injective and projective sheaves are provided in Lemma \ref{lem:loc_inj} and Section \ref{sect:proof1090} below, respectively.

\begin{lemma}[{cf.\ \cite[Lemma 3.3]{ostrik03}}]\label{lem:sHom-linear}
Consider the structural map
\begin{equation}\label{eq:469}
\msc{F}\star \sHom_{\dG/B_q}(M,N)\to \sHom_{\dG/B_q}(M,\msc{F}\star N)
\end{equation}
at $M$ and $N$ in $\QCoh(\dG/B_q)$, and $\msc{F}$ in $\QCoh(\dG/\dB)$.  The map \eqref{eq:469} is an isomorphism under any of the following hypotheses:
\begin{itemize}
\item $\msc{F}$ is a coherent vector bundle.
\item $M$ is relatively projective and $\msc{F}$ is coherent.
\item $M$ is coherent and relatively projective, and $\msc{F}$ is arbitrary.
\item $M$ admits a presentation $E'\to E\to M$ by coherent, relatively projective sheaves and $\msc{F}$ is flat.
\end{itemize}
\end{lemma}

\begin{proof}
When $\msc{F}$ is a vector bundle it is dualizable, so that we have (explicit) natural isomorphisms \cite[Lemma 2.2]{ostrik03}
\[
\begin{array}{rl}
\Hom_{\dG/\dB}(-,\msc{F}\star \sHom(M,N))& \cong \Hom_{\dG/B_q}(\msc{F}^\vee\ot-,\sHom(M,N))\vspace{1mm}\\
&\cong \Hom_{\dG/B_q}((\msc{F}^\vee\ot -)\star M,N)\vspace{1mm}\\
&\cong \Hom_{\dG/B_q}(-\star M,\msc{F}\star N)\vspace{1mm}\\
&\cong \Hom_{\dG/\dB}(-,\sHom(M,\msc{F}\star N))
\end{array}
\]
and deduce an isomorphism $\msc{F}\star \sHom(M,N)\cong \sHom(M,\msc{F}\star N)$ via Yoneda.  One traces the identity map through the above sequence to see that this isomorphism is in fact just the structure map \eqref{eq:469}.  The second statement follows from the first, after we resolve $\msc{F}$ by vector bundles.  The third statement follows from the second and the fact that $\sHom_{\dG/B_q}(M,-)$ commutes with colimits in this case.  The fourth statement follows from the second by resolving $M$ by relatively projective coherent sheaves.
\end{proof}

We will see at Corollary \ref{cor:enough_proj} below that all coherent sheaves in $\QCoh(\dG/B_q)$ admit a resolution by relative projectives.  So the fourth point of Lemma \ref{lem:sHom-linear} simply says that $\sHom_{\dG/B_q}(M,-)$ is linear with respect to the action of flat sheaves over $\dG/\dB$, whenever $M$ is coherent. 

\begin{remark}
One should compare Lemma \ref{lem:sHom-linear} with the familiar case of a linear category.  For a linear category $\msc{C}$, i.e.\ a module category over $Vect$, we have natural maps $V\ot_k\Hom_{\msc{C}}(A,B)\to \Hom_{\msc{C}}(A,V\ot_k B)$ which one generally thinks of as associated to an identification between $V\ot_k-$ and a large coproduct.  This map is an isomorphism provided $V$ is sufficiently finite (dualizable), or $A$ is sufficiently finite (compact).
\end{remark}

\subsection{Enhancing $\QCoh(\dG/B_q)$ via sheaf-Hom}

The evaluation maps for $\sHom_{\dG/B_q}$ provide a natural composition function
\[
\circ:\sHom_{\dG/B_q}(M,N)\ot\sHom_{\dG/B_q}(L,M)\to \sHom_{\dG/B_q}(L,N)
\]
which is adjoint to the map
\[
\sHom_{\dG/B_q}(M,N)\ot\sHom_{\dG/B_q}(L,M)\star L\overset{id\ot ev}\longrightarrow \sHom_{\dG/B_q}(M,N)\star M\overset{ev}\to N.
\]
We also have monoidal structure maps
\[
t:\sHom_{\dG/B_q}(M,N)\ot\sHom_{\dG/B_q}(M',N')\to \sHom(M\ot M',N\ot N')
\]
which are adjoint to the composition
\[
\begin{array}{l}
\sHom_{\dG/B_q}(M_1,N_1)\ot \sHom_{\dG/B_q}(M_2,N_2)\star (M_1\ot M_2)\vspace{2mm}\\
\hspace{1cm}\overset{symm}\longrightarrow(\sHom_{\dG/B_q}(M_1,N_1)\star M_1)\ot (\sHom_{\dG/B_q}(M_2,N_2)\star M_2)\vspace{2mm}\\
\hspace{2cm}\overset{ev\ot ev}\longrightarrow N_1\ot N_2.
\end{array}
\]
One can check the following basic claim, for which we provide a proof in Appendix \ref{sect:A} (cf.\ \cite{ostrik03,etingofostrik04}).

\begin{proposition}\label{prop:enriched}
The composition and monoidal structure maps for $\sHom_{\dG/B_q}$ are associative, and are compatible in the sense of Section \ref{sect:enrich}.
\end{proposition}

This result says that the pairing of objects from $\QCoh(\dG/B_q)$, along with the sheaf-morphisms $\sHom_{\dG/B_q}$, constitutes a monoidal category enriched in the symmetric monoidal category of quasi-coherent sheaves on $\dG/\dB$.

\begin{definition}
We let $\QCoh^{\Enh}(\dG/B_q)$ denote the enriched monoidal category
\[
\QCoh^{\Enh}(\dG/B_q):=\big(\operatorname{obj}\QCoh(\dG/B_q),\ \sHom_{\dG/B_q}\big),
\]
with composition and tensor structure maps as described above.
\end{definition}

As the notation suggests, the category $\QCoh^{\Enh}(\dG/B_q)$ does in fact provide an enhancement for the category of sheaves on the half-quantum flag variety.

\begin{theorem}\label{thm:enhA}
The adjunction isomorphism
\[
\Gamma(\dG/\dB,\sHom_{\dG/B_q}(-,-))\overset{\cong}\to \Hom_{\dG/B_q}(-,-)
\]
induces a isomorphism of monoidal categories
\[
\Gamma(\dG/\dB,\QCoh^{\Enh}(\dG/B_q))\overset{\cong}\to \QCoh(\dG/B_q).
\]
\end{theorem}

The proof of Theorem \ref{thm:enhA} is essentially the same as \cite[proof of Lemma 3.4.9]{riehl14}, for example, and is outlined in Appendix \ref{sect:A}.  Let us enumerate the main points here in any case.  Recall that the linear structure on $\QCoh(\dG/B_q)$ corresponds to an action of $Vect$ on $\QCoh(\dG/B_q)$.  The inner-Homs with respect to this action are the usual vector space of morphisms $\Hom_{\dG/B_q}$ with expected evaluation
\[
ev:\Hom_{\dG/B_q}(M,N)\ot_k M\to N,\ \ f\ot m\mapsto f(m).
\]
This evaluation map specifies a unique binatural morphism
\begin{equation}\label{eq:978}
\Hom_{\dG/B_q}(M,N)\ot_k \O_{\dG/\dB}\to \sHom_{\dG/B_q}(M,N)
\end{equation}
which is compatible with evaluation, in the sense that the diagram
\begin{equation}\label{eq:982}
\xymatrix{
\Hom_{\dG/B_q}(M,N)\ot_k M\ar[d] \ar[r] & N\\
\sHom_{\dG/B_q}(M,N)\star M\ar[ur] & 
}
\end{equation}
commutes.  If we consider $\Hom_{\dG/B_q}(M,N)$ as a constant sheaf of vector spaces, the map \eqref{eq:978} is specified by a morphism of sheaves $\Hom_{\dG/B_q}(M,N)\to \sHom_{\dG/B_q}(M,N)$, which is in turn specified by its value on global sections.
\par

One can check that the global sections of the map $\Hom_{\dG/B_q}(M,N)\to \sHom_{\dG/B_q}(M,N)$ of \eqref{eq:978} recovers the adjunction isomorphism referenced in Theorem \ref{thm:enhA}.  One then uses compatibility with evaluation \eqref{eq:982} to see that the adjunction isomorphism is compatible with composition and the monoidal structure maps, and hence that $\QCoh^{\Enh}(\dG/B_q)$ provides an enhancement of $\QCoh(\dG/B_q)$ over the flag variety, as claimed.

\subsection{Implications for the quantum Frobenius kernel $\FK{G}_q$}

Recall that the forgetful functor $\FK{G}_q=\Coh(\dG)^{G_q}\to \QCoh(\dG/B_q)$ is a monoidal embedding, which we have called the Kempf embedding.  By restricting along this embedding the enhancement $\sHom_{\dG/B_q}$ for $\QCoh(\dG/B_q)$ restricts to an monoidal enhancement for the quantum Frobenius kernel.

\begin{theorem}
Let $\sHom_{\FK{G}_q}$ denote the restriction of the inner-$\operatorname{Hom}$s $\sHom_{\dG/B_q}$ to the full monoidal subcategory $\FK{G}_q$ via Kempf embedding,
\[
\sHom_{\FK{G}_q}(M,N):=\sHom_{\dG/B_q}(M,N).
\]
Then the pairing $(\operatorname{obj}\FK{G}_q,\ \sHom_{\FK{G}_q})$ provides a monoidal enhancement for the quantum Frobenius kernel $\FK{G}_q$ in the category of quasi-coherent sheaves over the flag variety $\dG/\dB$.
\end{theorem}

\subsection{A local-to-global spectral sequence}

The following lemma says that injectives in $\QCoh(\dG/B_q)$ are relatively injective for the sheaf-Hom functor.

\begin{lemma}\label{lem:loc_inj}
If $M$ is flat over $\dG/B_q$ then the functor $\sHom_{\dG/B_q}(M,-)$ sends injectives in $\QCoh(\dG/B_q)$ to injectives in $\QCoh(\dG/\dB)$.  Additionally, when $I$ is injective over $\dG/B_q$, the functor $\sHom_{\dG/B_q}(-,I)$ is exact.
\end{lemma}

\begin{proof}
Suppose that $I$ is injective over $\dG/B_q$ and that $M$ is flat.  Then the functor $\Hom_{\dG/B_q}(-\ot M,I)$ is exact, and hence $\Hom_{\dG/B_q}(-\star M,I)$ is an exact functor from $\QCoh(\dG/\dB)$.  Via adjunction we find that the functor $\Hom_{\dG/B_q}(-,\sHom_{\dG/B_q}(M,I))$ is exact.  So $\sHom_{\dG/B_q}(M,I)$ is injective.
\par

Now, if $I$ is injective then for all vector bundles $\msc{E}$ over $\dG/\dB$ the operation $\Hom_{\dG/\dB}(\msc{E}\star -,I)$ is exact, which implies that each functor
\[
\Hom_{\dG/\dB}(\msc{E},\sHom_{\dG/B_q}(-,I))
\]
is exact.  This is sufficient to ensure that $\sHom_{\dG/B_q}(-,I)$ is exact.
\end{proof}

Recall that we have the natural identification
\[
\begin{array}{rl}
\Hom_{\dG/B_q}(M,N)&\cong \Hom_{\dG/\dB}(\O_{\dG/\dB},\sHom_{\dG/B_q}(M,N))\vspace{1mm}\\
&=\Gamma(\dG/\dB,\sHom_{\dG/B_q}(M,N))
\end{array}
\]
provided by adjunction.  We therefore obtain a natural map
\[
\RHom_{\dG/B_q}(M,N)\to \operatorname{R}\Gamma(\dG/\dB,\sRHom_{\dG/B_q}(M,N)),
\]
where we derive $\sHom_{\dG/B_q}(M,-)$ by taking injective resolutions.  Lemma \ref{lem:loc_inj} implies that this map is a quasi-isomorphism.

\begin{corollary}\label{cor:372}
The natural map $\RHom_{\dG/B_q}(M,N)\to \operatorname{R}\Gamma(\dG/\dB,\sRHom_{\dG/B_q}(M,N))$ is a quasi-isomorphism.  Hence we have a local-to-global spectral sequence
\[
H^\ast(\dG/\dB,\sExt^\ast_{\dG/B_q}(M,N))\ \Rightarrow\ \Ext^\ast_{\dG/B_q}(M,N)
\]
at arbitrary $M$ and $N$ in $\QCoh(\dG/B_q)$.
\end{corollary}

\begin{remark}
Spectral sequences analogous to Corollary \ref{cor:372} can be found in much earlier works of Suslin, Friedlander, and Bendal \cite[Theorem 3.6]{suslinfriedlanderbendel97b}.  So certain pieces of the enhancement $\QCoh^{\Enh}(\dG/\dB)$, and Kempf embedding, had already been employed in works which appeared as early as the 90's.
\end{remark}

\section{Structure of $\sHom$ II: explicit description of sheaf-Hom}
\label{sect:shom2}

We show that the morphisms $\sHom_{\dG/B_q}(M,N)$ are explicitly the descent
\begin{equation}\label{eq:nmbvz}
\sHom_{\dG/B_q}(M,N)=\text{descent of }\Hom_{\QCoh(\dG)^{\uqB}}(M,N)^\sim
\end{equation}
of morphisms in the category $\QCoh(\dG)^{\uqB}$ of $\uqB$-equivariant sheaves over $\dG$, whenever $M$ is coherent.  Here $\uqB$ is taken to act trivially on $\dG$, so that $\uqB$ acts by $\O_{\dG}$-linear endomorphisms on such sheaves.  Also, implicit in the above formula is a claim that the morphism spaces over $\QCoh(\dG)^{\uqB}$ admit natural, compatible actions of $\O(\dG)$ and $\dB$, so that the associated sheaf over $\dG$ is $\dB$-equivariant.  We then apply descent to produce a corresponding sheaf on the flag variety.  (See Proposition \ref{prop:shom_exp} below.)
\par

We note that the category $\Coh(\dG)^{\uqB}$ is also monoidal under the product $\ot=\ot_{\O_{\dG}}$.  We prove additionally that, under the identification \eqref{eq:nmbvz} the composition and tensor structure maps for $\sHom_{\dG/B_q}$ are identified with those provided by the monoidal structure on $\Coh(\dG)^{\uqB}$.

\subsection{$B_q$-equivariant structure on $\sHom_{\dG}$}
\label{sect:HomG}

Recall that for an algebraic group $H$ acting on a finite-type scheme $Y$, the usual sheaf-Hom functor $\sHom_Y$ provides the inner-Homs for the tensor action of $\Coh(Y)^H$ on itself.  We claim that this is also true when we act via a quantum group.  Specifically, we claim that when $M$ and $N$ are $B_q$-equivariant sheaved on $\dG$, and $M$ is coherent, the sheaf-morphisms $\sHom_{\dG}(M,N)$ admit a natural $B_q$-equivariant structure.  Indeed, the functor $\sHom_{\dG}(M,-)$, with its usual evaluation morphism, provides the right adjoint for the action of $M$ on $\QCoh(\dG/B_q)$.  We describe the equivariant structure on $\sHom_{\dG}(M,N)$ explicitly below.
\par
Let $-|_{B_q}:\QCoh(\dG/B_q)\to \Rep B_q$ denote the global sections functor, which we understand as adjoint to the vector bundle map $E_-:\Rep B_q\to \QCoh(\dG/B_q)$.  The $U_q(\mfk{b})$-actions on $M|_{B_q}$ and $N|_{B_q}$ induce an action of $U_q(\mfk{b})$ on the linear morphisms $\Hom_k(M|_{B_q},N|_{B_q})$, via the usual formula $(x\cdot f)(m)=x_1f(S(x_2)m)$.  The compatibility between the $\O(\dG)$ and $U_q(\mfk{b})$-actions on $M$ and $N$ ensure that the subspace of $\O(\dG)$-linear maps
\[
\Hom_{\O(\dG)}(M|_{B_q},N|_{B_q})\subset \Hom_k(M|_{B_q},N|_{B_q})
\]
forms a $U_q(\mfk{b})$-subrepresentation, and the natural action of $\O(\dG)$ provides the space $\Hom_{\O(\dG)}(M|_{B_q},N|_{B_q})$ with the structure of a $U_q(\mfk{b})$-equivariant $\O(\dG)$-module.
\par

Now, provided $M$ is coherent, the action of $U_q(\mfk{b})$ on the above $\O(\dG)$-linear morphism space integrates to a $B_q$-action, so that $\Hom_{\O(\dG)}(M|_{B_q},N|_{B_q})$ is naturally an object in the category of relative Hopf modules $_{\O(\dG)}\mbf{M}^{\O(B_q)}$.  Indeed, one can observe such integrability by resolving $M$ by vector bundles $E_W\to E_V\to M$.  Since
\[
\Gamma(\dG,\sHom_{\dG}(M,N))=\Hom_{\O(\dG)}(M|_{B_q},N|_{B_q})
\]
we see that $\sHom_{\dG}(M,N)$ is naturally a $B_q$-equivariant sheaf on $\dG$.
\par

We now understand that we have an endofunctor
\[
\sHom_{\dG}(M,-):\QCoh(\dG/B_q)\to \QCoh(\dG/B_q)
\]
provided by usual sheaf-Hom.  The evaluation maps exhibiting $\sHom_{\dG}(M,-)$ as the right adjoint to the functor $-\ot M$ are the expected ones
\[
ev:\sHom_{\dG}(M,N)\ot M\to N,\ \ f\ot m\mapsto f(m).
\]

\subsection{Explicit description of $\sHom_{\dG/B_q}$}

Suppose that $M$ is a coherent sheaf over $\dG/B_q$.  Recall our notation $\su=\uqB$.  Let us define
\[
\sHom_{\dG/\su}(M,-):=\sHom_{\dG}(M,-)^{\uqB}:\QCoh(\dG/B_q)\to \QCoh(\dG)^{\dB}.
\]
By the materials of Section \ref{sect:HomG}, the above expression makes sense.  Left exactness of the invariants functor ensures that the sections of $\sHom_{\dG/\su}$ over opens are as expected,
\[
\begin{array}{rl}
\sHom_{\dG/\su}(M,N)(U)&=\Hom_{\O(U)}(M(U),N(U))^{\uqB}\\
&=\Hom_{\O(U)\ot \uqB}(M(U),N(U)).
\end{array}
\]

\begin{proposition}\label{prop:shom_exp}
Let $M$ be in $\Coh(\dG/B_q)$.  Then, at arbitrary $N$ in $\QCoh(\dG/B_q)$, we have a natural identification
\[
\sHom_{\dG/B_q}(M,N)=\text{\rm descent of the $\dB$-equivariant sheaf }\sHom_{\dG/\su}(M,N).
\]
Under the subsequent natural isomorphism $\pi^\ast\sHom_{\dG/B_q}(M,-)\cong \sHom_{\dG/\su}(M,-)$, the evaluation maps for $\sHom_{\dG/B_q}$ are identified with the morphisms
\[
\sHom_{\dG/\su}(M,N)\ot M\to M,\ \ f\ot m\mapsto f(m).
\]
\end{proposition}

\begin{proof}
By the materials of Section \ref{sect:HomG} we have the adjunction
\[
\Hom_{\dG/B_q}(L\ot M,-)\cong \Hom_{\dG/B_q}(L,\sHom_{\dG}(M,-)).
\]
For any $F$ in $\QCoh(\dG)^{\dB}\subset \QCoh(\dG/B_q)$ we have
\[
\Hom_{\dG/B_q}(F,-)=\Hom_{\QCoh(\dG)^{\dB}}(F,(-)^{\su}).
\]
So for $\msc{F}$ in $\QCoh(\dG/\dB)$ the above two formulae give
\begin{equation}\label{eq:1409}
\begin{array}{rl}
\Hom_{\dG/B_q}(\msc{F}\star M,-)&=\Hom_{\dG/B_q}(\pi^\ast(\msc{F})\ot M,-)\vspace{1mm}\\
&=\Hom_{\dG/B_q}(\pi^\ast(\msc{F}),\sHom_{\dG}(M,-))\vspace{1mm}\\
&=\Hom_{\QCoh(\dG)^{\dB}}(\pi^\ast(\msc{F}),\sHom_{\dG}(M,-)^\su)\vspace{1mm}\\
&\cong \Hom_{\dG/\dB}(\msc{F},\text{desc.\ of\ }\sHom_{\dG/\su}(M,-)).
\end{array}
\end{equation}
The above formula demonstrates the descent of the sheaf $\sHom_{\dG/\su}(M,-)$ as the right adjoint to $-\star M$, and hence identifies $\sHom_{\dG/B_q}(M,-)$ with the descent of $\sHom_{\dG/\su}(M,-)$.  Tracing the identity map through the sequence \eqref{eq:1409} calculates evaluation for $\sHom_{\dG/B_q}$ as the expected morphism
\[
\pi^\ast\sHom_{\dG/B_q}(M,N)\ot M\cong \sHom_{\dG/\su}(M,N)\ot M\to M,\ \ f\ot m\mapsto f(m).
\]
\end{proof}

For arbitrary $M$ in $\QCoh(\dG/B_q)$ we may write $M$ as a colimit $M=\varinjlim_\alpha M_\alpha$ of coherent sheaves to obtain
\[
\sHom_{\dG/B_q}(M,N)=\varprojlim_\alpha \sHom_{\dG/B_q}(M_\alpha,N),
\]
where the final limit is the (somewhat mysterious) limit in the category of quasi-coherent sheaves over $\dG/\dB$.  So, Proposition \ref{prop:shom_exp} provides a complete description of the inner-Hom functor.

\subsection{Relatively projective sheaves}
\label{sect:relative_p}

For any finite-dimensional $B_q$-representation $V$ we consider the functor
\[
\Hom_{\su}(V,-):\QCoh(\dG/B_q)\to \QCoh(\dG)^{\dB},
\]
where specifically $\Hom_{\su}(V,-)=\Hom_k(V,-)^{\uqB}$.  We have the following description of sheaf-Hom for the equivariant vector bundles which refines the description of Proposition \ref{prop:shom_exp}.

\begin{proposition}
For any finite-dimensional $B_q$-representation $V$, there is a natural identification between $\sHom_{\dG/B_q}(E_V,-)$ and the descent of the functor $\Hom_{\su}(V,-)$.
\end{proposition}

\begin{proof}
For any $\msc{F}$ in $\QCoh(\dG/\dB)$ we calculate
\[
\begin{array}{l}
\Hom_{\dG/B_q}(\msc{F}\star E_V,-)=\Hom_{\dG/B_q}(\pi^\ast(\msc{F})\ot_k V,-)\\
\hspace{1cm}=\Hom_{\dG/B_q}(\pi^\ast(\msc{F}),-\ot_k V^\ast)\\
\hspace{1cm}=\Hom_{\dG/B_q}(\pi^\ast(\msc{F}),\Hom_k(V,-))\\
\hspace{1cm}=\Hom_{\dG/B_q}(\pi^\ast(\msc{F}),\Hom_k(V,-)^{\su})=\Hom_{\dG/\dB}(\msc{F},\operatorname{desc.\ of\ }\Hom_\su(V,-)).
\end{array}
\]
Thus, by uniqueness of adjoints, we find that $\sHom_{\dG/B_q}(E_V,-)$ is identified with the descent of the functor $\Hom_{\su}(V,-)$.
\end{proof}

As a corollary we observe a natural class of relative projective sheaves in $\QCoh(\dG/B_q)$.

\begin{corollary}\label{cor:loc_proj}
Suppose that $V$ is a finite-dimensional $B_q$-representation which is projective over $\uqB$.  Then the functor $\sHom_{\dG/B_q}(E_V,-)$ is exact.
\end{corollary}

\begin{proof}
To establish exactness of $\sHom_{\dG/B_q}(E_V,-)$ it suffices to show that the functor $\Hom_{\su}(V,-)$ is exact.  But this follows by projectivity of $V$ over $\su=\uqB$.
\end{proof}

Note that any coherent sheaf $M$ in $\QCoh(\dG/B_q)$ admits a surjection $E\to M$ from an equivariant vector bundle $E=E_V$, with $V$ finite-dimensional and projective over $\uqB$.  So Corollary \ref{cor:loc_proj} implies that the category of coherent sheaves over $\dG/B_q$ has enough relative projectives.

\begin{corollary}\label{cor:enough_proj}
Any coherent sheaf $M$ in $\QCoh(\dG/B_q)$ admits a resolution $\cdots\to E^{-1}\to E^0\to M$ by coherent, relatively projective sheaves $E^i$.
\end{corollary}

\subsection{A proof of Lemma \ref{lem:1090}}
\label{sect:proof1090}

At Lemma \ref{lem:1090} above, we have claimed that the vector bundle $E_V$ is compact in the unbounded derived category $D(\dG/B_q)$ whenever the given $B_q$-representation $V$ is projective over $\uqB$.  We can now prove this result.

\begin{proof}[Proof of Lemma \ref{lem:1090}]
We have the functor $\sHom_{\dG/B_q}(E_V,-)$ which is the descent of the functor $\Hom_{\su}(V,-)$.  This functor is exact, and finiteness of $V$ implies that $\Hom_{\su}(V,-)$ commutes with set indexed sums.  Hence these inner-Homs provide a well-defined operation
\[
\sHom_{\dG/B_q}(E_V,-):D(\dG/B_q)\to D(\dG/\dB)
\]
which commutes with set indexed sums.  We then have
\[
\Ext_{\dG/B_q}^i(E_V,-)=\Ext_{\dG/\dB}^i(\O_{\dG/\dB},\sHom_{\dG/B_q}(E_V,-))
\]
at each integer $i$ by Corollary \ref{cor:372}.  Compactness of $\O_{\dG/\dB}$ over $\dG/\dB$ therefore implies compactness of $E_V$ over $\dG/B_q$.
\end{proof}

\subsection{Composition and tensor structure maps}
\label{sect:comp_ot_exp}

Suppose that $M$ is coherent in $\QCoh(\dG/B_q)$.   From Proposition \ref{prop:shom_exp} we have an identification
\begin{equation}\label{eq:1450}
\pi^\ast\sHom_{\dG/B_q}(M,-)\cong \sHom_{\dG/\su}(M,-)
\end{equation}
under which the evaluation morphisms for $\sHom_{\dG/B_q}$ are identified with the usual evaluation morphisms
\[
\sHom_{\dG/\su}(M,N)\ot M\to M,\ \ f\ot m\mapsto f(m)
\]
for $\sHom_{\dG/\su}$.  It follows that, under the identification \eqref{eq:1450}, the composition and tensor maps
\[
\circ:\sHom_{\dG/B_q}(M,N)\ot\sHom_{\dG/B_q}(L,M)\to \sHom_{\dG/B_q}(L,N)
\]
and
\[
tens:\sHom_{\dG/B_q}(M,N)\ot\sHom_{\dG/B_q}(M',N')\to \sHom_{\dG/B_q}(M\ot M',N\ot N')
\]
pull back to the expected morphisms
\[
\begin{array}{c}
\pi^\ast\circ:\sHom_{\dG/\su}(M,N)\ot\sHom_{\dG/\su}(L,M)\to \sHom_{\dG/\su}(L,N)\\
(f, g)\mapsto f\circ g
\end{array}
\]
and
\[
\begin{array}{c}
\pi^\ast tens:\sHom_{\dG/\su}(M,N)\ot\sHom_{\dG/\su}(M',N')\to \sHom_{\dG/\su}(M\ot M',N\ot N')\\
(f,f')\mapsto f\ot f'.
\end{array}
\]
So, under the identification of Proposition \ref{prop:shom_exp} the composition and tensor structure maps for $\sHom_{\dG/B_q}$ are seen to pull back to the usual structure maps for $\Coh(\dG)^{\uqB}$.

\section{The enhanced derived category $D^{\Enh}(\dG/B_q)$}
\label{sect:Enh_derived}

In Section \ref{sect:shom1} we saw that the sheaf-Hom functor $\sHom_{\dG/B_q}$ provides a monoidal enhancement for $\QCoh(\dG/B_q)$ over the classical flag variety.  At this point we want to provide a corresponding enhancement for the (unbounded) derived category $D(\dG/B_q)$ of quasi-coherent sheaves over the half-quantum flag variety.  Given the information we have already collected, this move to the derived setting is a relatively straightforward process.  We record some of the details here.

\subsection{$\sHom_{\dG/B_q}$ for complexes}

Let $\operatorname{dgQCoh}(\dG/B_q)$ denote the category of quasi-coherent dg sheaves on $\dG/B_q$.  This is the category of quasi-coherent sheaves $M$ with a grading $M=\oplus_{n\in \mbb{Z}} M^n$, and a degree $1$ square zero map $d_M:M\to M$.  Morphisms in this category are the usual morphisms of complexes.  We similarly define $\operatorname{dgQCoh}(\dG/\dB)$.  Consider the forgetful functor $f:\operatorname{dgQCoh}(\dG/B_q)\to \QCoh(\dG/B_q)$.
\par

Let $M$ and $N$ be in $\operatorname{dgQCoh}(\dG/B_q)$.  For an open $U\subset \dG/\dB$ and $M$ in $\QCoh(\dG/B_q)$ write $M|_U=M|_{\pi^{-1}U}$.  A section
\[
s:\O_U\to \sHom_{\dG/B_q}(fM,fN)|_U
\]
over an open $U\subset \dG/\dB$ is said to be homogenous of degree $n$ if, for all $i\in \mbb{Z}$, the restriction $s|_{M^i}:M^i|_U\to N|_U$ has image in $N^{i+n}|_U$.  Here $s:M|_U\to N|_U$ is specifically the composite
\[
M|_U=\O_U\star M|_U\to \sHom_{\dG/B_q}(fM,fN)|_U\star M|_U\overset{ev|_U}\to N|_U.
\]
The collection of degree $n$ maps in $\sHom_{\dG/B_q}(fM,fN)$ form a subsheaf which we denote
\[
\sHom_{\dG/B_q}^n(M,N)\subset \sHom_{\dG/B_q}(fM,fN).
\]
The sum of all homogeneous morphisms also provides a subsheaf
\[
\oplus_{n\in \mbb{Z}}\sHom^n_{\dG/B_q}(M,N)\subset \sHom_{\dG/B_q}(fM,fN).
\]

\begin{definition}
For $M$ and $N$ in $\operatorname{dgQCoh}(\dG/B_q)$ we define the inner-$\Hom$ complex to be the dg sheaf consisting of all homogenous inner morphisms
\[
\sHom_{\dG/B_q}(M,N)=\oplus_{n\in \mbb{Z}}\sHom^n_{\dG/B_q}(M,N)
\]
equipped with the usual differential $d_{\sHom(M,N)}=(d_N)_\ast-(d_M)^\ast$.
\end{definition}

We note that, when $M$ is bounded above and $N$ is bounded below, this complex is just the expected one
\[
\sHom_{\dG/B_q}(M,N)=\oplus_{n\in \mbb{Z}}(\oplus_{i}\sHom_{\dG/B_q}(M^i,N^{i+n}))\ \ \text{along with } d_{\sHom}.
\]
Evaluation for $\sHom_{\dG/B_q}(fM,fN)$ restricts to an evaluation map for the Hom complex $ev:\sHom_{\dG/B_q}(M,N)\star M\to N$.  This evaluation map induces an adjunction
\[
\Hom_{\operatorname{dgQCoh}(\dG/B_q)}(\msc{F}\star M,N)\overset{\cong}\to \Hom_{\operatorname{dgQCoh}(\dG/\dB)}(\msc{F},\sHom_{\dG/B_q}(M,N)).
\]

\subsection{Enhancements for the derived category}
\label{sect:D_Enh}

We consider the central action
\[
\star:D(\dG/\dB)\times D(\dG/B_q)\to D(\dG/B_q)
\]
for the unbounded derived categories of quasi-coherent sheaves, and we have an adjunction
\[
\Hom_{D(\dG/B_q)}(\msc{F}\star M,N)\cong \Hom_{D(\dG/\dB)}(\msc{F},\sRHom_{\dG/B_q}(M,N))
\]
which one deduces abstractly, or from the adjunction at the cochain level described above.  Here the $\star$-action is derived by resolving $M$ by $K$-flat sheaves and $\sRHom_{\dG/B_q}(M,N)=\sHom_{\dG/B_q}(M,I_N)$ for a $K$-injective resolution $N\to I_N$.  Evaluation
\[
ev:\sRHom_{\dG/B_q}(M,N)\ot M\to N
\]
provides composition and tensor structure maps for derived sheaf-Hom $\sRHom_{\dG/B_q}$ so that we obtain a monoidal category
\[
D^{\Enh}(\dG/B_q)=(\text{ob}D(\dG/B_q),\ \sRHom_{\dG/B_q})
\]
which is enriched in the unbounded derived category of quasi-coherent sheaves on the flag variety.  One restricts along the Kempf embedding, Theorem \ref{thm:Kempf}, to obtain a corresponding enriched category $D^{\Enh}(\FK{G}_q)$ for the quantum Frobenius kernel.
\par

We have the derived global sections
\[
\mbb{H}^0(\dG/\dB,-)=\Hom_{D(\dG/\dB)}(\O_{\dG/\dB},-)
\]
and corresponding adjunction isomorphism
\begin{equation}\label{eq:1699}
\Hom_{D(\dG/B_q)}\overset{\cong}\longrightarrow \mbb{H}^0(\dG/\dB,\sRHom_{\dG/B_q}).
\end{equation}
Just as in the proof of Theorem \ref{thm:enhA} (Appendix \ref{sect:A}), one finds that the binatural isomorphism \eqref{eq:1699} realizes $D^{\Enh}(\dG/B_q)$, and $D^{\Enh}(\FK{G}_q)$, as enhancements for their respective derved categories.

\begin{proposition}\label{prop:D_Enh}
The isomorphism \eqref{eq:1699} induces an isomorphism of monoidal categories
\[
\mbb{H}^0(\dG/\dB, D^{\operatorname{Enh}}(\dG/B_q))\overset{\sim}\to D(\dG/B_q),
\]
and similarly $\mbb{H}^0(\dG/\dB, D^{\operatorname{Enh}}(\FK{G}_q))\overset{\sim}\to D(\FK{G}_q)$.
\end{proposition}

\begin{proof}
As in the proof of Theorem \ref{thm:enhA}, one sees that the adjunction map provides a morphism from the constant sheaf
\[
\Hom_{D(\dG/B_q)}(M,N)\to \sRHom_{\dG/B_q}(M,N)
\]
which recovers the standard evaluation maps for $\Hom_{D(\dG/B_q)}$.  This is sufficient to deduce the result.
\end{proof}

\subsection{Coherent dg sheaves}
\label{sect:coherent_sheaves}

In the following section, and also in Part II of the text, we will be interested in objects in $D(\dG/B_q)$ and $D(\FK{G}_q)$ which satisfy certain finiteness conditions.  Abstractly, we consider the subcategories of perfect, or dualizable, objects.  In terms of our specific geometric presentations of these categories, we are interested in coherent dg sheaves.  Let us take a moment to describe these categories clearly.
\par

We consider the derived category $D_{\coh}(\dG/B_q)\subset D(\dG/B_q)$ of coherent, $B_q$-equivariant dg sheaves over $\dG/B_q$.  These are, equivalently, complexes in $D(\dG/B_q)$ with bounded coherent cohomology, or complexes in $D(\dG/B_q)$ which are dualizable with respect to the product $\ot=\ot^{\operatorname{L}}$.  We have the subcategory $D_{\coh}(\FK{G}_q)\subset D_{\coh}(\dG/B_q)$ of coherent $G_q$-equivariant dg sheaves.  These are, again, the dualizable objects in $D(\FK{G}_q)$ and the subcategory $D_{\coh}(\FK{G}_q)$ is equivalent to the full subcategory of objects with bounded coherent cohomology.
\par

In the case of $\FK{G}_q$, all objects in $D_{\coh}(\FK{G}_q)$ are obtainable from the simples in $\FK{G}_q$ via a finite sequence of extensions, so that
\[
D_{\coh}(\FK{G}_q)=\langle S\rangle,\ \ S=\text{ the sum of the simples in }\FK{G}_q.
\]
Here $\langle A\rangle$ denotes the thick subcategory in $D(\FK{G}_q)$ generated by a given object $A$.  All objects in $D_{\coh}(\FK{G}_q)$ have finite length cohomology, and under the equivalence
\[
1^\ast:D(\FK{G}_q)\overset{\sim}\to D(\uqG)(=D(\Rep\uqG))
\]
the subcategory $D_{\coh}(\FK{G}_q)$ is sent to the subcategory $D_{\rm fin}(\uqG)$ in $D(\uqG)$ consisting of bounded complexes of finite-dimensional representations.

\subsection{Enhancements for the coherent derived categories}
\label{sect:perf_Enh}
When we consider the enhancements $D^{\Enh}(\dG/B_q)$ and $D^{\Enh}(\FK{G}_q)$ provided above, we can be much more explicit about the evaluation and tensor maps when we restrict to the subcategories of perfect (i.e.\ dualizable) objects.  Let $D_{\coh}^{\Enh}(\dG/B_q)$ and $D_{\coh}^{\Enh}(\FK{G}_q)$ denote the full (enriched, monoidal) subcategories consisting of coherent dg sheaves in $D^{\Enh}(\dG/B_q)$ and $D^{\Enh}(\FK{G}_q)$ respectively.
\par

For coherent $M$ and bounded $N$ we can adopt any of the explicit models
\[
\begin{array}{rl}
\sRHom_{\dG/B_q}(M,N)&=\sHom_{\dG/B_q}(M,I_N)\ \text{or}\ \ \sHom_{\dG/B_q}(P_M,I_N)\vspace{1mm}\\ &\hspace{1cm}\text{or}\ \ \sHom_{\dG/B_q}(P_M,N)
\end{array}
\]
depending on our needs, where $P_M\to M$ and $N\to I_N$ are resolutions by relative projectives and injectives respectively.  The composition maps for these inner morphisms can then be obtained from composition at the dg level
\[
\circ:\sHom_{\dG/B_q}(M,I_N)\ot\sHom_{\dG/B_q}(P_L,M)\to \sHom_{\dG/B_q}(P_L,I_N),
\]
as can the tensor structure maps
\[
tens:\sHom_{\dG/B_q}(P_M,N)\ot\sHom_{\dG/B_q}(P_{M'},N')\to \sHom_{\dG/B_q}(P_M\ot P_{M'},N\ot N').
\]
Of course, the equivalences of Proposition \ref{prop:D_Enh} realize $D_{\coh}^{\Enh}(\dG/B_q)$ and $D_{\coh}^{\Enh}(\FK{G}_q)$ as enhancements of $D_{\coh}(\dG/B_q)$ and $D_{\coh}(\FK{G}_q)$ respectively.

\section{Fibers over $\dG/\dB$ and the small quantum Borels}
\label{sect:fibers_G/B}

For each geometric point $\lambda:\Spec(K)\to \dG/\dB$, pulling back along the map $\iota_\lambda:\dB_\lambda\to \dG$ provides a monoidal functor
\[
\iota_\lambda^\ast:\QCoh(\dG/B_q)\to \QCoh(\dB_\lambda)^{(B_K)_q}=\msc{B}_\lambda.
\]
When we precompose with the Kempf embedding we recover the standard restriction functor from the quantum Frobenius kernel $\res_\lambda=\iota_\lambda^\ast\circ\Kempf$, as described in Section \ref{sect:borels}.
\par

In this section we show that the derived pullback $\operatorname{L}\iota_\lambda^\ast:D(\dG/B_q)\to D(\msc{B}_\lambda)$ localizes to provide natural maps
\begin{equation}\label{eq:1763}
\operatorname{L}\lambda^\ast\sRHom_{\dG/B_q}(M,N)\to \RHom_{\msc{B}_\lambda}(\operatorname{L}\iota_\lambda^\ast M,\operatorname{L}\iota_\lambda^\ast N)
\end{equation}
which are in fact quasi-isomorphisms whenever $M$ is coherent and $N$ is bounded.  (See Theorem \ref{thm:fiber_l} below.)  We view this result as reflecting a calculation of the fibers
\begin{equation}\label{eq:1621}
\xymatrix@C=18pt{
Vect_{K(\lambda)}\ot_{\QCoh_{\dg}(\dG/\dB)}\QCoh_{\dg}(\dG/B_q)\ar[rr]^(.7){\sim}_(.7){\rm Expected} & & D_{\dg}(\msc{B}_\lambda),
}
\end{equation}
at the level of infinity categories.  One should compare with the calculations of \cite{benzvifrancisnadler10} in the symmetric setting.
\par

The calculation \eqref{eq:1763} is of fundamental importance in our study of support theory for the small quantum group, as it allows us to reduce analyses of support for the small quantum group to corresponding analyses of support for the small quantum Borels.

\subsection{Elaborating on the pullback map}

Consider a geometric point $\lambda:\Spec(K)\to \dG/\dB$.  As discussed above we have the pullback functor
\[
\iota_\lambda^\ast:\QCoh(\dG/B_q)\to \msc{B}_\lambda.
\]
We let $\QCoh(\dG/\dB)$ act on $\msc{B}_\lambda$ via the fiber $\lambda^\ast:\QCoh(\dG/\dB)\to Vect$ and the linear structure on $\msc{B}_\lambda$.  We claim that under this action on $\msc{B}_\lambda$ the map $\iota^\ast_\lambda$ is $\QCoh(\dG/\dB)$-linear.  Indeed, for any sheaf $\msc{F}$ over $\dG/\dB$ we have
\[
\iota_\lambda^\ast(\msc{F}\star-)=\iota_\lambda^\ast(\pi^\ast \msc{F}\ot -)=\iota_\lambda^\ast\pi^\ast(\msc{F})\ot\iota_\lambda^\ast-
\]
and $\iota_\lambda^\ast\pi^\ast$ is naturally isomorphic to $\operatorname{unit}^\ast\lambda^\ast$, since we have the diagram
\[
\xymatrix{
\dB_\lambda\ar[rr]^{\iota_\lambda}\ar[d]_{\operatorname{unit}} & & \dG\ar[d]^\pi\\
\Spec(K)\ar[rr]_\lambda & & \dG/\dB.
}
\]
This identification of pullbacks therefore provides a natural isomorphism
\[
\iota_\lambda^\ast\pi^\ast(\msc{F})\ot\iota_\lambda^\ast-\cong \operatorname{unit}^\ast\lambda^\ast(\msc{F})\ot\iota_\lambda-=\lambda^\ast(\msc{F})\ot_k\iota_\lambda-.
\]
So in total we have the natural isomorphism
\[
\iota_\lambda^\ast(\msc{F}\star-)\cong \lambda^\ast(\msc{F})\ot_k\iota_\lambda^\ast(-)
\]
at all $\msc{F}$ in $\QCoh(\dG/\dB)$ which provides the pullback functor $\iota_\lambda^\ast:\QCoh(\dG/B_q)\to \msc{B}_\lambda$ with the claimed $\QCoh(\dG/\dB)$-linear structure.

\subsection{The natural map $\lambda^\ast\sHom_{\dG/B_q}\to \Hom_{\msc{B}_\lambda}$}
\label{sect:1226}

Consider $M$ and $N$ in $\QCoh(\dG/B_q)$.  By the information of the previous subsection we can apply the pullback functor $\iota^\ast_\lambda$ to the evaluation maps
\[
ev:\sHom_{\dG/B_q}(M,N)\star M\to N
\]
to obtain a map
\begin{equation}\label{eq:1705}
\iota^\ast_\lambda ev:\lambda^\ast\sHom_{\dG/B_q}(M,N)\ot_k \iota_\lambda^\ast M\to\iota_\lambda^\ast N\ \ \text{in the category }\msc{B}_\lambda.
\end{equation}
By adjunction we then obtain a natural map
\begin{equation}\label{eq:1352}
\phi_{M,N}=\phi(\lambda)_{M,N}:\lambda^\ast\sHom_{\dG/B_q}(M,N)\to \Hom_{\msc{B}_\lambda}(\iota^\ast_\lambda M,\iota^\ast_\lambda N)
\end{equation}
which is compatible with the evaluation, and hence compatible with composition and the tensor structure.  This is to say, the maps $\phi_{M,N}$ collectively provide a linear monoidal functor
\[
\phi(\lambda):\lambda^\ast\QCoh^{\Enh}(\dG/B_q)\to \msc{B}_\lambda.
\]

\begin{proposition}\label{prop:1356}
The map $\phi_{M,N}$ of \eqref{eq:1352} is an isomorphism whenever $M$ is coherent and relatively projective.
\end{proposition}

We delay the proof to the end of the section, and focus instead on the implications of Proposition \ref{prop:1356} to our analysis of the derived inner-$\Hom$ functor.

\subsection{The derived map $\operatorname{L}\phi(\lambda):\operatorname{L}\lambda^\ast D^{\Enh}(\dG/B_q)\to D^{\Enh}(\msc{B}_\lambda)$}

Let us begin with a basic lemma.

\begin{lemma}\label{lem:1399}
\begin{enumerate}
\item If $M$ is relatively projective in $\QCoh(\dG/B_q)$ and $N$ is flat, then the sheaf $\sHom_{\dG/B_q}(M,N)$ is flat over $\dG/\dB$.
\item For any sheaf $\msc{F}$ over $\dG/\dB$, complex $P$ of relatively projective sheaves over $\dG/B_q$, and bounded complex $N$ of flat sheaves over $\dG/B_q$, the natural map
\[
\msc{F}\ot^{\operatorname{L}}\sHom_{\dG/B_q}(P,N)\to \msc{F}\ot\sHom_{\dG/B_q}(P,N)
\]
is a quasi-isomorphism.
\item For any closed subscheme $i:Z\to \dG/\dB$, and $P$ and $N$ as in {\rm (2)}, the natural map
\[
\operatorname{L}i^\ast\sHom_{\dG/B_q}(P,N)\to i^\ast\sHom_{\dG/B_q}(P,N)
\]
is a quasi-isomorphism.
\end{enumerate}
\end{lemma}

\begin{proof}
Let us note, before beginning, that flatness of $N$ implies exactness of the operation $-\star N=\pi^\ast(-)\ot N$.  For the first point, consider an exact sequence $0\to \msc{F}'\to \msc{F}\to \msc{F}''\to 0$ of coherent sheaves and the corresponding possibly (non-)exact sequence
\[
0\to \msc{F}'\ot \sHom(M,N)\to \msc{F}\ot \sHom(M,N)\to \msc{F}''\ot \sHom(M,N)\to 0.
\]
By the $\Coh(\dG/\dB)$-linearity of $\sHom(M,-)$, Lemma \ref{lem:sHom-linear}, the above sequence is isomorphic to the sequence
\[
0\to \sHom(M,\msc{F}'\star N)\to \sHom(M,\msc{F}\star N)\to \sHom(M,\msc{F}''\star N)\to 0.
\]
The second sequence is exact by flatness of $N$ and local projectivity of $M$.  So we see that $\sHom(M,N)$ is flat relative to the action of coherent sheaves on $\dG/\dB$.  This is sufficient to see that $\sHom(M,N)$ is flat in $\QCoh(\dG/\dB)$.
\par

For the second point, resolve $\msc{F}$ by a finite complex of flat sheaves $\msc{F}'\to \msc{F}$.  Via a spectral sequence argument, using flatness of $\sHom(P,N)$ in each degree, one sees that the induced map
\[
\msc{F}\ot^{\operatorname{L}}\sHom(P,N)=\msc{F}'\ot\sHom(P,N)\to \msc{F}\ot\sHom(P,N)
\]
is a quasi-isomorphism.  Point (3) follows from point (2) and the identification $\operatorname{L}i^\ast(-)=i_\ast\O_Z\ot^{\operatorname{L}}-$.
\end{proof}

We now consider the derived category $D(\dG/B_q)$.  We have the derived pullback
\[
\operatorname{L}\iota^\ast_\lambda:D(\dG/B_q)\to D(\msc{B}_\lambda)
\]
which still annihilates the $D(\dG/\dB)$-action.  So, as in Section \ref{sect:1226}, we get an induced map on inner-Homs
\[
\operatorname{L}\phi_{M,N}:\operatorname{L}\lambda^\ast\sRHom_{\dG/B_q}(M,N)\to \RHom_{\msc{B}_\lambda}(\operatorname{L}\iota^\ast_\lambda M,\operatorname{L}\iota^\ast_\lambda N)
\]
which is compatible with composition and the tensor structure.  We therefore obtain a monoidal functor
\begin{equation}\label{eq:1470}
\operatorname{L}\phi(\lambda):\operatorname{L}\lambda^\ast D^{\Enh}(\dG/B_q)\to D^{\Enh}(\msc{B}_\lambda),
\end{equation}
where $D^{\Enh}(\msc{B}_\lambda)$ is the linear enhancement of $D(\msc{B}_\lambda)$ implied by the action of $D(Vect_k)$.
\par

Consider coherent $M$ and bounded $N$ in $D(\dG/B_q)$.  (By coherent we mean that $M$ is in $D_{\coh}(\dG/B_q)$.)   If we express $\sRHom_{\dG/B_q}$ by resolving the first coordinate by relatively projective sheaves, which are necessarily flat, and we replace $N$ with a bounded complex of flat sheaves if necessary, the map $\operatorname{L}\phi_{M,N}$ is simply the fiber map
\[
\operatorname{L}\phi_{M,N}=\phi_{M,N}:\lambda^\ast\sHom_{\dG/B_q}(P_M,N)\to \Hom_{\msc{B}_\lambda}(\iota_\lambda^\ast P_M,\iota_\lambda^\ast N)
\]
defined at equation \eqref{eq:1352}, via Lemma \ref{lem:1399} (3).  By Proposition \ref{prop:1356}, the map $\operatorname{L}\phi_{M,N}$ is then seen to be an isomorphism whenever $M$ is coherent and $N$ is bounded.

\begin{theorem}\label{thm:fiber_l}
The map
\[
\operatorname{L}\phi_{M,N}:\operatorname{L}\lambda^\ast\sRHom_{\dG/B_q}(M,N)\to \RHom_{\msc{B}_\lambda}(\operatorname{L}\iota^\ast_\lambda M,\operatorname{L}\iota^\ast_\lambda N)
\]
is a quasi-isomorphism whenever $M$ is coherent and $N$ is bounded.  Consequently, the monoidal functor
\[
\operatorname{L}\phi(\lambda):\operatorname{L}\lambda^\ast D^{\Enh}(\dG/B_q)\to D^{\Enh}(\msc{B}_\lambda)
\]
is fully faithful when restricted to the full subcategory of coherent dg sheaves.  Furthermore, every simple representation in $\msc{B}_\lambda$ is in the image of $\operatorname{L}\phi(\lambda)$, and a projective generator for $\msc{B}_\lambda$ is also in the image of $\operatorname{L}\phi(\lambda)$.
\end{theorem}

\begin{proof}
We have already argued above that $\operatorname{L}\phi(\lambda)$ is fully faithful.  For the claim about the simples and projectives, we just note that $\operatorname{L}\iota^\ast_\lambda E_V=\iota^\ast_\lambda E_V$ for each $B_q$-representation $V$.  Hence all of the vector bundles $\O_{\dB_\lambda}\ot_k V=\iota^\ast_\lambda E_V$ are in the image of $\operatorname{L}\phi(\lambda)$, which we recall is simply $\operatorname{L}\iota^\ast_\lambda$ on objects.
\par

Now, for any $x\in\dB_\lambda$, the tensor equivalence $x:\msc{B}_\lambda\to \msc{B}_1=\FK{B}_q$ sends the bundle $\O_{\dB_\lambda}\ot_kV$ to the vector bundle $\O_{\dB}\ot_k V=E'_V$, where $E'_{-}:\rep B_q\to \FK{B}_q$ is the de-equivariantization map.  Since all of the simples, and a projective generator, for $\FK{B}_q$ lie in the image of de-equivariantization, it follows that all simples and a projective generator for $\msc{B}_\lambda$ lie in the image of $\operatorname{L}\iota_\lambda$, and hence in the image of $\operatorname{L}\phi(\lambda)$.
\end{proof}

One might read the latter point of Theorem \ref{thm:fiber_l}, loosely, as the claim that each fiber $\operatorname{L}\lambda^\ast D^{\Enh}(\dG/B_q)$ is approximately the derived category of sheaves for the associated Borel
\[
\operatorname{L}\lambda^\ast D^{\Enh}(\dG/B_q)\approx D^{\Enh}(\msc{B}_\lambda).
\]

\subsection{Proof of Proposition \ref{prop:1356}}

\begin{proof}[Proof of Proposition \ref{prop:1356}]
Recall the explicit expression of $\sHom_{\dG/B_q}$ in terms of morphisms in $\QCoh(\dG/\su)=\QCoh(\dG)^{\uqB}$, provided by Proposition \ref{prop:shom_exp}.  After pulling back $\pi^\ast:\QCoh(\dG/\dB)\overset{\cong}\to \QCoh(\dG)^{\dB}$ we, equivalently, have a morphism
\begin{equation}\label{eq:1723}
(\O(\dB_\lambda)\ot_{\O(\dG)}\Hom_{\dG/\su}(M,N))^{\dB}\to \Hom_{\msc{B}_\lambda}(\iota^\ast_\lambda M,N)
\end{equation}
which is compatible with evaluation, and we are claiming that this map is an isomorphism.  By Lemma \ref{lem:sHom-linear}, or rather the proof of Lemma \ref{lem:sHom-linear}, and local projectivity of $M$ the map
\[
\O(\dB_\lambda)\ot_{\O(\dG)}\Hom_{\dG/\su}(M,N)\to \Hom_{\dG/\su}(M,\iota^\ast_\lambda N),\ \ f\ot \xi\mapsto f\cdot \operatorname{red}_\lambda \xi
\]
is an isomorphism, where $\operatorname{red}_\lambda:N\to \iota^\ast_\lambda N$ is the reduction map.  (Here we are viewing sheaves on $\dB_\lambda$ as sheaves on $\dG$ via pushforward.)  Furthermore $\Hom_{\dG/\su}(M,\iota^\ast_\lambda N)=\Hom_{\dG/\su}(\iota^\ast_\lambda M,\iota^\ast_\lambda N)$ and when we take invariants we have
\[
\Hom_{\dG/\su}(\iota^\ast_\lambda M,\iota^\ast_\lambda N)^{\dB}=\Hom_{\Coh(\dG/B_q)}(\iota^\ast_\lambda M,\iota^\ast_\lambda N).
\]
We therefore have an isomorphism
\begin{equation}\label{eq:1731}
(\O(\dB_\lambda)\ot_{\O(\dG)}\Hom_{\dG/\su}(M,N))^{\dB}\overset{\cong}\to \Hom_{\dG/B_q}(\iota^\ast_\lambda M,\iota^\ast_\lambda N)
\end{equation}
given by $f\ot \xi\mapsto f\operatorname{red}_\lambda(\xi)$, and under this isomorphism the reduction of evaluation for $\Hom_{\dG/\su}(M,N)$ appears as the expected evaluation map
\[
\Hom_{\dG/B_q}(\iota^\ast_\lambda M,\iota^\ast_\lambda N)\ot_k \iota^\ast_\lambda M\to \iota^\ast_\lambda N,\ \ \xi\ot m\mapsto \xi(m).
\]
\par

Now, under the identification \eqref{eq:1731} the map \eqref{eq:1723} now appears as
\begin{equation}\label{eq:1742}
\Hom_{\dG/B_q}(\iota^\ast_\lambda M,\iota^\ast_\lambda N)\to \Hom_{\msc{B}_\lambda}(\iota_\lambda^\ast M,\iota^\ast_\lambda N),\ \ \xi\mapsto \xi.
\end{equation}
One simply observes that these morphism spaces are literally equal, since $\msc{B}_\lambda=\Coh(\dB_\lambda)^{B_q}$ and $\Coh(\dG/B_q)=\Coh(\dG)^{B_q}$, and notes the map \eqref{eq:1742} is the identity to see that \eqref{eq:1742} is an isomorphism.  It follows that our original map \eqref{eq:1723} is an isomorphism.
\end{proof}


\part{Support Theory and the Springer Resolution}

In Part II of this document we explain how the monoidal enhancement $D^{\Enh}(\FK{G}_q)$ from Part I can be used to produce a $\tN$-valued support theory for the small quantum group.  Having accepted the appearance of the flag variety $\dG/\dB$ from Part I, the Springer resolution then appears via the global (tensor) action of the endomorphism algebra $\sRHom_{\FK{G}_q}(\1,\1)$ on $D^{\Enh}(\FK{G}_q)$.  The cohomology of this algebra is shown to be the pushforward $p_\ast \O_{\tN}$ of the structure sheaf for the Springer resolution along the (affine) projection $p:\tN\to \dG/\dB$ (see Theorem \ref{thm:enh_GK} below).
\par

We show that our $\tN$-valued support for the small quantum group induces a surjection from the projectivized Springer resolution $\mbb{P}(\tN)$ to the Balmer spectrum for $\FK{G}_q$, at least in type $A$.  This surjection is a generic homeomorphism, so that the Springer resolution essentially resolves the Balmer spectrum in this case.
\par 

Throughout this portion of the paper we focus on the category of sheaves for the quantum Frobenius kernel $\FK{G}_q$, rather than the half-quantum flag variety.  When necessary, we translate results about the half-quantum flag variety to the quantum Frobenius kernel by restricting along the Kempf embedding (Theorem \ref{sect:Kempf}).

\section{The algebra $\msc{R}(G_q)$ and the Springer resolution}
\label{sect:RGq}

Recall our monoidal enhancement $D^{\Enh}(\FK{G}_q)$ with its corresponding morphism sheaves $\sRHom_{\FK{G}_q}$ (see Section \ref{sect:D_Enh}).  We have the algebra
\[
\msc{R}(G_q):=\sRHom_{\FK{G}_q}(\1,\1)
\]
in the derived category $D(\dG/\dB)$ which acts on the left and right of all morphisms $\sRHom_{\FK{G}_q}(M,N)$ via the tensor structure.
\par

We provide below some preliminary comments on the algebra $\msc{R}(G_q)$ which orient our support theoretic discussions.  The first point of the section is to give a basic description of the algebra $\msc{R}(G_q)$ and its binatural global action on $D^{\Enh}(\FK{G}_q)$.  The second point is to recognize that the cohomology of $\msc{R}(G_q)$ recovers the structure sheaf for the Springer resolution,
\[
H^\ast(\msc{R}(G_q))=p_\ast\O_{\tN},\ \ \text{where }p:\tN\to \dG/\dB\text{ is the projection}.
\]
The cohomology groups $\sExt_{\FK{G}_q}(M,N)$ can then be understood as $\mbb{G}_m$-equivariant sheaves on the Springer resolution, at arbitrary $M$ and $N$.

\subsection{Commutativity of $\msc{R}(G_q)$ and centrality of the global action}

For any $M$ and $N$ in $D(\FK{G}_q)$ we claim that the diagram
\[
\xymatrix{
\msc{R}(G_q)\ot\sRHom(M,N)\ar[rr]^{\rm act}\ar[d]_{\rm symm} & & \sRHom(M,N)\\
 \sRHom(M,N)\ot \msc{R}(G_q)\ar[urr]_{\rm act} & 
}
\]
commutes.  This is equivalent to the claim that the adjoint diagram
\begin{equation}\label{eq:diagram}
\xymatrix{
\msc{R}(G_q)\ot\sRHom(M,N)\star M\ar[rr]\ar[d]_{\rm symm\ot 1} & & N\\
\sRHom(M,N)\ot \msc{R}(G_q)\star M\ar[urr] & 
}
\end{equation}
commutes.

\begin{proposition}\label{prop:1660}
The algebra $\msc{R}(G_q)$ is commutative in $D(\dG/\dB)$, and the tensor actions give each $\sRHom_{\FK{G}_q}(M,N)$ the structure of a symmetric $\msc{R}(G_q)$-bimodule in $D(\dG/\dB)$.
\end{proposition}

This result is not completely formal, and is not valid when $M$ and $N$ are replaced with objects in the ambient category $D(\dG/B_q)$, as far as we can tell.  (The problem is that the evaluation maps $\sHom(M,N)\star M\to M$ are not central morphisms in $D(\dG/B_q)$, even when $M$ and $N$ are central objects.)  We provide a proof in Section \ref{sect:1660} of the appendix.

\subsection{The Springer resolution}

Recall that the Springer resolution $\tN$ is the affine bundle $\tN=\dG\times_{\dB}\mfk{n}$ over the flag variety, where $\mathfrak{n}$ is the (positive) nilpotent subalgebra $\mfk{n}\subset \mfk{g}$.  (Here $\mfk{g}$ is the Lie algebra for $\dG$, which is the same as the Lie algebra for $G$.)  Equivalently, the Springer resolution is obtained as the relative spectrum \cite[Ex.\ 5.17]{hartshorne13} of the descent of the $\dB$-equivariant algebra $\O_{\dG}\ot \Sym(\mfk{n}^\ast)$ over $\dG$,
\begin{equation}\label{eq:1932}
\tN=\dG\times_{\dB}\mfk{n}=\Spec_{\dG/\dB}\left(\text{descent\ of }\O_{\dG}\ot_k \Sym(\mfk{n}^\ast)\right)=\Spec_{\dG/\dB}(\Sym(\msc{E})).
\end{equation}
Here $\msc{E}$ is the vector bundle on $\dG/\dB$ associated to the $\dB$-representation $\mfk{n}^\ast$.
\par

From this construction of $\tN$ as the relative spectrum of $\Sym(\msc{E})$ we see that pushing forward along the (affine) structure map $p:\tN\to \dG/\dB$ provides an identification $p_\ast\O_{\tN}=\Sym(\msc{E})$ and also an equivalence of monoidal categories
\[
p_\ast:\QCoh(\tN)\overset{\sim}\longrightarrow \QCoh(p_\ast\O_{\tN}).
\]
To be clear, the latter category is the category of modules over the algebra object $p_\ast\O_{\tN}$ in $\QCoh(\dG/\dB)$, and the monoidal product is as expected $\ot_{p_\ast\O_{\tN}}$.
\par

It is well-known that the Springer resolution $\tN$ is identified with the moduli space of choices of a Borel in $\mfk{g}$, and a nilpotent element in the given Borel,
\begin{equation}\label{eq:moduli}
\mcl{M}oduli=\{(\mfk{b}_\lambda,x):\mfk{b}_\lambda\subset\mfk{g}\ \text{a Borel}\ x\in \mfk{b}_\lambda\ \text{is nilpotent}\}
\end{equation}
\cite[\S 3.2]{chrissginzburg09}.  This moduli space sits in the product $\mcl{M}oduli\subset \dG/\dB\times \mcl{N}$, where $\mcl{N}$ is the nilpotent cone in $\mfk{g}$, and we have an explicit isomorphism between $\tN$ and $\mcl{M}oduli$ given by the $\dG$-actions on the two factors in this product,
\[
\tN=\dG\times_{\dB} \mfk{n}\overset{\cong}\to \mcl{M}oduli,\ \ (g,x)\mapsto (\operatorname{Ad}_g(\mfk{b}),\operatorname{Ad}_g(x)).
\]
In the above formula $\mfk{b}$ is the positive Borel in $\mfk{g}$.  We identify $\tN$ with this moduli space when convenient, via the above isomorphism.

We consider $\tN$ as a conical variety over $\dG/\dB$ by taking the generating bundle $\msc{E}$ to be in (cohomological) degree $2$.  In terms of the moduli description given above, this conical structure corresponds to a $\mbb{G}_m$-action defined by the squared scaling $c\cdot (\mfk{b}_\lambda, x)=(\mfk{b}_\lambda, c^2\cdot x)$.

\begin{remark}
For certain applications $\tN$ might be viewed, more fundamentally, as a dg scheme over $\dG/\dB$ which is generated in degree $2$ and has vanishing differential.  Compare with \cite{arkhipovbezrukavnikovginzburg04} and Part III of this text.
\end{remark}

\subsection{The moment map}

Given our identification of the Springer resolution $\tN$ with the moduli of pairs \eqref{eq:moduli}, we have two projections $\mfk{b}_\lambda\leftarrow (\mfk{b}_\lambda,x)\to x$ which define maps $p:\tN\to \dG/\dB$ and $\kappa:\tN\to \mcl{N}$.  The map $p$ is affine, and simply recovers the structural map $\tN=\Spec_{\dG/\dB}(\Sym(\msc{E}))\to \dG/\dB$.  The map $\kappa$ provides an identification of the affinization of $\tN$ with the nilpotent cone,
\[
\bar{\kappa}:\tN_{\rm aff}\overset{\cong}\longrightarrow \mcl{N}.
\]
The map $\kappa$ is called the \emph{moment map}.  The moment map is a proper birational equivalence, and so realizes the Springer resolution as a resolution of singularities for the nilpotent cone \cite[Theorem 10.3.8]{hottatanisaki07}.

\subsection{A calculation of cohomology}
\label{sect:GK}

The following result is deduced immediately from results of Ginzburg and Kumar \cite{ginzburgkumar93}.

\begin{theorem}\label{thm:enh_GK}
There is a canonical identification
\[
H^\ast(\msc{R}(G_q))=p_\ast\O_{\tN},
\]
as sheaves of graded algebras over $\dG/\dB$.
\end{theorem}

\begin{proof}
Recall that $\QCoh(\dG/u)=\QCoh(\dG)^{\uqB}$, by definition.  We have
\[
\sRHom_{\dG/u}(\1,\1)=\O_{\dG}\ot_k \RHom_{\uqB}(k,k),
\]
where $\RHom_{\uqB}(k,k)$ is given its natural $\dB$-action as inner morphisms for the $\Rep(\dB)$-action on $\Rep(B_q)$.  Hence, by the calculation of $\Ext_{\uqB}(k,k)$ provided in \cite[Lemma 2.6]{ginzburgkumar93}, we have
\[
H^\ast(\sRHom_{\dG/u}(\1,\1))=\O_{\dG}\ot_k \Ext_{\uqB}(k,k)=\O_{\dG}\ot_k \Sym(\mfk{n}^\ast).
\]
One therefore applies Proposition \ref{prop:shom_exp} to obtain
\[
H^\ast(\msc{R}(G_q))=H^\ast(\sRHom_{\dG/B_q}(\1,\1))=\text{descent of }\O_{\dG}\ot_k \Sym(\mfk{n}^\ast)=p_\ast\O_{\tN}.
\]
\end{proof}

As noted in the original work \cite{ginzburgkumar93}, the higher global section $H^{>0}(\dG/\dB,p_\ast\O_{\tN})$ vanish so that Theorem \ref{thm:enh_GK} implies a computation of extensions for $\FK{G}_q$.  In the following statement $\O(\mcl{N})$ is considered as a graded algebra with generators in degree $2$.

\begin{corollary}[\cite{ginzburgkumar93}]\label{cor:GK}
There is an identification of graded algebras $\Ext_{\FK{G}_q}(\1,\1)=\O(\mcl{N})$.
\end{corollary}


\section{$\tN$-support, and an approximate derived category}
\label{sect:AH!}

In this section we define $\tN$-support for the category $D_{\coh}(\FK{G}_q)$ of coherent dg sheaves for the quantum Frobenius kernel.
\par

If we ignore some subtle points, one can think of the situation as follows:  For the category $D^{\Enh}(\FK{G}_q)$, we have the algebra of extensions $\msc{R}(G_q)$ which acts centrally on all objects in $D^{\Enh}(\FK{G}_q)$ via the tensor structure.  Following a standard philosophy of support theory, articulated in \cite{bensoniyengarkrause08} for example, we understand that this central action of $\msc{R}(G_q)$ then defines a support theory for objects in $D^{\Enh}(\FK{G}_q)=\operatorname{obj}D(\FK{G}_q)$ which takes values in a spectrum $\Spec_{\dG/\dB}(\msc{R}(G_q))$.  This spectrum is approximately the Springer resolution, by Theorem \ref{thm:enh_GK}.
\par

Below, we restrict ourselves to coherent dg sheaves and work with the cohomology of $D^{\Enh}(\FK{G}_q)$, rather than $D^{\Enh}(\FK{G}_q)$ itself.  At the level of cohomology, we have precisely $\Spec_{\dG/\dB}(H^\ast(\msc{R}(G_q)))=\tN$, and thus gain immediate access to the Springer resolution.

\subsection{Preliminary remarks on support notation}

For a Noetherian scheme $Y$, and $\msc{F}$ in $\QCoh(Y)$, we denote the standard sheaf-theoretic support of $\msc{F}$ as
\[
\Supp_Y(\msc{F})=\{y\in Y:\msc{F}_y\neq 0\}.
\]
When $\msc{F}$ is coherent we usually employ the (equivalent) definition via vanishing of fibers.  More generally, if $p:X\to Y$ is an affine map, and $\msc{G}$ is a sheaf of $p_\ast\O_{X}$-modules in $\QCoh(Y)$, then we let $\Supp_{X}(\msc{G})$ denote the support of the unique-up-to-isomorphism quasi-coherent sheaf on $X$ which pushes forward to $\msc{G}$.  In the coherent setting this support can again be calculated via the vanishing of fibers $K(x)\ot_{p_\ast\O_{X}}\msc{G}$ along algebra maps $x: p_\ast\O_{X}\to K(x)$.
\par

In addition to sheaf-theoretic supports we are also interested in support \emph{theories} for various triangulated categories.  (One can see Section \ref{sect:supp_nons} for a thorough discussion of support theories in the abstract.)  Such support theories take values in a scheme $Y$, and will be written in a lower case
\[
\supp_Y\ \ \text{or}\ \ \supp^{\operatorname{flav}}_Y.
\]
The optional label ``$\operatorname{flav}$" indicates the particular type of support theory under consideration (cohomological, universal, hypersurface, etc.).  This upper case/lower case distinction for support is employed throughout the text, without exception.

\subsection{Defining $\tN$-support}
\label{sect:supp_tN}

Let $M$ be an object in $D(\FK{G}_q)$ and consider the sheafy derived endomorphisms $\sRHom_{\FK{G}_q}(M,M)$.  Via the monoidal structure on $D^{\Enh}(\FK{G}_q)$ these endomorphisms form a sheaf of modules over $\msc{R}(G_q)=\sRHom_{\FK{G}_q}(\1,\1)$ in the derived category of sheaves over $\dG/\dB$.  Hence we take cohomology to obtain a sheaf
\[
\sExt_{\FK{G}_q}(M,M):=H^\ast(\sRHom_{\FK{G}_q}(M,M))
\]
of $H^\ast(\msc{R}(G_q))=p_\ast\O_{\tN}$-modules (see Theorem \ref{thm:enh_GK}).  We can then consider the support
\begin{equation}\label{eq:2061}
\Supp_{\tN}\sExt_{\FK{G}_q}(M,M)\subset \Spec_{\dG/\dB}(p_\ast\O_{\tN})=\tN.
\end{equation}
This support is a conical subspace in $\tN$.  The following lemma is deduced as an immediate consequence of Proposition \ref{prop:coherent} below.

\begin{lemma}\label{lem:2066}
For $M$ in $D_{\coh}(\FK{G}_q)$, the support \eqref{eq:2061} is closed in $\tN$.
\end{lemma}

If we take this point for granted for the moment, we obtain a reasonable theory of support for coherent dg sheaves.

\begin{definition}
For any $M$ in $D_{\coh}(\FK{G}_q)$ we define the $\tN$-support $\supp_{\tN}(M)$ as
\[
\supp_{\tN}(M):=\mbb{P}\left(\Supp_{\tN}\sExt_{\FK{G}_q}(M,M)\right)\ \subset\ \mbb{P}(\tN).
\]
\end{definition}

By Lemma \ref{lem:2066}, this support provides a map
\[
\supp_{\tN}:\text{obj}D_{\coh}(\FK{G}_q)\to \{\text{closed subsets in }\mbb{P}(\tN)\}
\]
We discuss the basic properties of $\tN$-support below, after establishing the appropriate global framework for our discussions.

\subsection{The $\QCoh(\tN)^{\mbb{G}_m}$-enriched category $\msc{D}(\FK{G}_q)$}
\label{sect:RGq_act}

The cohomology sheaves $\sExt_{\FK{G}_q}(M,N)$, introduced above, exist within an enriched monoidal category obtained as the cohomology the category $D^{\Enh}(\FK{G}_q)$.  We take a moment to describe this enriched category.
\par

Recall our algebra $\msc{R}(G_q)=\sRHom_{\FK{G}_q}(\1,\1)$ from Section \ref{sect:RGq}, and take $\msc{R}_{MN}=\sRHom_{\FK{G}_q}(M,N)$.  The associativity of the monoidal structure on $D^{\Enh}(\FK{G}_q)$ ensures that the product maps
\[
\msc{R}_{MN}\ot\msc{R}_{M'N'}\to \msc{R}_{M\ot M',N\ot N'}
\]
are $\msc{R}(G_q)$-bilinear, in the sense that we have a diagram
\[
\msc{R}_{MN}\ot\msc{R}(G_q)\ot\msc{R}_{M'N'}\rightrightarrows\msc{R}_{MN}\ot\msc{R}_{M'N'}\overset{tens}\to \msc{R}_{M\ot M',N\ot N'}.
\]
By a similar reasoning, composition is $\msc{R}(G_q)$-bilinear in the sense that we have a diagram
\[
\msc{R}_{MN}\ot\msc{R}(G_q)\ot\msc{R}_{LM}\rightrightarrows\msc{R}_{MN}\ot\msc{R}_{LM}\overset{\circ}\to \msc{R}_{LN}.
\]
Hence we apply the lax monoidal functor $H^\ast:D(\dG/\dB)\to \QCoh(\dG/\dB)^{\mbb{G}_m}$ and observe the identification $\tN=\Spec_{\dG/\dB}(H^\ast(\msc{R}(G_q)))$ of Theorem \ref{thm:enh_GK} above to find the following.

\begin{proposition}\label{prop:1749}
The cohomology $H^\ast(D^{\Enh}(\FK{G}_q))$ of the $D(\dG/\dB)$-enriched category $D^{\Enh}(\FK{G}_q)$ is naturally enriched in the symmetric monoidal category of $\mbb{G}_m$-equivariant sheaves on the Springer resolution $\tN$.
\end{proposition}

To be more precise, Theorem \ref{thm:enh_GK} tells us that the cohomology $H^\ast(D^{\Enh}(\FK{G}_q))$ is naturally enriched in the category of graded, quasi-coherent $p_\ast\O_{\tN}$-modules over $\dG/\dB$.  Via the equivalence $p_\ast:\QCoh(\tN)^{\mbb{G}_m}\to \QCoh(p_\ast\O_{\tN})^{\operatorname{graded}}$ we then view $H^\ast(D^{\Enh}(\FK{G}_q))$ as enriched in the category of $\mbb{G}_m$-equivariant sheaves over the Springer resolution, by an abuse of notation.

\begin{definition}
We let $\msc{D}(\FK{G}_q)$ denote the $\QCoh(\tN)^{\mbb{G}_m}$-enriched monoidal category
\[
\msc{D}(\FK{G}_q):=H^\ast(D^{\Enh}(\FK{G}_q)).
\]
We take
\[
\sExt_{\FK{G}_q}(M,N)^{\Sp}=\left\{\begin{array}{c}
\text{the quasi-coherent $\mbb{G}_m$-equivariant sheaf}\\
\text{on $\tN$ associated to the graded $p_\ast\O_{\tN}$-module}\\
\sExt_{\FK{G}_q}(M,N)
\end{array}
\right\}.
\]
So the morphisms in $\msc{D}(\FK{G}_q)$ are the sheaves $\sExt_{\FK{G}_q}(M,N)^{\Sp}$.
\end{definition}

As we have just explained, we have an identification of $p_\ast\O_{\tN}$-modules
\[
p_\ast\sExt_{\FK{G}_q}(M,N)^{\Sp}=\sExt_{\FK{G}_q}(M,N)
\]
at arbitrary $M$ and $N$ in $D(\FK{G}_q)$.  We note that the category $\msc{D}(\FK{G}_q)$ is not an enhancement of the derived category $D(\FK{G}_q)$, but is some approximation of $D(\FK{G}_q)$.  Indeed, we have the local-to-global spectral sequence (Corollary \ref{cor:372})
\[
\begin{array}{l}
E_2^{i,j}=\\
\operatorname{R}\Gamma(\tN,\sExt_{\FK{G}_q}(M,N)^{\Sp})^{i,j}=\operatorname{R}\Gamma^i(\dG/\dB,\sExt^j_{\FK{G}_q}(M,N))\ \Rightarrow\ \Ext^{i+j}_{\FK{G}_q}(M,N)
\end{array}
\]
which calculates the extensions $\Ext_{\FK{G}_q}(M,N)$ as an $\O(\mcl{N})$-module.

\begin{remark}
To say that the cohomology $\msc{D}(\FK{G}_q)$ of the category $D^{\Enh}(\FK{G}_q)$ is monoidally enriched in $\QCoh(\tN)^{\mbb{G}_m}$ simply collects a number of structural results about the cohomology of the derived sheaf-morphisms $\sRHom_{\FK{G}_q}$ into a single statement.  Having accepted these structural results, one can easily avoid direct references to the ambient category $\msc{D}(\FK{G}_q)$.  For this reason, $\msc{D}(\FK{G}_q)$ makes very few appearances in later sections of the text.
\end{remark}

\subsection{A coherence result}

Before stating our coherence result we note that the functors
\[
\begin{array}{c}
\sRHom_{\FK{G}_q}(M,-):D(\FK{G}_q)\to D(\dG/\dB)\vspace{2mm}\\
\text{and}\ \ \sRHom_{\FK{G}_q}(-,M):D(\FK{G}_q)^{op}\to D(\dG/\dB)
\end{array}
\]
are exact.  One can see this directly, or by referring to the abstract result \cite[Lemma 5.3.6]{neeman01}.  We recall also, from Section \ref{sect:coherent_sheaves}, that the category $D_{\coh}(\FK{G}_q)$ of coherent dg sheaves for the quantum Frobenius kernel is equal to the thick subcategory in $D(\FK{G}_q)$ generated by the simple (coherent) objects in $\FK{G}_q$.  The following result is proved by a standard reduction to the simples.

\begin{proposition}\label{prop:coherent}
For $M$ and $N$ in $D_{\coh}(\FK{G}_q)$ the sheaf $\sExt_{\FK{G}_q}(M,N)^{\Sp}$ is coherent over $\tN$.
\end{proposition}

\begin{proof}
For any $L$ in $D(\FK{G}_q)$ the functors $\sRHom_{\FK{G}_q}(-,L)$ and $\sRHom_{\FK{G}_q}(L,-)$ are exact, and hence provide long exact sequences on cohomology when applied to exact triangles.  For a triangle $M\to N\to M'$ in $D(\FK{G}_q)$ the corresponding long exact sequence on cohomology provides exact sequences
\[
\sExt(M',L)\to \sExt(N,L)\to \sExt(M,L)\ \ \text{and}\ \ \sExt(L,M)\to \sExt(L,N)\to \sExt(L,M')
\]
of $p_\ast\O_{\tN}$-modules.  Via these exact sequences we see that it suffices to prove coherence of the sheaf $\sExt_{\FK{G}_q}(M,N)^{\Sp}$ for simple $M$ and $N$ in $\FK{G}_q$.  By Proposition \ref{prop:andersen} it then suffices to establish coherence of $\sExt_{\FK{G}_q}(E_V,E_W)^{\Sp}$ for equivariant vector bundles associated to finite-dimensional $G_q$-representations $V$ and $W$.  In this case we have
\[
\sExt_{\FK{G}_q}(E_V,E_W)=\text{descent of }\O_{\dG}\ot_k\Ext_{\uqB}(V,W)
\]
as a $p_\ast\O_{\tN}$-module, by Proposition \ref{prop:shom_exp}.  Since $\Ext_{\uqB}(V,W)$ is finitely-generated over $\Ext_{\uqB}(k,k)=\Sym(\mfk{n}^\ast)$ \cite[Theorem 2.5]{ginzburgkumar93}, it follows that the descended sheaf $\sExt_{\FK{G}_q}(E_V,E_W)$ is generated by a coherent subsheaf $\msc{F}\subset \sExt_{\FK{G}_q}(E_V,E_W)$ over $\dG/\dB$, as a $p_\ast\O_{\tN}$-module, and hence the associated sheaf $\sExt_{\FK{G}_q}(E_V,E_W)^{\Sp}$ over $\tN$ is coherent.
\end{proof}

\subsection{Properties of $\tN$-support}
\label{sect:RGq_tN}

As was mentioned above, the coherence result of Proposition \ref{prop:coherent} implies that the support
\[
\Supp_{\tN}\sExt_{\FK{G}_q}(M,M)\ \subset\ \tN
\]
of a coherent dg sheaf $M$ is a closed conical subvariety in the Springer resolution.  Subsequently, the $\tN$-supports $\supp_{\tN}(M)$ are closed subspaces in $\mbb{P}(\tN)$.  This point was already recorded at Lemma \ref{lem:2066} above.
\par

The following standard properties of $\tN$-support are proved as in \cite[\S 5.7]{benson91}, and simply follow from naturality of the $\msc{R}(G_q)$-action on the sheaves $\sRHom_{\FK{G}_q}(M,N)$ and the fact that $\sRHom_{\FK{G}_q}$ is exact in each argument.  We let $\Sigma$ denotes the shift operation on the derived category $D(\FK{G}_q)$.

\begin{lemma}\label{lem:triangle}
The assignment $M\mapsto \supp_{\tN}(M)$ is a triangulated support theory for $D_{\coh}(\FK{G}_q)$.  That is to say, $\tN$-support has the following properties:\vspace{1mm}
\begin{itemize}
\item $\supp_{\tN}(\1)=\mbb{P}(\tN)$.
\item $\supp_{\tN}(M)=\supp_{\tN}(\Sigma M)$ for all $M$ in $D_{\coh}(\FK{G}_q)$, and $\supp_{\tN}(0)=\emptyset$.\vspace{2mm}
\item $\supp_{\tN}(M\oplus N)=\supp_{\tN}(M)\cup\supp_{\tN}(N)$.\vspace{2mm}
\item For any triangle $M\to N\to M'$ in $D_{\coh}(\FK{G}_q)$ we have
\[
\supp_{\tN}(N)\subset(\supp_{\tN}(M)\cup\supp_{\tN}(M')).
\]
\item For the sum of the simples $S$ in $\FK{G}_q$ one has
\[
\supp_{\tN}(M)=\mbb{P}\left(\Supp_{\tN}\sExt_{\FK{G}_q}(S,M)\right).
\]
\end{itemize}
\end{lemma}


\section{Sheaf-extensions and projectivity via the Borels}
\label{sect:perf_shf}

In this section we analyze the behaviors of sheaf-Ext for coherent dg sheaves in $D(\FK{G}_q)$.  We prove that $\tN$-support for the coherent derived category vanishes precisely on the subcategory of bounded complexes of projectives in $D_{\coh}(\FK{G}_q)$.  See Proposition \ref{prop:projectives} below.
\par

In the quasi-coherent setting, we use sheaf-extensions directly to provide a strong projectivity test for objects in $D(\FK{G}_q)$ via the quantum Borels.  Our primary result for the section is the following.

\begin{theorem}[Projectivity test]\label{thm:projectivity}
For a bounded complex $M$ of quasi-coherent sheaves in $D(\FK{G}_q)$ the following are equivalent:
\begin{itemize}
\item[(a)] $M$ is isomorphic to a bounded complex of projectives in $D(\FK{G}_q)$.
\item[(b)] At each geometric point $\lambda:\Spec(K)\to \dG/\dB$, the restriction $\res_\lambda M$ is isomorphic to a bounded complex of projectives in $D(\msc{B}_\lambda)$.
\end{itemize}
\end{theorem}

We note that (a) and (b) are equivalent to the vanishing of $M$ and $\res_\lambda M$ in their respective stable categories (see Section \ref{sect:stable}).  So one can view Theorem \ref{thm:projectivity} as a vanishing result for objects in the stable category $\Stab(\FK{G}_q)$.  Theorem \ref{thm:projectivity} plays a fundamental role in our analysis of thick ideals and the Balmer spectrum for $D_{\coh}(\FK{G}_q)$, provided in Sections \ref{sect:spec}--\ref{sect:spec2} below.

\subsection{Sheaf-$\Ext$ and projectivity}

We have the following general result about the behaviors of $\sExt_{\FK{G}_q}$ at projectives.

\begin{lemma}\label{lem:sExt_proj}
Let $S$ denote the sum of the simple objects in $\FK{G}_q$.  Then an object $N$ in $D^+(\FK{G}_q)$ (resp.\ $D_{\coh}(\FK{G}_q)$) is isomorphic to a bounded complex of projectives (resp.\ coherent projectives) if and only if $\sExt^{\gg 0}_{\FK{G}_q}(S, N)=0$.
\end{lemma}

\begin{proof}
Suppose that $N$ is projective (equivalently injective) in $\FK{G}_q$.  For a projective generator $W$ from $\rep G_q$, we therefore have a split surjection $\oplus_{i\in I}E_W\to N$ from some large sum of the associated free module.  Since
\[
\oplus_{i\in I}\sExt_{\FK{G}_q}(S,E_W)=\sExt_{\FK{G}_q}(S,\oplus_{i\in I}E_W)
\]
we see that vanishing of the extension $\sExt^{>0}_{\FK{G}_q}(S,E_W)$ at the sum of the simples implies vanishing of the extensions $\sExt^{>0}_{\FK{G}_q}(S,N)$.  But now, by Proposition \ref{prop:andersen}, there is a $G_q$-representation $V$ for which $S$ is a summand of $E_V$.  Hence vanishing of $\sExt^{>0}_{\FK{G}_q}(E_V,E_W)$ at such $V$ and $W$ implies vanishing of $\sExt^{>0}_{\FK{G}_q}(S,N)$.
\par

Now, having established the above information, we compute
\[
\begin{array}{rl}
\pi^\ast\sExt_{\FK{G}_q}(E_V,E_W)&=\sExt_{\dG/\su}(E_V,E_W)\vspace{2mm}\\
&=\O_{\dG}\ot_k\Ext_{\uqB}(V,W).
\end{array}
\]
Since $W$ is projective/injective in $\Rep G_q$, its restriction is injective over $\uqG$ and thus injective over $\uqB$ as well.  So we see $\Ext_{\uqB}^{>0}(V,W)=0$ and hence that $\sExt_{\FK{G}_q}^{>0}(S,N)=0$ at any projective $N$.  It follows that $\sExt_{\FK{G}_q}^{\gg 0}(S,N)=0$ at any bounded complex of projectives $N$.
\par

For the converse, suppose that $\sExt^{>m}_{\FK{G}_q}(S,N)=0$ at some $m$.  Then the complex $\sRHom_{\FK{G}_q}(S,N)$ is isomorphic to a bounded complex $\msc{F}$ of sheaves over $\dG/\dB$.  From the identification
\[
\Ext_{\FK{G}_q}(S,N)=\mbb{H}^\ast(\dG/\dB,\sRHom_{\FK{G}_q}(S,N))\cong \mbb{H}^\ast(\dG/\dB,\msc{F})
\]
of Proposition \ref{prop:D_Enh} (or Corollary \ref{cor:372}), and the fact that the hypercohomology of a bounded dg sheaf over $\dG/\dB$ is bounded, it follows that $\Ext^{\gg 0}_{\FK{G}_q}(S,N)=0$.  Therefore $N$ is isomorphic to a bounded complex of injectives in $D^+(\FK{G}_q)$, and hence a bounded complex of projectives since $\FK{G}_q$ is Frobenius.  When $N$ is in $D_{\coh}(\FK{G}_q)$ these projectives can furthermore be chosen to be coherent.
\end{proof}

\subsection{Projectivity and $\tN$-support}

Essentially as a corollary to Lemma \ref{lem:sExt_proj} we obtain the following.

\begin{proposition}\label{prop:projectives}
For an object $M$ in $D_{\coh}(\FK{G}_q)$ the following are equivalent:
\begin{itemize}
\item [a)] The $\tN$-support of $M$ vanishes, $\supp_{\tN}(M)=\emptyset$.
\item[b)] $M$ is isomorphic to a bounded complex of projectives.
\end{itemize}
\end{proposition}

\begin{proof}
A $\mbb{G}_m$-equivariant coherent sheaf $\msc{F}$ over $\tN$ is supported at the $0$-section of the flag variety $\dG/\dB\to \tN$ if and only if $\msc{F}$ is annihilated by a power of the ideal of definition $m\subset \O_{\tN}$ for the $0$-section.  By checking over affine opens $U\to \dG/\dB$, whose preimages $p^{-1}U\subset \tN$ provide a $\mbb{G}_m$-stable affine cover of $\tN$, one sees that such an $\msc{F}$ must have bounded grading when considered as a graded sheaf over $\dG/\dB$.  From this generic information we see that $\sExt_{\FK{G}_q}^{\gg 0}(S,M)=0$, where $S$ is the sum of the simples, if and only if $\sExt_{\FK{G}_q}(S,M)$ is supported at $\dG/\dB\subset \tN$.  This in turn occurs if and only if $\supp_{\tN}(M)=\emptyset$.  Finally, Lemma \ref{lem:sExt_proj} tells us that $\sExt_{\FK{G}_q}^{\gg 0}(S,M)=0$ if and only if $M$ is isomorphic to a bounded complex of projectives.
\end{proof}

\subsection{A strong projectivity test: The proof of Theorem \ref{thm:projectivity}}

We now arrive at the main point of the section.  We first recall a commutative algebra lemma.

\begin{lemma}\label{lem:CA1}
Let $\msc{F}$ be a complex of quasi-coherent sheaves over a finite type scheme $X$.  Suppose that there exists some $n>0$ such, at all geometric points $\lambda:\Spec(K)\to X$, $H^{i}(\operatorname{L}\lambda^\ast\msc{F})=0$ whenever $i>n$.  Then $H^{j}(\msc{F})=0$ for all $j>n+\dim(X)$.
\end{lemma}

For $X$ smooth this result follows by \cite[Lemma 1.3]{pevtsova02}.  In general one can proceed by induction on $\dim(X)$ and argue as in the proof of \cite[Proposition 5.3.F]{avramovfoxby91}.  Since we only use this result in the smooth setting we leave the details for the singular case to the interested reader.

\begin{proof}[Proof of Theorem \ref{thm:projectivity}]
Since base change and restriction preserve projectives, (a) implies (b).  We suppose now that $M$ has the property that $\res_\lambda M$ is isomorphic to a bounded complex of projectives in $D(\msc{B}_\lambda)$, at all geometric points $\lambda:\Spec(K)\to \dG/\dB$.  We want to prove that $M$ is isomorphic to a bounded complex of projectives in $D(\FK{G}_q)$.  Consider the sheaf $\sRHom_{\FK{G}_q}(S,M)=\sHom_{\FK{G}_q}(P_S,M)$, where $P_S\to S$ is a projective resolution of the simples in $\FK{G}_q$.  By Proposition \ref{prop:1356} and Lemma \ref{lem:1399} we calculate the fibers at geometric point $\lambda:\Spec(K)\to \dG/\dB$ as
\[
\operatorname{L}\lambda^\ast\sRHom_{\FK{G}_q}(S,M)\overset{\sim}\to \sRHom_{\msc{B}_\lambda}(\res_\lambda S,\res_\lambda M).
\]
Since $\res_\lambda M$ is isomorphic to a bounded complex of projective, and hence injective objects in $\msc{B}_\lambda$ by hypothesis, we have
\[
H^{>r}(\operatorname{L}\lambda^\ast\sRHom_{\FK{G}_q}(S,M))=0
\]
for some uniformly chose $r$ at all $\lambda$.  (Here $r$ can be chosen to be the maximal integer so that $H^r(M)\neq 0$.)  By Lemma \ref{lem:CA1} it follows that
\[
\sExt^{>n}_{\FK{G}_q}(S,M)=H^{>n}(\sRHom_{\FK{G}_q}(S,M))=0
\]
at some $n$, and hence that $M$ is isomorphic to a bounded complex of projective in $D(\FK{G}_q)$ by Lemma \ref{lem:sExt_proj}.  We are done.
\end{proof}

\begin{remark}
One can compare the above result with its modular analog from \cite{pevtsova02}, where the methods employed are somewhat different.
\end{remark}

\section{$\tN$-support as a global support over the Borels}
\label{sect:naturality}

We consider cohomological support for the small quantum Borels, and adopt the shorthand
\[
\supp^{\operatorname{chom}}_{\lambda}(L):=\supp^{\operatorname{chom}}_{\mathfrak{n}_\lambda}(L)=\mbb{P}(\Supp_{\mfk{n}_\lambda}\Ext_{\msc{B}_\lambda}(L,L))\ \ \text{for $L$ in }D_{\coh}(\msc{B}_\lambda).
\]
Here we employ the identification $\Ext_{\msc{B}_\lambda}(\1,\1)=\O(\mfk{n}_\lambda)$ induced by the restriction functor
\[
\res_\lambda:\O(\mcl{N})=\Ext_{\FK{G}_q}(\1,\1)\to \Ext_{\msc{B}_\lambda}(\1,\1).
\]

\begin{theorem}[Naturality and reconstruction]\label{thm:natrl}
For any $M$ in $D_{\coh}(\FK{G}_q)$, choice of geometric point $\lambda:\Spec(K)\to \dG/\dB$, and corresponding map $i_\lambda:\mfk{n}_\lambda\to \tN$ we have
\[
\mbb{P}(i_\lambda)^{-1}\supp_{\tN}(M)=\supp^{\operatorname{chom}}_{\lambda}(\res_\lambda M).
\]
Furthermore, $\tN$-support is reconstructed (as a set) from the supports over the Borels,
\[
\supp_{\tN}(M)=\bigcup_{\lambda:\Spec(K)\to \dG/\dB} \mbb{P}(i_\lambda)\left(\supp^{\operatorname{chom}}_\lambda(\res_\lambda M)\right).
\]
\end{theorem}

This naturality property allows us to reduce various analyses of support for the quantum Frobenius kernel to corresponding analysis of support for the quantum Borels $\msc{B}_\lambda$.  One can see, for example, Proposition \ref{prop:2588} and Lemma \ref{lem:2596} below.

\subsection{Some useful lemmas}

Fix a geometric point $\lambda:\Spec(K)\to \dG/\dB$.  We have the unit of the pullback-pushforward adjunction which provides natural maps
\begin{equation}\label{eq:2276}
\sRHom_{\FK{G}_q}(M,N)\to \lambda_\ast\operatorname{L}\lambda^\ast\sRHom_{\FK{G}_q}(M,N).
\end{equation}
We take cohomology to get maps
\[
\sExt_{\FK{G}_q}(M,N)\to \lambda_\ast H^\ast(\operatorname{L}\lambda^\ast\sRHom_{\FK{G}_q}(M,N))
\]
which then reduce to give natural maps
\begin{equation}\label{eq:2284}
\lambda^\ast\sExt_{\FK{G}_q}(M,N)\to H^\ast(\operatorname{L}\lambda^\ast\sRHom_{\FK{G}_q}(M,N))
\end{equation}
via adjunction.

\begin{lemma}\label{lem:2289}
At an arbitrary geometric point $\lambda:\Spec(K)\to \dG/\dB$, and coherent equivariant vector bundles $E_V$ and $E_W$ in $\FK{G}_q$, the map \eqref{eq:2284} provides a natural isomorphism
\[
\lambda^\ast\sExt_{\FK{G}_q}(E_V,E_W)\overset{\cong}\longrightarrow H^\ast(\operatorname{L}\lambda^\ast\sRHom_{\FK{G}_q}(E_V,E_W)).
\]
\end{lemma}

\begin{proof}
Let $\iota=\iota_\lambda:B_\lambda\to \dG$ be the fiber over the point $\lambda:\Spec(K)\to \dG/\dB$.  We can check that the given map is an isomorphism after applying the equivalence $\pi^\ast:\QCoh(\dG/\dB)\to \QCoh(\dG)^{\dB}$.  Over $\dG$ we are pulling back along the map $\iota$ and \eqref{eq:2284} is the descent of the corresponding morphism
\[
\iota^\ast\sExt_{\dG/\su}(E_V,E_W)\to H^\ast(\operatorname{L}\iota^\ast\sRHom_{\dG/\su}(E_V,E_W)),
\]
by the explicit description of sheaf-Hom given in Section \ref{sect:shom2}.  At the level of equivariant cochains we have $\sRHom_{\dG/\su}(E_V,E_W)=\O_{\dG}\ot_k\RHom_{\su}(V,W)$ so that
\[
\operatorname{L}\iota^\ast\sRHom_{\dG/\su}(E_V,E_W)=\iota^\ast\O_{\dG}\ot_k\RHom_{\su}(V,W)
\]
and the map \eqref{eq:2276} is just the expected reduction
\[
\O_{\dG}\ot_k\RHom_{\su}(V,W)\to \iota_\ast\O_{\dB_\lambda}\ot_k\RHom_{\su}(V,W)
\]
induced by the reduction on the first factor $\O_{\dG}$.  One takes cohomology to observe that the map of interest
\[
\iota^\ast\O_{\dG}\ot_k\Ext_{\su}(V,W)\to H^\ast(\iota^\ast\O_{\dG}\ot_k\RHom_{\su}(V,W))
\]
is just the obvious isomorphism.
\end{proof}

Now, we have the global sections morphism $\O_{\dG/\dB}\ot_k\RHom_{\FK{G}_q}\to \sRHom_{\FK{G}_q}$ of Section \ref{sect:shom1} and the map
\[
\operatorname{L}\lambda^\ast\sRHom_{\FK{G}_q}\to \RHom_{\msc{B}_\lambda}(\res_\lambda -,\res_\lambda -)
\]
of Theorem \ref{thm:fiber_l}.  We therefore compose with \eqref{eq:2276} and take cohomology to obtain a binatural morphism
\begin{equation}\label{eq:mys}
\operatorname{mystery.map}:\O_{\dG/\dB}\ot_k\Ext_{\FK{G}_q}(M,N)\to \lambda_\ast \Ext_{\msc{B}_\lambda}(\res_\lambda M,\res_\lambda N).
\end{equation}
This mystery map is defined by its global sections, at which point the $\O_{\dG/\dB}$ factor, and the pushforward $\lambda_\ast$ in \eqref{eq:mys} disappear.  The following lemma clarifies a subtle point.

\begin{lemma}\label{lem:mys}
The global sections of the mystery map \eqref{eq:mys} are identified with restriction $\res_\lambda:\Ext_{\FK{G}_q}(M,N)\to \Ext_{\msc{B}_\lambda}(\res_\lambda M,\res_\lambda N)$.
\end{lemma}

\begin{proof}[Sketch proof]
As usual, the identification of these two maps follows by some compatibility with evaluation.  To summarize: The map \eqref{eq:mys} is alternatively specified, at the cochain level, by the corresponding pullback
\begin{equation}\label{eq:2448}
\RHom_{\FK{G}_q}(M,N)\cong \operatorname{L}\lambda^\ast\O_{\dG/\dB}\ot_k\RHom_{\FK{G}_q}(M,N)\to \RHom_{\msc{B}_\lambda}(\res_\lambda M,\res_\lambda N),
\end{equation}
which is given by precomposing the global sections morphism for $\sRHom_{\FK{G}_q}$ with the reduction
\begin{equation}\label{eq:2452}
\sRHom_{\FK{G}_q}(M,N)\to \operatorname{L}\lambda^\ast\sRHom_{\FK{G}_q}(M,N)\underset{\rm Theorem\ \ref{thm:fiber_l}}\cong \RHom_{\msc{B}_\lambda}(\res_\lambda M,\res_\lambda N).
\end{equation}
Since the map \eqref{eq:2452} is compatible with evaluation, as is the global sections map $\O_{\dG/\dB}\ot_k\RHom_{\FK{G}_q}\to \sRHom_{\FK{G}_q}$, we see that \eqref{eq:2448} is compatible with evaluation.  This specifies the morphism \eqref{eq:2448} as the one induced by restriction $\res_\lambda:\FK{G}_q\to \msc{B}_\lambda$.
\end{proof}

\subsection{Supports and the fibers}

At a given point $\lambda:\Spec(K)\to \dG/\dB$ the associated embedding $\mfk{n}_\lambda\to \tN_K$ provides a calculation of the fiber $\O(\mfk{n}_\lambda)\cong \lambda^\ast p_\ast\O_{\tN}$, as an algebra over $\Spec(K)$.  By Lemma \ref{lem:2289} we have a corresponding calculation $\O(\mfk{n}_\lambda)\cong H^\ast(\operatorname{L}\lambda^\ast\msc{R}(G_q))$.  So, for any $\msc{R}(G_q)$-module $\msc{G}$ in $D(\dG/\dB)$ one takes fibers and cohomology to obtain two graded modules
\begin{equation}\label{eq:2371}
\lambda^\ast H^\ast(\msc{G})\ \ \text{and}\ \ H^\ast(\operatorname{L}\lambda^\ast\msc{G})
\end{equation}
over $\O(\mfk{n}_\lambda)$.
\par

In the case of $\msc{G}=\sRHom_{\FK{G}_q}(M,N)$, with $M$ and $N$ coherent and in the image of the de-equivariantization $E_-:D_{\rm fin}(G_q)\to D_{\coh}(\FK{G}_q)$, one can employ Lemma \ref{lem:2289} to deduce an identification of $\O(\mfk{n}_\lambda)$-modules $\lambda^\ast H^\ast(\msc{G})= H^\ast(\operatorname{L}\lambda^\ast\msc{G})$.  At general $M$ and $N$ one does not expect such an identification of the modules \eqref{eq:2371}.  However, one does find that the supports of these two modules are identified at arbitrary coherent dg sheaves.

\begin{proposition}\label{prop:lambda_supp}
At an arbitrary geometric point $\lambda:\Spec(K)\to \dG/\dB$, and $M$ and $N$ in $D_{\coh}(\FK{G}_q)$, we have an equality of supports
\[
\Supp_{\mfk{n}_\lambda}\lambda^\ast\sExt_{\FK{G}_q}(M,N)=\Supp_{\mfk{n}_\lambda}H^\ast(\operatorname{L}\lambda^\ast\sRHom_{\FK{G}_q}(M,N)).
\]
\end{proposition}

The point is essentially that the support of the cohomology of a dg sheaf over a formal dg scheme can be calculated by taking derived fibers at the dg level.  (By a formal dg scheme here we mean one which is quasi-isomorphic to its cohomology.)  However, one has to deal with certain complications due to the fact that the algebra $\msc{R}(G_q)$ has not been shown to be formal at this point (cf.\ Conjecture \ref{conj:formality} below).

\begin{proof}
We may assume $\lambda$ is a $k$-point by changing base if necessary.  We first make an argument about the pullback $\pi^\ast\sRHom_{\FK{G}_q}(M,N)=\sRHom_{\dG/\su}(M,N)$ along the flat cover $\pi:\dG\to \dG/\dB$ (see Proposition \ref{prop:shom_exp}).\vspace{2mm}

\noindent {\bf Step 1:} {\it The algebra $\sRHom_{\dG/\su}(\1,\1)$ is realizable as a quasi-coherent sheaf of dg algebras over $\dG$, and $\sRHom_{\dG/\su}(M,N)$ is realizable as a dg module over this algebra.  Furthermore, the dg algebra $\sRHom_{\dG/\su}(\1,\1)$ is formal.}\vspace{2mm}

Consider the sheaf $\sRHom_{\dG/\su}(M,N)$ of $\sRHom_{\dG/\su}(\1,\1)$-modules in $D(\dG)$.  (Here we forget about $\dB$-equivariance.)  We note that, over $\dG$, $\sRHom_{\dG/\su}(\1,\1)$ is represented by the dg algebra $\sEnd_{\dG/\su}(\mcl{P})$ and $\sRHom_{\dG/\su}(M,N)$ is represented by the dg module $\sHom_{\dG/\su}(\mcl{P}\ot M,N)$ over this dg algebra.  Here $\mcl{P}\to \1$ is a projective resolution of the unit, and we may choose $\mcl{P}=E_{P}(=\O_{\dG}\ot_k P)$ for a projective resolution $P\to k$ of the trivial representation in $\Rep G_q$.  Then we have the dg algebra quasi-isomorphism
\[
\O_{\dG}\ot_k \End_{\su}(P)\to \sEnd_{\dG/\su}(\mcl{P}),\ \ f\ot_k\xi\mapsto (f\cdot-)\ot_k \xi 
\]
and by the generic formality result \cite[Proposition 3.7.7]{arkhipovbezrukavnikovginzburg04} we have a quasi-isomorphism of dg algebras
\[
\O_{\dG}\ot_k\Sym(\mfk{n}^\ast)\overset{\sim}\to \O_{\dG}\ot_k \End_{\su}(P),
\]
for some special choice of $P$.  Here the generators $\mfk{n}^\ast$ in the algebra $\O_{\dG}\ot_k\Sym(\mfk{n}^\ast)$ are taken to be in cohomological degree $2$.  So, we conclude that the dg algebra $\sEnd_{\dG/\su}(\mcl{P})$ is formal, and we have a dg sheaf of $\O_{\dG}\ot_k \Sym(\mfk{n}^\ast)$-modules $\sHom_{\dG/\su}(\mcl{P}\ot M,N)$ whose cohomology is equal to $\sExt_{\dG/\su}(M,N)$ as an $\O_{\dG}\ot_k\Sym(\mfk{n}^\ast)$-module.  We have now established the claims of Step 1.
\par

Let us now consider the closed embedding $\iota_\lambda:\dB_\lambda\to \dG$ of the $B$-coset over the given point $\lambda:\Spec(k)\to \dG$.  Pulling back along the map $\iota_\lambda$ provides a functor
\[
\operatorname{L}\iota_\lambda^\ast:\O_{\dG}\ot_k\Sym(\mfk{n}^\ast)\text{-dgmod}\to \O_{\dB_\lambda}\ot_k\Sym(\mfk{n}^\ast)\text{-dgmod}
\]
between derived categories of sheaves of dg modules.  We can also consider graded modules over these algebras, and have the pullback functor $\iota^\ast_\lambda$ between the corresponding graded module categories, at an abelian level.
\vspace{2mm}

\noindent {\bf Step 2:} {\it For a coherent dg sheaf $\msc{F}$ over $\O_{\dG}\ot\Sym(\mfk{n}^\ast)$, there is an equality of supports
\begin{equation}\label{eq:2490}
\Supp_{\dB_\lambda\times \mfk{n}}\iota_\lambda^\ast H^\ast(\msc{F})=\Supp_{\dB_\lambda\times \mfk{n}}H^\ast(\operatorname{L}\iota_\lambda^\ast\msc{F}).
\end{equation}
In particular, one has such an equality of supports for $\msc{F}=\sRHom_{\dG/\su}(M,N)$.}\vspace{2mm}

Take $\msc{A}=\O_{\dG}\ot\Sym(\mfk{n}^\ast)$. By a coherent dg sheaf over $\msc{A}$ we mean specifically a dg sheaf $\msc{F}$ with coherent cohomology.  Up to quasi-isomorphism, such $\msc{F}$ can be assumed to be a locally finitely generated over $\msc{A}$ \cite[Lemma 4.5]{negronpevtsova}.
\par

We consider a point in the dg spectrum $\Spec(\msc{A})$, i.e.\ a graded algebra map $\msc{A}\to K(x)[t,t^{-1}]$ with $\deg(t)=2$.  Here by $K(x)$ we mean the pushforward of the structure sheaf on $\Spec(K)$ along a given map $x:\Spec(K)\to \dG$, so that the dg algebra $K(x)[t,t^{-1}]$ is supported at the image of $x$.  By dg commutative algebra \cite[Theorem 2.4]{carlsoniyengar15} we understand that at any point $\msc{A}\to K(x)[t,t^{-1}]$ we have
\begin{equation}\label{eq:2395}
\begin{array}{rl}
K(x)[t,t^{-1}]\ot_{\msc{A}}H^\ast(\msc{F})=0& \Leftrightarrow\ \ K(x)[t,t^{-1}]\ot^{\operatorname{L}}_{\msc{A}}\msc{F}\ \ \text{is acyclic}.\\
\end{array}
\end{equation}
\par

If we consider the embedding $\iota_\lambda:\dB_\lambda\to \dG$ as above, the formula \eqref{eq:2395} tells us
\[
\Supp_{\dB_\lambda\times \mfk{n}}\iota_\lambda^\ast H^\ast(\msc{F})=\text{the dg support of $\operatorname{L}\iota_\lambda^\ast\msc{F}$ over }\dB_\lambda\times \mfk{n},
\]
where the dg support is given by taking derived fibers along graded algebra maps $\msc{A}\to K(x)[t,t^{-1}]$.  Now a similar calculation to \eqref{eq:2395}, with $\msc{F}$ replaced by $\overline{\msc{F}}=\operatorname{L}\iota_\lambda^\ast\msc{F}$, gives
\[
\text{the dg support of $\operatorname{L}\iota_\lambda^\ast\msc{F}$ over }\dB_\lambda\times \mfk{n}=\Supp_{\dB_\lambda\times \mfk{n}}H^\ast(\operatorname{L}\iota_\lambda^\ast\msc{F})
\]
and thus $\Supp_{\dB_\lambda\times \mfk{n}}\iota_\lambda^\ast H^\ast(\msc{F})=\Supp_{\dB_\lambda\times \mfk{n}}H^\ast(\operatorname{L}\iota_\lambda^\ast\msc{F})$.
\vspace{2mm}

\noindent {\bf Step 3:} {\it The equality of supports \eqref{eq:2490} implies an equality of supports $\Supp_{\mfk{n}_\lambda}\lambda^\ast H^\ast(\msc{G})=\Supp_{\mfk{n}_\lambda}H^\ast(\operatorname{L}\lambda^\ast\msc{G})$ for the $\msc{R}(G_q)$-module $\msc{G}=\sRHom_{\FK{G}_q}(M,N)$.}\vspace{2mm}

Via the pullback diagram
\[
\xymatrix{
\dB_\lambda\ar[rr]^{\iota_\lambda}\ar[d]_{\rm unit} & & \dG\ar[d]^\pi\\
\Spec(k)\ar[rr]^\lambda & & \dG/\dB.
}
\]
we observe an identification of monoidal functors $\operatorname{L}\iota_\lambda^\ast\pi^\ast\cong\operatorname{unit}^\ast\operatorname{L}\lambda^\ast:D(\dG/\dB)\to D(\dB_\lambda)$ which restricts to an isomorphism of functors between module categories
\begin{equation}\label{eq:2503}
\operatorname{L}\iota_\lambda^\ast\pi^\ast\cong \operatorname{unit}^\ast\operatorname{L}\lambda^\ast:\msc{R}(G_q)\text{-mod}\to \O_{\dB_\lambda}\ot_k\Sym(\mfk{n}^\ast)\text{-mod},
\end{equation}
where we employ the isomorphism of $\O_{\dB_\lambda}$-algebras
\[
\operatorname{L}\iota_\lambda^\ast\pi^\ast\msc{R}(G_q)\cong \operatorname{L}\iota^\ast_\lambda\sRHom_{\dG/\su}(\1,\1)\cong \O_{\dB_\lambda}\ot_k\Sym(\mfk{n}^\ast)
\]
deduced from Step 1 above.  To be clear, in \eqref{eq:2503} we are considering the module categories for these \emph{algebra objects} in their respective derived categories.
\par

We also calculate
\[
\Sym(\mfk{n}^\ast_\lambda)\overset{\sim}\to \RHom_{\msc{B}_\lambda}(\1,\1)\cong\operatorname{L}\lambda^\ast\msc{R}(G_q)
\]
by Theorem \ref{thm:fiber_l} and another application of the formality result of \cite[Proposition 3.7.7]{arkhipovbezrukavnikovginzburg04}.  Hence $\operatorname{unit}^\ast\operatorname{L}\lambda^\ast\msc{R}(G_q)\cong \O_{\dB_\lambda}\ot_k\Sym(\mfk{n}^\ast_\lambda)$ and the identification of functors $\operatorname{L}\iota_\lambda^\ast\pi^\ast\cong\operatorname{unit}^\ast\operatorname{L}\lambda^\ast$ provides an isomorphism of algebras
\begin{equation}\label{eq:2518}
\phi:\O_{\dB_\lambda}\ot_k\Sym(\mfk{n}^\ast_\lambda)\overset{\cong}\to \O_{\dB_\lambda}\ot_k\Sym(\mfk{n}^\ast).
\end{equation}
One changes base along $\phi$ to view the pullback functor
\[
\operatorname{unit}^\ast:\Sym(\mfk{n}^\ast_\lambda)\text{-mod}\to \O_{\dB_\lambda}\ot_k\Sym(\mfk{n}^\ast_\lambda)\text{-mod}\underset{\res_{\phi^{-1}}}\cong\O_{\dB_\lambda}\ot_k\Sym(\mfk{n}^\ast)\text{-mod}
\]
as a functor to the category of $\O_{\dB_\lambda}\ot_k\Sym(\mfk{n}^\ast)$-modules, and one employs this particular understanding of the pullback $\operatorname{unit}^\ast$ in the identification of functors \eqref{eq:2503}.
\par

To elaborate further on the isomorphism \eqref{eq:2518}, the map $\phi$ is dual to an isomorphism of $B_\lambda$-schemes $\varphi:B_\lambda\times \mfk{n}\to B_\lambda\times \mfk{n}_\lambda$, which is subsequently specified by a map of $k$-schemes
\[
t:B_\lambda\times \mfk{n}\to \mfk{n}_\lambda.
\]
The map $t$ is the composite of $\varphi$ with the projection $p_2:B_\lambda\times \mfk{n}_\lambda\to \mfk{n}_\lambda$.  The functor $\operatorname{unit}^\ast:\Sym(\mfk{n}^\ast_\lambda)\text{-mod}\to \O_{\dB_\lambda}\ot_k\Sym(\mfk{n}^\ast)\text{-mod}$ is then, geometrically, just pulling back along $t:B_\lambda\times \mfk{n}\to \mfk{n}_\lambda$.  It does not matter what the map $t$ is precisely, but one can check that $t$ is the expected flat covering $B_\lambda\times \mfk{n}\to \mfk{n}_\lambda$, $(g,x)\mapsto \Ad_g(x)$.

In any case, we consider the map $t:\dB_\lambda\times \mfk{n}\to \mfk{n}_\lambda$, and the sheaves $\msc{G}=\sRHom_{\FK{G}_q}(M,N)$ and $\msc{F}=\sRHom_{\dG/\su}(M,N)$.  We note that $\pi^\ast\msc{G}\cong \msc{F}$, by Proposition \ref{prop:shom_exp}.  We therefore calculate
\[
\begin{array}{rl}
t^{-1}\Supp_{\mfk{n}_\lambda}H^\ast(\operatorname{L}\lambda^\ast\msc{G})&=\Supp_{\dB_\lambda \times \mfk{n}} \operatorname{unit}^\ast H^\ast(\operatorname{L}\lambda^\ast\msc{G})\\
&=\Supp_{\dB_\lambda \times \mfk{n}}H^\ast(\operatorname{unit}^\ast \operatorname{L}\lambda^\ast\msc{G})\\
&=\Supp_{\dB_\lambda \times \mfk{n}}H^\ast(\operatorname{L}\iota_\lambda^\ast\pi^\ast\msc{G})=\Supp_{\dB_\lambda\times \mfk{n}}H^\ast(\operatorname{L}\iota_\lambda^\ast\msc{F}).
\end{array}
\]
A similar argument, now at the abelian level, provides a calculation
\[
t^{-1}\Supp_{\mfk{n}_\lambda}\lambda^\ast H^\ast(\msc{G})=\Supp_{\dB_\lambda\times \mfk{n}}\iota_\lambda^\ast\pi^\ast H^\ast(\msc{G})=\Supp_{\dB_\lambda\times \mfk{n}}\iota_\lambda^\ast H^\ast(\msc{F}).
\]
So we see that \eqref{eq:2490} implies an equality $t^{-1}\Supp_{\mfk{n}_\lambda}\lambda^\ast H^\ast(\msc{G})=t^{-1}\Supp_{\mfk{n}_\lambda}H^\ast(\operatorname{L}\lambda^\ast\msc{G})$.  Since $t$ is surjective we find the desired equality
\[
\Supp_{\mfk{n}_\lambda}\lambda^\ast H^\ast(\msc{G})=\Supp_{\mfk{n}_\lambda}H^\ast(\operatorname{L}\lambda^\ast\msc{G}).
\]
\end{proof}

\subsection{Proof of the naturality and reconstruction}

\begin{proof}[Proof of Theorem \ref{thm:natrl}]
The reconstruction claim of Theorem \ref{thm:natrl} follows from the naturality claim, $\supp^{\operatorname{chom}}_\lambda(\res_\lambda M)=\mbb{P}(i_\lambda)^{-1}\supp_{\tN}(M)$, since all points in $\tN$ lie in the image of some $i_\lambda:\mfk{n}_\lambda\to \tN$.  So we need only establish the equaltiy $\supp^{\operatorname{chom}}_\lambda(\res_\lambda M)=\mbb{P}(i_\lambda)^{-1}\supp_{\tN}(M)$ at fixed $\lambda:\Spec(K)\to \dG/\dB$.
\par

We have the diagram
\[
\xymatrix{
\operatorname{L}\lambda^\ast\sRHom_{\FK{G}_q}(\1,\1)\ar[d]^\sim\ar[rr]^{-\ot M|_\lambda}& & \operatorname{L}\lambda^\ast\sRHom_{\FK{G}_q}(M,M)\ar[d]^\sim\\
\RHom_{\msc{B}_\lambda}(\1,\1)\ar[rr]^{-\ot \res_\lambda M} & & \RHom_{\msc{B}_\lambda}(\res_\lambda M,\res_\lambda M)
}
\]
via Theorem \ref{thm:fiber_l}, and take cohomology to obtain a diagram
\begin{equation}\label{eq:2492}
\xymatrix{
\O(\mfk{n}_\lambda)=\lambda^\ast\sExt_{\FK{G}_q}(\1,\1)\ar[d]^\cong\ar[rr]^{H^\ast(-\ot M|_\lambda)}& & H^\ast(\operatorname{L}\lambda^\ast\sRHom_{\FK{G}_q}(M,M))\ar[d]^\cong\\
\Ext_{\msc{B}_\lambda}(\1,\1)\ar[rr]^{-\ot \res_\lambda M} & & \Ext_{\msc{B}_\lambda}(\res_\lambda M,\res_\lambda M).
}
\end{equation}
(We employ Lemma \ref{lem:2289} here to calculate the cohomology of $\operatorname{L}\lambda\sRHom(\1,\1)$.)  
\par

Since $\Supp_{\mfk{n}_\lambda}H^\ast(\operatorname{L}\lambda^\ast\sRHom_{\FK{G}_q}(M,M))$ is the vanishing locus of the kernel $\msc{I}_\lambda$ of the action map $H^\ast(-\ot M|_\lambda)$ appearing in \eqref{eq:2492}, and similarly
\[
\Supp_{\mfk{n}_\lambda}\Ext_{\msc{B}_\lambda}(\res_\lambda M,\res_\lambda M)=Van(I_\lambda)
\]
for $I_\lambda=\ker(-\ot\res_\lambda M)$, we see that
\begin{equation}\label{eq:2523}
\Supp_{\mfk{n}_\lambda}\Ext_{\msc{B}_\lambda}(\res_\lambda M,\res_\lambda M)=\Supp_{\mfk{n}_\lambda}H^\ast(\operatorname{L}\lambda^\ast\sRHom_{\FK{G}_q}(M,M)).
\end{equation}
That is, we have such an equality of supports \emph{provided} we employ the identification
\[
\O(\mfk{n}_\lambda)\cong \Ext_{\msc{B}_\lambda}(\1,\1)
\]
induced by the isomorphism appearing on the left-hand side of \eqref{eq:2492}.
\par

By Proposition \ref{prop:lambda_supp}, \eqref{eq:2523} implies a calculation of support
\[
\Supp_{\mfk{n}_\lambda}\Ext_{\msc{B}_\lambda}(\res_\lambda M,\res_\lambda M)=\Supp_{\mfk{n}_\lambda}(\lambda^\ast\sExt_{\FK{G}_q}(M,M)).
\]
and pulling back along $i_\lambda:\mfk{n}_\lambda\to \tN$ gives
\[
\Supp_{\mfk{n}_\lambda}(\lambda^\ast\sExt_{\FK{G}_q}(M,M))=i_\lambda^{-1}\Supp_{\tN}\sExt_{\FK{G}_q}(M,M).
\]
We take projectivizations to obtain the desired equality
\[
\supp^{\operatorname{chom}}_\lambda(\res_\lambda M)=\mbb{P}(i_\lambda)^{-1}\supp_{\tN}(M).
\]

All that is left to check is that the identification of $\Ext_{\msc{B}_\lambda}(\1,\1)$ with $\O(\mfk{n}_\lambda)$ via the map from $\lambda^\ast\sExt_{\FK{G}_q}(\1,\1)$ agrees with the identification provided by the restriction functor from $\Ext_{\FK{G}_q}(\1,\1)$.  But this point was already dealt with at Lemma \ref{lem:mys} above.
\end{proof}


\section{Cohomological support and the moment map}
\label{sect:cohom}

We consider the moment map $\kappa:\tN\to \mcl{N}$ and its projectivization $\mbb{P}(\kappa):\mbb{P}(\tN)\to \mbb{P}(\mcl{N})$.  We let $\supp^{\operatorname{chom}}_{\mcl{N}}$ denote the standard cohomological support for $\FK{G}_q$,
\[
\supp^{\operatorname{chom}}_{\mathcal{N}}(M):=\mbb{P}\left(\Supp_{\mcl{N}}\Ext_{\FK{G}_q}(M,M)\right)\ \subset\ \mbb{P}(\mcl{N}).
\]
Recall the algebra identification $\Ext_{\FK{G}_q}(\1,\1)=\O(\mcl{N})$ of Corollary \ref{cor:GK}, and let $\O(\mcl{N})$ act on $\Ext_{\FK{G}_q}(M,M)$ via the tensor structure.  In this section we prove the following localization result for cohomological support.

\begin{theorem}\label{thm:tN_N}
For any $M$ in $D_{\coh}(\FK{G}_q)$ we have
\[
\mbb{P}(\kappa)\left(\supp_{\tN}(M)\right)=\supp^{\operatorname{chom}}_{\mcl{N}}(M).
\]
\end{theorem}

\subsection{The first inclusion}

Recall the multiplicative spectral sequence
\[
E_2=H^\ast(\dG/\dB,\sExt_{\FK{G}_q}(M,M))\ \Rightarrow\ \Ext_{\FK{G}_q}(M,M)
\]
of Corollary \ref{cor:372} (see also Proposition \ref{prop:D_Enh}).  This spectral sequence provides a bounded below filtration by ideals
\[
\begin{array}{l}
0=F_{-N-1}\Ext_{\FK{G}_q}(M,M)\subset F_{-N}\Ext_{\FK{G}_q}(M,M)\subset\vspace{1mm}\\
\hspace{5cm}\dots\subset F_{0}\Ext_{\FK{G}_q}(M,M)=\Ext_{\FK{G}_q}(M,M),
\end{array}
\]
and we find a nilpotent ideal
\[
\mathfrak{N}:=F_{-1}\Ext_{\FK{G}_q}(M,M)
\]
for which the quotient embeds into the corresponding global sections of the sheaf-extensions
\[
\Ext_{\FK{G}_q}(M,M)/\mathfrak{N}\hookrightarrow \Gamma(\dG/\dB,\sExt_{\FK{G}_q}(M,M)).
\]
(cf.\ \cite[proof of Theorem 4.1]{suslinfriedlanderbendel97b}).  Now, the global sections of $\sExt_{\FK{G}_q}(M,M)$ over $\dG/\dB$ are identified with the global sections of the pushforward along the moment map $\kappa:\tN\to \mcl{N}$, so that we have an inclusion of sheaves
\begin{equation}\label{eq:2081}
(\Ext_{\FK{G}_q}(M,M)/\mathfrak{N})^\sim\hookrightarrow\kappa_\ast\sExt_{\FK{G}_q}(M,M)
\end{equation}
of $\O_{\mcl{N}}$-algebras.  We therefore observe the following.

\begin{lemma}\label{lem:2086}
For any $M$ in $D_{\coh}(\FK{G}_q)$ there is an inclusion of supports
\[
\Supp_{\mcl{N}}\Ext_{\FK{G}_q}(M,M)\subset \kappa(\Supp_{\tN}\sExt_{\FK{G}_q}(M,M)).
\]
\end{lemma}

\begin{proof}
The support of $\Ext_{\FK{G}_q}(M,M)$ as an $\Ext_{\FK{G}_q}(\1,\1)=\O(\mcl{N})$-module is, as an embedded topological subspace in $\mcl{N}$, the spectrum of the quotient
\[
\Supp_{\mcl{N}}\Ext_{\FK{G}_q}(M,M)\underset{\rm top}=\Spec(\O(\mcl{N})/I_M)
\]
where $I_M$ is the kernel of the algebra map $-\ot M:\Ext_{\FK{G}_q}(\1,\1)\to \Ext_{\FK{G}_q}(M,M)$.  Since the ideal $\mfk{N}\subset \Ext_{\FK{G}_q}(M,M)$, defined above, is nilpotent the above spectrum is equal to the spectrum $\Spec(\O(\mcl{N})/I_M)=_{\rm top}\Spec(\O(\mcl{N})/I'_M)$ where
\[
I'_M=\ker\big(\Ext_{\FK{G}_q}(\1,\1)\to \Ext_{\FK{G}_q}(M,M)\to \Ext(M,M)/\mfk{N}\big).
\]
This is to say,
\[
\Supp_{\mcl{N}}\Ext_{\FK{G}_q}(M,M)=\Supp_{\mcl{N}}(\Ext(M,M)/\mfk{N}).
\]
From the inclusion \eqref{eq:2081} we deduce an inclusion
\[
\Supp_{\mcl{N}}(\Ext_{\FK{G}_q}(M,M)/\mfk{N})\subset \Supp_{\mcl{N}}\kappa_\ast\sExt_{\FK{G}_q}(M,M),
\]
Finally, since the moment map $\kappa:\tN\to \mcl{N}$ is proper, and in particular closed, we have an inclusion
\[
\Supp_{\mcl{N}}\kappa_\ast\sExt_{\FK{G}_q}(M,M)\subset \kappa(\Supp_{\tN}\sExt_{\FK{G}_q}(M,M)).
\]
In total, we obtain the claimed inclusion
\[
\Supp_{\mcl{N}}\Ext_{\FK{G}_q}(M,M)=\Supp_{\mcl{N}}(\Ext_{\FK{G}_q}(M,M)/\mfk{N})\subset \kappa(\Supp_{\tN}\sExt_{\FK{G}_q}(M,M)).
\]
\end{proof}

\subsection{The second inclusion, and proof of Theorem \ref{thm:tN_N}}

Having observed Lemma \ref{lem:2086}, the proof of Theorem \ref{thm:tN_N} is reduced to an inclusion
\[
\kappa(\Supp_{\tN}\sExt_{\FK{G}_q}(M,M))\ \underset{\rm desired}\subset\ \Supp_{\mcl{N}}\Ext_{\FK{G}_q}(M,M)
\]
at general $M$ in $D_{\coh}(\FK{G}_q)$.  Since these sheaves are coherent, the above inclusion is equivalent to the claim that, for any geometric point $x:\Spec(K)\to \tN$ at which the fiber $x^\ast\sExt_{\FK{G}_q}(M,M)$ is non-vanishing, $\Ext_{\FK{G}_q}(M,M)$ is supported at the corresponding point $\kappa(x):\Spec(K)\to \mcl{N}$.

\begin{lemma}\label{lem:2128}
Suppose that the fiber $x^\ast\sExt_{\FK{G}_q}(M,M)$ is non-vanishing, at a given geometric point $x:\Spec(K)\to \tN$.  Then the fiber $\kappa(x)^\ast\Ext_{\FK{G}_q}(M,M)$ is non-vanishing as well.
\end{lemma}

\begin{proof}
It suffices to consider the case $K=k$, via base change.  Take $\lambda:\Spec(k)\to \dG/\dB$ to be the corresponding point in the flag variety, $\lambda=p(x)$, and let $j_\lambda:\mfk{n}_\lambda\to \mcl{N}$ be the embedding from the corresponding nilpotent subalgebra.  Via Theorem \ref{thm:natrl} we have
\[
x\in \Supp_{\tN}\sExt_{\FK{G}_q}(M,M)\ \Leftrightarrow\ x\in \Supp_{\mfk{n}_\lambda}\Ext_{\msc{B}_{\lambda}}(\res_{\lambda} M,\res_{\lambda} M).
\]
Since $\res_\lambda$ is a tensor functor, we have the diagram
\[
\xymatrix{
\O(\mcl{N})\ar[d]_{\res_\lambda}\ar[rr]^(.4){-\ot M} & & \Ext_{\FK{G}_q}(M,M)\ar[d]^{\res_\lambda}\\
\O(\mfk{n}_{\lambda})\ar[rr]_(.35){-\ot\res_\lambda M} & & \Ext_{\msc{B}_{\lambda}}(\res_{\lambda} M,\res_{\lambda} M).
}
\]
Hence we see that
\[
j_{\lambda}\Supp_{\mfk{n}_\lambda}\Ext_{\msc{B}_{\lambda}}(\res_{\lambda}M,\res_{\lambda}M)\subset \Supp_{\mcl{N}}\Ext_{\FK{G}_q}(M,M).
\]
So, since $\res_\lambda M$ is supported at $\kappa(x)=j_\lambda(x)$ it follows that $k(x)$ is in the support of $\Ext_{\FK{G}_q}(M,M)$ over $\mcl{N}$.  We are done.
\end{proof}

As suggested above, Lemma \ref{lem:2128} is just a reformulation of the inclusion
\[
\kappa(\Supp_{\tN}\sExt_{\FK{G}_q}(M,M))\subset \Supp_{\mcl{N}}\Ext_{\FK{G}_q}(M,M),
\]
which we have now verified.  Theorem \ref{thm:tN_N} follows.

\begin{proof}[Proof of Theorem \ref{thm:tN_N}]
Combine Lemmas \ref{lem:2086} and \ref{lem:2128}.
\end{proof}


\section{Covering the Balmer Spectrum}
\label{sect:spec}

In \cite{balmer05} Balmer introduces the notion of a prime ideal spectrum $\Spec(T)$ for a braided monoidal triangulated category $T$.  This construction is strongly related to the production of so-called support data for the given monoidal category.  The goal of this section is to prove the following result, which provides a first half of a ``birational approximation" of the Balmer spectrum for the stable category of small quantum group representations.

\begin{proposition}\label{prop:bspec1}
Suppose $G$ is an almost-simple algebraic group in type A. Then $\tN$-support provides a support data for the stable category of $\FK{G}_q$, and the corresponding map to the universal support space
\[
f_{\rm univ}:\mbb{P}(\tN)\to \Spec\left(\stab\FK{G}_q\right)
\]
is surjective.
\end{proposition}

We recall the construction of the Balmer spectrum below, after discussing related notions of support theories, and support data for monoidal triangulated categories.

The proof of Proposition \ref{prop:bspec1} is a fairly straightforward application of our analysis of $\tN$-support, in particular Theorems \ref{thm:projectivity} and \ref{thm:natrl}, in conjunction with results from \cite{negronpevtsova3}.  The proof is provided in Subection \ref{sect:proof_bspec1}.  We note that the arguments employed below are not specific to type $A$, as the Dynkin type bottleneck comes from the work \cite{negronpevtsova3}.  Indeed, in Subection \ref{sect:the_others} we provide a conditional extension of Proposition \ref{prop:bspec1}, and its big sister Theorem \ref{thm:bspec2}, to arbitrary almost-simple algebraic groups.

\subsection{Support data and lavish support theories}
\label{sect:supp_nons}

Consider a general triangulated category $T$.  A (triangulated) support theory for $T$, which takes values in a given topological space $X$, is a function
\[
\supp:\operatorname{obj}T\to \{\text{subsets in }X\}
\]
which is stable under shifting, splits over sums, has $\supp(0)=\emptyset$, and satisfies $\supp(N)\subset (\supp(M)\cup\supp(M'))$ whenever one has a triangle $M\to N\to M'$, just as in Lemma \ref{lem:triangle}.  When $T$ is monoidal we require also that $\supp(M)$ is \emph{closed} in $X$ for dualizable $M$, and that $X$ appears as the support of the unit object in $T$, $\supp(\1)=X$.
\par

Suppose now that $T$ is monoidal and rigid, i.e.\ that all objects in $T$ admit left and right duals.  In this case we say a support theory for $T$ satisfies the \emph{tensor product property} if
\begin{equation}\label{eq:tpp}
\supp(M\ot N)=\supp(M)\cap\supp(N)\ \ \text{for all $M$ and $N$ in }T.
\end{equation}
Following \cite[Definition 3.1]{balmer05}, a triangulated support theory for a rigid, \emph{braided} monoidal triangulated category $T$ which satisfies the tensor product property is called a \emph{support data} for $T$.

\begin{remark}\label{rem:2720}
At this point it is understood that the tensor product property \eqref{eq:tpp} is not the appropriate ``multiplicative" condition for support theories in the non-braided context \cite{bensonwitherspoon14,plavnikwitherspoon18,negronpevtsova,negronpevtsova3}--although it does hold in some cases of interest.  So one should be more liberal in their discussions of support data outside of the braided setting.
\end{remark}

Now let us consider a monoidal triangulated category $T^{+}$ which admits all set indexed sums and is compactly generated.  Suppose that $T$ is the subcategory of rigid compact objects in $T^{+}$.  Let $\supp$ be a support theory for $T$ which takes values in a given space $X$, and suppose that this support theory satisfies the tensor product property \eqref{eq:tpp}.  Following the language of \cite[Definition 4.7]{negronpevtsova3}, we say that $\supp$ is a \emph{lavish} support theory for $T$ if there exists an extension
\[
\supp':\operatorname{obj} T^{+}\to\{\operatorname{subsets\ in\ }X\}
\]
of $\supp$ to all of $T^{+}$ which satisfies the following:
\begin{itemize}
\item[(E1)] $\supp'$ is a (triangulated) support theory for $T^{+}$.
\item[(E2)] $\supp'|_{T}=\supp$.
\item[(E3)] We have $\supp'(\oplus_{i\in I}N_i)=\cup_{i\in I}\supp'(N_i)$ for all set indexed sums in $T^{+}$.
\item[(E4)] For each $M$ in $T$ and $N$ in $T^{+}$ we have
\[
\supp'(M\ot N)=\supp(M)\cap \supp'(N).
\]
\end{itemize}

We note that the extension $\supp'$ takes values in (not-necessarily-closed) sub\emph{sets} in $X$.  We also note that the particular choice of extension $\supp\rightsquigarrow \supp'$ is \emph{not} specified in the claim that $\supp$ is lavish, one is only concerned with the existence of such an extension.

\begin{remark}
The definition of a lavish support theory given above is slightly stronger than the one given in \cite{negronpevtsova}, precisely for the reasons discussed in Remark \ref{rem:2720}.  However, in all cases of interest in this work the two definitions will coincide.
\end{remark}

\begin{remark}
The framework presented above is not really original to the work \cite{negronpevtsova3}, specifically in the braided context.  Indeed, these kinds of conditions (E1)--(E4) already appear in works of Rickard and coathors from the 90's \cite{rickard97,bensoncarlsonrickard97} (cf.\ \cite[\S 1.3]{nakanovashawyakimov1} as well).
\end{remark}

\subsection{Universal support data}
\label{sect:universal}

Let $T$ be a rigid, braided, monoidal triangulated category.  Recall that a thick tensor ideal $\msc{K}$ in $T$ is a thick subcategory \cite[\S 2.1]{neeman01} which is stable under the (left and right) tensor actions of $T$ on itself.  A thick ideal $\msc{P}$ in $T$ is a \emph{prime} thick ideal if an inclusion $M\ot N\in \msc{P}$ implies that either $M$ is in $\msc{P}$, or $N$ is in $\msc{P}$.
\par

In \cite{balmer05}, Balmer produces a universal support data for $T$, and a corresponding universal support space.  This universal support space is the spectrum of thick prime ideals
\[
\Spec(T)=\{\msc{P}\subset T:\msc{P}\text{ is a thick prime tensor ideal}\},
\]
and the universal support of an object $M$ in $T$ is defined as
\[
\supp^{\operatorname{univ}}(M)=\{\msc{P}\subset T: \msc{P}\text{ is a thick prime ideal with }M\notin \msc{P}\}.
\]
The spectrum $\Spec(T)$ has a unique topology for which the supports $\supp^{\operatorname{univ}}(M)$ provide a basis of closed subsets.  The pairing of $\Spec(T)$ with the support function $\supp^{\operatorname{univ}}$ constitutes a support data for $T$, and Balmer establishes the following universal property.

\begin{theorem}[{\cite[Theorem 3.2]{balmer05}}]\label{thm:le-balmer}
Suppose that $T$ is a rigid, braided monoidal triangulated category.  Given any support data $(X,\supp_X)$ for $T$, there is a unique continuous map $f_{X,\rm univ}:X\to \Spec(T)$ such that
\[
f_{X,\rm univ}^{-1}\supp^{\operatorname{univ}}(M)=\supp_X(M)
\]
for all $M$ in $T$.
\end{theorem}

\begin{remark}
The topological space which we denote as $\Spec(T)$ is written $\operatorname{Spc}(T)$ in \cite{balmer05}.  Balmer reserves the notation $\Spec(T)$ for $\operatorname{Spc}(T)$ along with a canonical locally ringed structure \cite[\S 6]{balmer05}.  Also, rigidity is not assumed in \cite{balmer05}, though it does appear in later works.
\end{remark}

\subsection{Stable categories versus derived categories}
\label{sect:stable}

For concreteness let $\msc{C}$ be one of $\FK{G}_q$ or $\msc{B}_\lambda$, at fixed geometric point $\lambda:\Spec(K)\to \dG/\dB$.  Recall that $\msc{C}$ is a certain category of equivariant quasi-coherent sheaves, and that $D(\msc{C})$ is the associated unbounded derived category of quasi-coherent sheaves.  We've taken $D_{\coh}(\msc{C})$ and $D^b(\msc{C})$ to be the subcategories of coherent dg sheaves, and bounded quasi-coherent dg sheaves in $D(\msc{C})$, respectively.
\par

It is traditional to use the stable categories
\begin{equation}\label{eq:2750}
\begin{array}{l}
\stab(\msc{C})=D_{\coh}(\msc{C})/\langle \operatorname{proj}(\msc{C})\rangle\vspace{2mm}\\ \Stab(\msc{C})=D^b(\msc{C})/\langle \operatorname{Proj}(\msc{C})\rangle
\end{array}
\end{equation}
when discussing support for $\msc{C}$ \cite{rickard89}, rather than the derived categories.  In the stable setting, the big stable category $\Stab(\msc{C})$ admits sums indexed over arbitrary sets and recovers the small stable category as the subcategory of rigid-compact objects
\[
\Stab(\msc{C})^c=\stab(\msc{C}).
\]
So $\Stab(\msc{C})$ can serve as a ``big" category $T^{+}$ for $T=\stab(\msc{C})$, as in Section \ref{sect:supp_nons}.
\par

From the support theory point of view, the distinction between the ``small" categories $D_{\coh}(\msc{C})$ and $\stab(\msc{C})$ is not so significant, and a support theory/data for $\stab(\msc{C})$ is simply a support theory/data for $D_{\coh}(\msc{C})$ which vanishes on projectives \cite[Theorem 3.2 \& Proposition 3.11]{balmer05}.  So, for example, if we take $Y$ to be the conical space $\Spec\Ext_{\msc{C}}(\1,\1)$, cohomological support defines a support theory for the stable category
\[
\supp^{\operatorname{chom}}_Y:\operatorname{obj}\stab(\msc{C})\to \{\text{closed subsets in }\mbb{P}(Y)=\Proj\Ext_{\msc{C}}(\1,\1)\}
\]
whose value on an object $M$ in $\stab(\msc{C})$ is just the value of $\supp^{\operatorname{chom}}_Y$ on $M$ as an object in $D_{\coh}(\msc{C})$.  This makes sense because the objects in $\stab(\msc{C})$ are equal to the objects in $D_{\coh}(\msc{C})$, according to the formulation \eqref{eq:2750}.

\subsection{Support for the small quantum Borel}
\label{sect:extend}

Given a geometric point $\lambda:\Spec(K)\to \dG/\dB$, and the corresponding small quantum Borel $\msc{B}_\lambda$, we recall our shorthand $\supp^{\operatorname{chom}}_\lambda:=\supp^{\operatorname{chom}}_{\mfk{n}_\lambda}$ for the associated cohomological support.  We consider $\supp^{\operatorname{chom}}_{\lambda}$ as a triangulated support theory for the small stable category $\stab(\msc{B}_\lambda)$, as discussed above.  We have the following assessment of the small quantum Borel.

\begin{theorem}[\cite{negronpevtsova,negronpevtsova3}]\label{thm:np}
Let $G$ be an almost-simple algebraic group in type $A$, and let $\lambda:\Spec(K)\to \dG/\dB$ be an arbitrary geometric point.  Cohomological support for $\msc{B}_\lambda$ satisfies the tensor product property and is a lavish support theory.  Indeed, hypersurface support
\[
\supp^{\operatorname{hyp}}_{\lambda}:\operatorname{obj} \Stab(\msc{B}_\lambda)\to\{\operatorname{subsets\ in\ }\mbb{P}(\mfk{n}_\lambda)\}
\]
\cite[Definition 5.8]{negronpevtsova2} provides the necessary extension of cohomological support to the big stable category $\Stab(\msc{B}_\lambda)$.
\end{theorem}

\begin{remark}
Note that cohomological support for any given $\msc{B}_\lambda$ is a lavish support theory if and only if cohomological support for the standard (positive) small quantum Borel is lavish, over an arbitrary algebraically closed field of characteristic $0$.  This just follows from the fact that each $\msc{B}_\lambda$ is isomorphic to a base change $(\msc{B}_1)_K$ of the standard Borel as a tensor category (see Section \ref{sect:borels}).
\end{remark}

\subsection{$\tN$-support as a support data}

By the vanishing result of Proposition \ref{prop:projectives} we understand that $\tN$-support factors to provide a (triangulated) support theory for the stable category of sheaves for the quantum Frobenius kernel
\begin{equation}\label{eq:2807}
\supp_{\tN}:\operatorname{obj}\stab(\FK{G}_q)\to \{\text{closed subsets in }\mbb{P}(\tN)\}.
\end{equation}
Throughout this section when we speak of $\tN$-support for the quantum Frobenius kernel we are speaking specifically of theory \eqref{eq:2807}.

\begin{proposition}\label{prop:2588}
Let $G$ be an almost-simple group in type $A$.  Then $\tN$-support satisfies the tensor product property, and hence defines a support data for the stable category $\stab(\FK{G}_q)$.  Consequently, there is a unique continuous map
\[
f_{\rm univ}:\mbb{P}(\tN)\to \Spec(\stab\FK{G}_q)
\]
for which $f_{\rm univ}^{-1}\supp^{\operatorname{univ}}(M)=\supp_{\tN}(M)$ at all $M$ in $\stab(\FK{G}_q)$.
\end{proposition}

\begin{proof}
By Theorem \ref{thm:np}, cohomological support for $\msc{B}_\lambda$ satisfies the tensor product property, at arbitrary $\lambda:\Spec(K)\to \dG/\dB$.  Hence at any point $x:\Spec(k)\to \tN$, with corresponding point $\lambda=p(x):\Spec(k)\to \dG/\dB$ for the flag variety, Theorem \ref{thm:natrl} gives
\[
\begin{array}{l}
x\in \supp_{\tN}(M\ot N)\\
\hspace{2cm} \Leftrightarrow\ x\in \supp^{\operatorname{chom}}_{\lambda}(\res_\lambda M\ot \res_\lambda N)\\
\hspace{2cm}\Leftrightarrow\ x\in \supp^{\operatorname{chom}}_{\lambda}(\res_\lambda M)\cap\supp^{\operatorname{chom}}_{\lambda}(\res_\lambda N)\\
\hspace{2cm}\Leftrightarrow\ x\in \supp_{\tN}(M)\cap\supp_{\tN}(N).
\end{array}
\]
So we see that the two sets $\supp_{\tN}(M\ot N)$ and $\supp_{\tN}(M)\cap \supp_{\tN}(N)$ share the same closed points.  Since both of these subsets are closed in $\tN$, an agreement on closed points implies an equality
\[
\supp_{\tN}(M\ot N)=\supp_{\tN}(M)\cap \supp_{\tN}(N).
\]
Since $\tN$-support is now seen to provide a support data for the stable category $\stab(\FK{G}_q)$, one obtains the corresponding map to the Balmer spectrum by Theorem \ref{thm:le-balmer}.
\end{proof}

\subsection{$\tN$-support as a lavish support theory}

\begin{lemma}\label{lem:2596}
Let $G$ be an almost-simple algebraic group in type $A$.  Then $\tN$-support provides a lavish support theory for the stable category $\stab(\FK{G}_q)$.
\end{lemma}

\begin{proof}
At each geometric point $\lambda:\Spec(K)\to \dG/\dB$ let us fix an appropriate extension $\supp'_\lambda$ of cohomological support from the small stable category $\stab(\msc{B}_\lambda)$ to the big stable category $\Stab(\msc{B}_\lambda)$.  (Such an extension exists by Theorem \ref{thm:np}.)  We construct from these $\supp'_\lambda$ the desired extension $\supp'_{\tN}$ of $\tN$-support.  
\par

In particular, we say an object $N$ in $\Stab(\FK{G}_q)$ is supported at a given point $x:\Spec(K)\to \mbb{P}(\tN)$, with corresponding point $\lambda=p(x):\Spec(K)\to \dG/\dB$ in the flag variety, if
\[
im(x)\in \supp'_\lambda(\res_\lambda N)\subset \mbb{P}(\mfk{n}_\lambda).
\]
(Note that in this case $x$ factors uniquely through the inclusion $\mbb{P}(\mfk{n}_\lambda)\to \mbb{P}(\tN_K)$.)  Take now
\[
\supp'_{\tN}(N):=\left\{\begin{array}{c}
\text{the image of all points }x:\Spec(K)\to \mbb{P}(\tN)\\
\text{at which $N$ is supported}
\end{array}\right\}\subset \mbb{P}(\tN).
\]
By the reconstruction theorem, Theorem \ref{thm:natrl}, we see that $\supp'_{\tN}(M)=\supp_{\tN}(M)$ whenever $M$ lies in $\stab(\FK{G}_q)\subset \Stab(\FK{G}_q)$.  The extension $\supp'_{\tN}$ inherits the remaining properties (E2)--(E4) from the corresponding properties for the extensions $\supp'_\lambda$.
\end{proof}

We recall that a specialization closed subset $\Theta$ in $\mathbb{P}(\tN)$ is a subset such that, for any point $x\in \Theta$, the closure $\overline{x}$ is contained in $\Theta$ as well.  In the lavish situation of Lemma \ref{lem:2596}, one can make standard arguments with Rickard idempotents to obtain the following.

\begin{lemma}\label{lem:2632}
Consider $G$ in type $A$.  Then $\tN$-support defines an injective map
\begin{equation}\label{eq:2713}
\begin{array}{c}
\supp_{\tN}:\{\text{\rm thick tensor ideals in }\stab(\FK{G}_q)\}\to \left\{\begin{array}{c}
\text{\rm specialization closed}\\
\text{\rm subsets in }\mbb{P}(\tN)\end{array}\right\}\vspace{1mm}\\
\msc{K}\mapsto \cup_{M\in \msc{K}}\supp_{\tN}(M).
\end{array}
\end{equation}
Indeed, the map
\[
\begin{array}{c}
\msc{K}_?:\left\{\begin{array}{c}
\text{\rm specialization closed}\\
\text{\rm subsets in }\mbb{P}(\tN)\end{array}\right\}\to \{\text{\rm thick tensor ideals in }\stab(\FK{G}_q)\}\\
\Theta\mapsto \msc{K}_{\Theta}
\end{array}
\]
provides a set theoretic section for \eqref{eq:2713}, where $\msc{K}_{\Theta}$ denotes the thick tensor ideal in $\stab(\FK{G}_q)$ consisting of all objects $M$ with $\supp_{\tN}(M)\subset \Theta$.
\end{lemma}

\begin{proof}
One employs the extension promised in Lemma \ref{lem:2596} and proceeds as in \cite[Proof of Theorem 6.3]{friedlanderpevtsova07}, or \cite[Proof of Theorem 7.4.1]{boekujawanakano}, or \cite[Proof of Proposition 5.2]{negronpevtsova3}.
\end{proof}

Note that we have made no claim that the map of \eqref{eq:2713} is surjective.  Technically speaking, this is because we have not shown that any closed subset in $\mbb{P}(\tN)$ is realized as the support $\supp_{\tN}(M)$ of some $M$ in $\stab(\FK{G}_q)$.  However, more directly, we have not claimed that \eqref{eq:2713} is surjective because it probably is not surjective (see the discussion of Section \ref{sect:3286} below).  So Lemma \ref{lem:2632} provides a non-unique classification of thick ideals in $\stab(\FK{G}_q)$ via specialization closed subsets in the projectivization $\mbb{P}(\tN)$.

\subsection{Surjectivity of $f_{\rm univ}$: proof of Proposition \ref{prop:bspec1}}
\label{sect:proof_bspec1}

We arrive finally at the main point of this section.
\par

By Proposition \ref{prop:2588} we have a map $f_{\rm univ}:\mbb{P}(\tN)\to \Spec(\stab\FK{G}_q)$ defined by the support data $\supp_{\tN}$, in type $A$.  So, to prove Proposition \ref{prop:bspec1} it suffices to show that $f_{\rm univ}$ is surjective.

\begin{proof}[Proof of Proposition \ref{prop:bspec1}]
By Lemma \ref{lem:2632}, $\tN$-support provides a non-unique classification of thick ideals in $\stab(\FK{G}_q)$.  Now, the arguments of \cite[Proof of Theorem 5.2]{balmer05} apply directly to see that $f_{\rm univ}:\mbb{P}(\tN)\to \Spec(\stab\FK{G}_q)$ is surjective.
\par

From a different perspective, one can consider the quotient space $\mbb{P}(\tN)/\sim$ by the equivalence relation defined by taking $x\sim y$ whenever any object $M$ with $x\in \supp_{\tN}(M)$ also has $y\in \supp_{\tN}(M)$, and vice versa.  Then $\tN$-support descends to a support data for $\stab(\FK{G}_q)$ which takes values in the quotient $\mbb{P}(\tN)/\sim$ (cf.\ \cite[Definition 2.5]{friedlanderpevtsova05}).  This descended theory is a classifying support data in the sense of \cite[Definition 5.1]{balmer05} so that we get a homeomorphism
\[
\bar{f}_{\rm univ}:\mbb{P}(\tN)/\sim\ \overset{\cong}\to \Spec(\stab\FK{G}_q)
\]
whose precomposition with the projection $\mbb{P}(\tN)\to \mbb{P}(\tN)/\sim$ recovers $f_{\rm univ}$ \cite[Theorem 5.2]{balmer05}.
\end{proof}

\section{The moment map and the Balmer spectrum}
\label{sect:spec2}

In this section we consider the continuous map $\rho:\Spec(\stab\FK{G}_q)\to \mbb{P}(\mcl{N})$ provided in \cite{balmer10}, and analyze the relationship between the map $\rho$ and the universal morphism $f_{\operatorname{univ}}$ of the previous section.  We prove the following.

\begin{theorem}\label{thm:bspec2}
Let $G$ be an almost-simple algebraic group in type $A$.  Then there are two continuous, surjective maps of topological spaces
\[
f_{\rm univ}:\mbb{P}(\tN)\to \Spec(\stab\FK{G}_q),\ \ \ \rho:\Spec(\stab\FK{G}_q)\to \mbb{P}(\mcl{N})
\]
whose composite is the projectivization of the moment map $\kappa:\tN\to \mcl{N}$.  In particular, $f_{\rm univ}$ restricts to a homeomorphism over a dense open subset in the Balmer spectrum
\[
f_{\rm univ,reg}:\mbb{P}(\tN)_{\rm reg}\overset{\cong}\longrightarrow \Spec(\stab\FK{G}_q)_{\rm reg}.
\]
\end{theorem}

Here $\mbb{P}(\tN)_{\rm reg}$ and $\mbb{P}(\mcl{N})_{\rm reg}$ are the projectivizations of the regular loci $\tN_{\rm reg}$ and $\mcl{N}_{\rm reg}$ \cite[\S 3.3]{chrissginzburg09}, respectively, and the corresponding locus $\Spec(\stab\FK{G}_q)_{\rm reg}$ in the Balmer spectrum is the preimage of $\mbb{P}(\mcl{N})_{\rm reg}$ along $\rho$.
\par

We note that in type $A_1$, the regular loci in $\tN$ and $\mcl{N}$ are just the complements of $0$-section $\dG/\dB\subset \tN$ and $\{0\}\subset \mcl{N}$ respectively, so that $\mbb{P}(\tN)_{\rm reg}=\mbb{P}(\tN)$ and $\mbb{P}(\mcl{N})_{\rm reg}=\mbb{P}(\mcl{N})$.  We therefore obtain a precise calculation of the Balmer spectrum in type $A_1$, as a corollary to Theorem \ref{thm:bspec2}.

\begin{corollary}\label{cor:bspecA1}
In type $A_1$ the Balmer spectrum for the small quantum group is precisely the projectivized nilpotent cone,
\[
\rho:\Spec\big(\stab\FK\operatorname{SL}(2)_q\big)\overset{\cong}\to\mbb{P}(\mcl{N})\ \ \text{and}\ \ \rho:\Spec\big(\stab\FK\operatorname{PSL}(2)_q\big)\overset{\cong}\to \mbb{P}(\mcl{N}).
\]
\end{corollary}

\subsection{Factoring the moment map}

By \cite[\S 5]{balmer10} we have the continuous map
\[
\rho:\Spec(D_{\coh}(\FK{G}_q))\to \Proj\Ext_{\FK{G}_q}(\1,\1)\cup\{0\}=\mbb{P}(\mcl{N})\cup\{0\}
\]
defined explicitly by taking
\[
\rho(\msc{P})=Van(\xi\in \Ext^i_{\FK{G}_q}(\1,\1):i> 0,\ \operatorname{cone}(\xi)\notin \msc{P}).
\]
Here $\{0\}$ is the maximal ideal of positive degree elements in $\Ext_{\FK{G}_q}(\1,\1)$, and each $\rho(\msc{P})$ is in fact a homogenous prime in $\Ext_{\FK{G}_q}(\1,\1)$ by \cite[Theorem 5.3]{balmer10}.  We note that the mapping cone $\operatorname{cone}(\xi)$ is stably isomorphic to (a shift of) the Carlson module $L_{\xi}$ for $\xi$ \cite[\S 5.9]{benson91}.
\par

Now, for $G$ in type $A$ we have maps from $\mbb{P}(\tN)$ to the Balmer spectra of the stable and derived categories for $\FK{G}_q$, by Proposition \ref{prop:2588}.  These maps fit into a diagram
\[
\xymatrix{
 & \mbb{P}(\tN)\ar[dl]|{\rm univ\ map}\ar[dr]|{\rm univ\ map}\\
\Spec(\stab\FK{G}_q)\ar[rr] & & \Spec(D_{\coh}\FK{G}_q),
}
\]
where $\Spec(\stab\FK{G}_q)$ is identified with the complement $\Spec(D_{\coh}\FK{G}_q)-\{0\}$ \cite[Proposition 3.11]{balmer05}, and $\{0\}$ is the ideal of acyclic complexes in $D_{\coh}(\FK{G}_q)$.  We can therefore speak unambiguously of the map $f_{\rm univ}$ to the spectrum of either the small stable, or coherent derived category for $\FK{G}_q$.

\begin{proposition}\label{prop:2720}
Consider $G$ in type $A$.  The composition $\rho\circ f_{\rm univ}:\mbb{P}(\tN)\to \mbb{P}(\mcl{N})\cup \{0\}$ has image in $\mbb{P}(\mcl{N})$ and is equal to the projectivized moment map,
\[
\rho\circ f_{\rm univ}=\mbb{P}(\kappa):\mbb{P}(\tN)\to \mbb{P}(\mcl{N}).
\]
\end{proposition}

\begin{proof}
For concreteness we consider $f_{\rm univ}$ as a map to $\Spec(D_{\coh}(\FK{G}_q))$, and $\rho$ is a map from this spectrum.  For $x\in \mbb{P}(\tN)$ we represent $x$ as a geometric point $x:\Spec(K[t,t^{-1}])\to \tN$ with corresponding point $\lambda:\Spec(K)\to \dG/\dB$.  We have a unique restriction of $x$ to a point $x:\Spec(K[t,t^{-1}])\to \mfk{n}_\lambda$.  By an abuse of notation we let $x\in \mbb{P}(\mfk{n}_\lambda)$ denote the corresponding (topological) point in the projectivization.
\par

We have
\[
\begin{array}{rl}
f_{\rm univ}(x)&=\{M\in D_{\coh}(\FK{G}_q):x\notin \supp_{\tN}(M)\}\vspace{1mm}\\
&=\{M\in D_{\coh}(\FK{G}_q):x\notin \supp^{\operatorname{chom}}_\lambda(\res_\lambda M)\}
\end{array}
\]
via the naturality result of Theorem \ref{thm:natrl}, and thus
\[
\begin{array}{rl}
\rho f_{\rm univ}(x)&=\{\xi\in \Ext^{>0}_{\FK{G}_q}(\1,\1):\operatorname{cone}(\xi)\notin f_{\rm univ}(x)\}\vspace{1mm}\\
&=\{\xi:x\in \supp^{\operatorname{chom}}_\lambda(\res_\lambda\operatorname{cone}(\xi))\}.
\end{array}
\]
Since restriction is an exact functor, we have $\res_\lambda\operatorname{cone}(\xi)=\operatorname{cone}(\res_\lambda \xi)$, where now $\res_\lambda(\xi)\in \Ext_{\msc{B}_\lambda}(\1,\1)$.  But now, the support of the mapping cone of some function in $\Ext_{\msc{B}_\lambda}(\1,\1)$ is just the vanishing locus
\[
\supp^{\operatorname{chom}}_\lambda(\operatorname{cone}(\res_\lambda \xi))=Van(\res_\lambda \xi)\subset \Proj\Ext_{\msc{B}_\lambda}(\1,\1)=\mbb{P}(\mfk{n}_\lambda).
\]
Since restriction
\[
\res_\lambda:\O(\mcl{N})=\Ext_{\FK{G}_q}(\1,\1)\to \Ext_{\msc{B}_\lambda}(\1,\1)=\O(\mfk{n}_\lambda)
\]
is dual to the map $j_\lambda:\mfk{n}_\lambda\hookrightarrow \mcl{N}_K\to \mcl{N}$ we have that
\[
x\in Van(\res_\lambda \xi)\ \ \Leftrightarrow\ \ j_\lambda(x)\in Van(\xi).
\]
So we have a sequence of identifications
\[
\begin{array}{rl}
\rho f_{\rm univ}(x)&=\{\xi:x\in \supp^{\operatorname{chom}}_\lambda(\res_\lambda\operatorname{cone}(\xi))\}\vspace{1mm}\\
&=\{\xi:x\in \supp^{\operatorname{chom}}_\lambda(\operatorname{cone}(\res_\lambda \xi))\}\vspace{1mm}\\
&=\{\xi:x\in Van(\res_\lambda \xi)\}\vspace{1mm}\\
&=\{\xi:j_\lambda(x)\in Van(\xi)\}=j_\lambda(x).
\end{array}
\]
This calculation also shows that $\rho f_{\rm univ}(x)$ is not the irrelevant ideal $\{0\}$, so that $\rho\circ f_{\rm univ}$ does in fact have image in $\Proj\Ext_{\FK{G}_q}(\1,\1)=\mbb{P}(\mcl{N})$.
\par

By the above information we see that the diagram of maps of topological spaces
\[
\xymatrix{
\mbb{P}(\tN)\ar[rr]^{\rho\circ f_{\rm univ}} & & \mbb{P}(\mcl{N})\\
 & \mbb{P}(\mfk{n}_\lambda) \ar[ur]_{\mbb{P}(j_\lambda)}\ar[ul]^{\mbb{P}(i_\lambda)}
}
\]
commutes when restricted to closed points in $\mbb{P}(\mfk{n}_\lambda)$, for all $\lambda:\Spec(K)\to \dG/\dB$.  By considering further base change to $(\mfk{n}_{\lambda})_L=\mfk{n}_{\lambda'}$ for any extension $K\to L$, it follows that the above diagram commutes on the entirety of $\mbb{P}(\mfk{n}_\lambda)$.  But now
\[
\rho f_{\rm univ}\big(\mbb{P}(i_\lambda)(x)\big)=\mbb{P}j_\lambda(x)=\mbb{P}(\kappa)\big(\mbb{P}(i_\lambda)(x)\big)\ \ \text{for all}\ x\in \mfk{n}_\lambda,
\]
where $\kappa:\tN\to \mcl{N}$ is the moment map.  Since all points in $\mbb{P}(\tN)$ are of the form $\mbb{P}(i_\lambda)(x)$ for some $\lambda$, we see $\rho\circ f_{\rm univ}=\mbb{P}(\kappa)$.
\end{proof}

\subsection{Proof of Theorem \ref{thm:bspec2}}

\begin{proof}[Proof of Theorem \ref{thm:bspec2}]
By Proposition \ref{prop:bspec1} the map $f_{\rm univ}:\mbb{P}(\tN)\to \Spec(\stab\FK{G}_q)$ is surjective, and by Proposition \ref{prop:2720} the composition
\[
\rho\circ f_{\rm univ}:\mbb{P}(\tN)\to \Spec(\stab\FK{G}_q)\to \mbb{P}(\mcl{N})
\]
is the projectivized moment map.  Since the moment map is an isomorphism over the regular loci in $\tN$ and $\mcl{N}$, it follows that $f_{\rm univ}$ restricts to an injective map over $\mbb{P}(\tN)_{\rm reg}$, and is hence a bijection over $\mbb{P}(\tN)_{\rm reg}$.  It follows that $\rho$ restricts to a bijection over $\Spec(\stab\FK{G}_q)_{\rm reg}$.  Since all of these maps are continuous, we find that the restriction of $f_{\rm univ}$ to the regular loci
\[
\mbb{P}(\tN)_{\rm reg}\to \Spec(\stab\FK{G}_q)_{\rm reg}
\]
is a heomeomorphism with inverse
\[
\Spec(\stab\FK{G}_q)_{\rm reg}\to \mbb{P}(\mcl{N})_{\rm reg}
\]
provided by the restriction of $\rho$.
\par

To see that the regular locus $\Spec(\stab\FK{G}_q)_{\rm reg}$ is dense in the spectrum, we simply consider its closure $\Theta\subset \Spec(\stab\FK{G}_q)$ and have that the preimage $f_{\rm univ}^{-1}\Theta\subset \mbb{P}(\tN)$ is a closed set which contains the regular locus.  Since the regular locus is dense in $\mbb{P}(\tN)$ we conclude that $f_{\rm univ}^{-1}\Theta=\mbb{P}(\tN)$ and surjectivity of $f_{\rm univ}$ implies $\Theta=\Spec(\stab\FK{G}_q)$.  We are done.
\end{proof}

\subsection{A conditional statement for arbitrary almost-simple $G$}
\label{sect:the_others}

We take a moment to speak to a version of Theorem \ref{thm:bspec2} in arbitrary Dynkin type.  The following was conjectured in \cite{negronpevtsova,negronpevtsova3}.

\begin{conjecture}[Standard conjecture]\label{conj:SC}
For an arbitrary almost-simple algebraic group $G$, and an arbitrary geometric point $\lambda:\Spec(K)\to \dG/\dB$,
\begin{enumerate}
\item[(A)] Cohomological support for $\stab(\msc{B}_\lambda)$ satisfies the tensor product property.
\item[(A+)] Cohomological support for $\stab(\msc{B}_\lambda)$ is a lavish support theory (see \S \ref{sect:supp_nons}).
\end{enumerate}
\end{conjecture}

Theorem \ref{thm:np} says that both Conjecture \ref{conj:SC} (A) and (A+) hold in type $A$, and we do believe that support for the small quantum Borels satisfies these conditions in arbitrary Dynkin type as well.  We have the following conditional extension of Theorem \ref{thm:bspec2}.

\begin{theorem}\label{thm:bspec2.5}
Suppose that the \emph{Standard Conjecture \ref{conj:SC}} holds for the small quantum Borel, in arbitrary Dynkin type, and let $G$ be any almost-simple algebraic group.  Then we have a surjective continuous map
\[
f_{\rm univ}:\mbb{P}(\tN)\to \Spec(\stab\FK{G}_q)
\]
which is a homeomorphism over a dense open subset in $\Spec(\stab\FK{G}_q)$.
\end{theorem}

This dense open subset is, of course, the regular locus in $\Spec(\stab\FK{G}_q)$, i.e.\ the preimage of $\mbb{P}(\mcl{N})_{\rm reg}$ along the map $\rho$.

\begin{proof}
Literally just copy the proofs of Proposition \ref{prop:bspec1} and Theorem \ref{thm:bspec2}, replacing Theorem \ref{thm:np} with Conjecture \ref{conj:SC} when necessary.
\end{proof}

\begin{remark}
The conclusions (A) and (A+) have, more-or-less, been proposed to hold in arbitrary Dynkin type in \cite{boekujawanakano,nakanovashawyakimov}.  Specifically, the conclusions of Theorem \ref{thm:bspec2.5} \emph{do hold} if one accepts the validity of the specific claims \cite[Theorem 6.2.1, Theorem 6.5.1, and proof of Theorem 7.4.1]{boekujawanakano}.  However, some points in \cite{boekujawanakano} have been difficult to clarify, at least from the authors' considerations of the text.  For this reason we have restricted some of our ``tt-geometry" claims to type $A$, where we can specifically employ Theorem \ref{thm:np}.  (See \cite[Section 8.7]{negronpevtsova3} for further discussion.)
\end{remark}

\subsection{The Balmer spectrum is most likely just $\mbb{P}(\mcl{N})$}
\label{sect:3286}

As mentioned previously, we do believe that the Standard Conjecture \ref{conj:SC} holds in arbitrary Dynkin type.  So let us take this point for granted for the moment.  Then we have the maps
\[
\mbb{P}(\tN)\to\Spec(\stab\FK{G}_q)\to \mbb{P}(\mcl{N})
\]
which approximate the Balmer spectrum via the nilpotent cone at arbitrary $G$.  The difference between $\tN$ and $\mcl{N}$ is of course the fact that the nilpotent subalgebras $\mfk{n}_\lambda$ are disjoint in $\tN$, while they generally intersect in $\mcl{N}$.  So, inside of the Balmer spectrum, one is asking about the separation of the images
\[
\res_\lambda^\ast:\Spec(\stab\msc{B}_\lambda)\to \Spec(\stab\FK{G}_q)
\]
in $\Spec(\stab\FK{G}_q)$, where the spectrum for $\msc{B}_\lambda$ is just the spectrum of two-sided primes in $\msc{B}_\lambda$ (cf.\ \cite{nakanovashawyakimov1,negronpevtsova3}).
\par

Our inability to access the intersections of these images in the Balmer spectrum is related to our inability to produce a categorical realization $``\msc{B}_\lambda\cap\msc{B}_\mu"$ of the intersection $\mfk{b}_\lambda\cap \mfk{b}_\mu$ at arbitrary $\lambda$ and $\mu$.  Indeed, one can see that $\tN$-support collapses to the cohomological $\mcl{N}$-support in the modular setting, for $\mbb{G}_{(r)}$, simply because we can realize the corresponding intersections $``\msc{B}_\lambda\cap\msc{B}_\mu"$ via the literal intersections $\Rep(\mbb{B}_{\lambda}\cap\mbb{B}_{\mu})_{(r)}$ of Borels in $\mbb{G}$.  The existence of such an intersection (and corresponding analysis of support) tells you that the map
\[
f_{\rm univ}:\mbb{P}(\tN_{r})\to \Spec(\stab\mbb{G}_{(r)})
\]
factors through the affinization to provide a homeomorphism
\[
\bar{f}_{\rm univ}:\mbb{P}(\mcl{N}_{r})\overset{\cong}\longrightarrow \Spec(\stab\mbb{G}_{(r)}),
\]
where $\tN_{r}=\Spec_{\dG/\dB}\left(H^\ast(\sRHom_{\mbb{G}_{(r)}}(\1,\1))\right)$ and $\mcl{N}_{r}=\Spec(\Ext_{\mbb{G}_{(r)}}(\1,\1))$.  

\begin{remark}
What we have suggested above is a somewhat different approach to support for infinitesimal group schemes than was taken in \cite{suslinfriedlanderbendel97b,friedlanderpevtsova07}.  An elaboration on such an approach to support in the modular setting may appear in a later text.
\end{remark}

Regardless of our inability to realize these quantum intersections $``\msc{B}_\lambda\cap\msc{B}_\mu"$ appropriately, at least at the present moment, we do expect that $\tN$-support for the quantum group collapses to the nilpotent cone, just as it does in the modular setting.  We therefore expect the corresponding calculation
\begin{equation}\label{eq:3211}
\xymatrix@C=18pt{
\bar{f}_{\rm univ}:\mbb{P}(\mcl{N})\ar[rr]^(.45){\cong}_(.45){\rm Expected} & & \Spec(\stab\FK{G}_q)
}
\end{equation}
in arbitrary Dynkin type.  We record some specific information in this direction, without providing further commentary.

\begin{proposition}\label{prop:equiv}
Suppose that the {\rm Standard Conjecture \ref{conj:SC}} holds for the small quantum Borel in arbitrary Dynkin type.  Then for the quantum Frobenius kernel $\FK{G}_q$ the following are equivalent:
\begin{enumerate}
\item[(a)] The canonical map $\rho:\Spec(\stab\FK{G}_q)\to \mbb{P}(\mcl{N})$ is a homeomorphism.
\item[(b)] The $\tN$-support of any coherent sheaf in $\FK{G}_q$ is simply the preimage of cohomological support along the moment map,
\[
\supp_{\tN}(M)=\mbb{P}(\kappa)^{-1}\supp^{\operatorname{chom}}_{\mcl{N}}(M).
\]
\item[(c)] Cohomological support for $\FK{G}_q$ satisfies the tensor product property.
\item[(d)] Naturality holds for cohomological support.  Specifically, for any closed point $\lambda:\Spec(k)\to \dG/\dB$ and coherent $M$ in $\FK{G}_q$ we have
\[
\supp^{\operatorname{chom}}_{\lambda}(\res_\lambda M)=\supp^{\operatorname{chom}}_{\mcl{N}}(M)\cap\mathbb{P}(\mfk{n}_\lambda).
\]
\end{enumerate}
\end{proposition}

We leave the proofs to the interested reader.  With regards to point (d), it has been shown in work of Boe, Kujawa, and Nakano that such naturality holds for representations of the big quantum group $G_q$, i.e.\ $\dG$-equivariant objects in $\FK{G}_q$.

\begin{theorem}[{\cite[Theorem 7.3.3]{boekujawanakano}}]
Suppose that $M$ in $D_{\coh}(\FK{G}_q)$ is in the image of the de-equivariantization map $E_-:D_{\coh}(G_q)\to D_{\coh}(\FK{G}_q)$.  Then for any closed point $\lambda:\Spec(k)\to \dG/\dB$ cohomological support for $M$ satisfies
\[
\supp^{\operatorname{chom}}_{\lambda}(\res_\lambda M)=\supp^{\operatorname{chom}}_{\mcl{N}}(M)\cap \mathbb{P}(\mfk{n}_\lambda).
\]
\end{theorem}

\begin{remark}
To say that $\tN$-support for objects is, essentially, just cohomological support over $\mcl{N}$ is not to say that the enhancement $D_{\coh}^{\Enh}(\FK{G}_q)$ of $D_{\coh}(\FK{G}_q)$ contains no additional information about the (tensor) category $\FK{G}_q$.  For example, the proposed equality $\supp_{\tN}(M)=\mbb{P}(\kappa)^{-1}\supp_{\mcl{N}}(M)$ does not imply that the sheaf $\sExt_{\FK{G}_q}(M,M)$ is isomorphic to the pullback $\kappa^\ast(\Ext_{\FK{G}_q}(M,M)^\sim)$, nor does it provide information about how the functor $\sExt_{\FK{G}_q}(M,-)$ varies within the category of sheaves on the Springer resolution.  One can see the discussions of Section \ref{sect:GRT} as a further elaboration of these points.
\end{remark}


\part{Geometric Representation Theory and Logarithmic TQFT}

In Part II of the text we have argued that the enhancement $D^{\Enh}(\FK{G}_q)$ of the derived category of quantum group representations $D(\FK{G}_q)$ (Theorem \ref{thm:enhA}) can be employed effectively in an analysis of support for the quantum Frobenius kernel.  In this final part of the paper we explain how the enhancement $D^{\Enh}(\FK{G}_q)$ might similarly contribute to our understandings of geometric representation theory, and of topological quantum field theories which one might associate to the small quantum group.  We also argue that, at least to some degree, these varied topics should be bridged by an analysis of quantum group representations via sheaves on the half-quantum flag variety, and a corresponding enhancements of $D(\FK{G}_q)$ not only in the category of sheaves over $\dG/\dB$, but in sheaves over (some $E_n$-twisting of) the Springer resolution as well (cf.\ Proposition \ref{prop:1749} below).
\par

All of the materials of this Part of the text are speculative, and should be enjoyed at one's leisure.  However, we \emph{have} tried to make our proposals below quite precise.  The main points of interest are Conjectures \ref{conj:formality} and \ref{conj:GRT}, and the framing of Section \ref{sect:sub_fin}.

\section{Remarks on geometric representation theory}
\label{sect:GRT}

We comment on the enhanced derived category $D^{\Enh}(\FK{G}_q)$ in relation to geometric representation theory, and explain how geometric representation theory and support theory are, conjecturally, bridged by the framework presented in Part I of this text.  We provide an infographic regarding the latter point in Section \ref{sect:conclusion} below.

\subsection{Geometric representation theory via two theorems}

Let us explain what we mean by ``geometric representation theory" in two theorems.  In the statement below we take $D_{\rm fin}(G_q)_0$ to be the principal block in the bounded derived category of finite-dimensional quantum group representations.  Recall that $\FK{G}_q$ denotes the category of small quantum group representations, which we (re)construct via equivariant sheaves on the dual group $\dG$ of Section \ref{sect:Gq} (see Theorem \ref{thm:ag}).  For simply-connected $G$ this dual group is just the adjoint form $\dG=G/Z(G)$.

\begin{theorem}[{\cite[Theorem 3.9.6]{arkhipovbezrukavnikovginzburg04}}]
Suppose $G$ is simply-connected.  There is an equivalence of triangulated categories $F^{\dG}:D_{\rm fin}(G_q)_0\overset{\sim}\to D_{\coh}(\tN)^{\dG}$ to the derived category of $\dG$-equivariant coherent dg sheaves on the Springer resolution.  The fiber of this equivalence along the inclusion $i_1:\mfk{n}\to \tN$ recovers derived morphisms over the Borel
\[
\operatorname{L}i_1^\ast F^{\dG}(V)=\RHom_{\uqB}(k,V),
\]
and invariant global sections recover maps over the quantum group
\[
\operatorname{R}\Gamma(\tN,F^{\dG}(V))^{\dG}=\RHom_{G_q}(k,V).
\]
\end{theorem}

This theorem was extended to the category of representations for the small quantum group in work of Bezrukavnikov and Lachowska \cite{bezrukavnikovlachowska07}.  Recall that, in terms of our presentation of the small quantum group from Theorem \ref{thm:ag}, the restriction map from $\Rep G_q$ is identified with the de-equivariantization/vector bundle map $E_-:\Rep G_q\to \FK{G}_q$.

\begin{theorem}[{\cite[Theorem 4, \S 2.4]{bezrukavnikovlachowska07}}]
For simply-connected $G$, there is an equivalence of triangulated categories $F:D_{\coh}(\FK{G}_q)_0\overset{\sim}\to D_{\coh}(\tN)$.  Global sections of $F$ recover maps over the small quantum group
\[
\operatorname{R}\Gamma(\tN,F(M))=\RHom_{\FK{G}_q}(\1,M).
\]
This functor fits into the appropriate diagram over its equivariant analog, so that $F\circ E_-\cong forget\circ F^{\dG}$, where $forget:D_{\coh}(\tN)^G\to D_{\coh}(\tN)$ forgets the equivariant structure.
\end{theorem}

There are parallel theorems which address arbitrary blocks $D_{\coh}(\FK{G}_q)_{\chi}$, see for example \cite{lachowskaqi19}.  The above two theorems are also somewhat amenable to computation, see for example \cite{bezrukavnikov06}.
\par

Of course, these results, and the subsequent utilization of the affine Grassmannian in analyses of quantum group representations \cite{arkhipovbezrukavnikovginzburg04} have had a significant impact in representation theory in general.  However they do force one to deal independently with each block in the category $D_{\coh}(\FK{G}_q)$.  Since these blocks are not themselves tensor subcategories in $D_{\coh}(\FK{G}_q)$, nor do they interact in a predictable manner under fusion, one cannot easily speak to the geometric interpretations of the quantum group provided in \cite{arkhipovbezrukavnikovginzburg04,bezrukavnikovlachowska07,lachowskaqi19,bezrukavnikov06} while simultaneously speaking to the tensor structure.
\par

From other perspectives however, say from support theory or from the perspective of certain field theories in mathematical physics, one certainly \emph{does} see the monoidal category $D_{\coh}(\FK{G}_q)$ as ``living over" a certain geometric object, and this geometric object is commonly approximated by the nilpotent cone (see our current analysis as well as \cite{costellocreutziggaitto19,dedushenkogukovnakajimapei20,gukovetal21,creutzigdimoftegarnergeer}).  With these points in mind we suggest a conjectural means of recovering the functors of \cite{arkhipovbezrukavnikovginzburg04,lachowskaqi19} from the tensor categorical analyses for the quantum group provided in Parts 1 and 2 of this text.

\subsection{Some preliminary words}

Recall that, via the embedding result of Theorem \ref{thm:Kempf}, we have a canonical monoidal embedding $\operatorname{Kempf}:D(\FK{G}_q)\to D(\dG/B_q)$ from the derived category of small quantum group representations to the derived category of (quasi-coherent) sheaves on the half-quantum flag variety.  Recall also that sheaves on $\dG/B_q$ are, by definition, simply $B_q$-equivariant sheaves on the dual group $\dG$ (see Section \ref{sect:G/Bq}).
\par

At this point we want to make a relatively straightforward claim.  Namely, we propose that the derived category $D(\dG/B_q)$ of sheaves on the half-quantum flag variety is itself a sheaf of tensor categories over the Springer resolution, and that fundamental information from geometric representation theory can be recovered from this global tensor categorical perspective.  Such a sheaf structure for the half-quantum flag variety is determined by an action of $D(\tN)$ on $D(\dG/B_q)$ (cf.\ \cite{gaitsgory15}).
\par

In making such claims there are two important points to consider: First, in order to even articulate our statements correctly the Springer resolution must be understood properly as a dg scheme and one must work in a dg, or linear infinity setting.  Second, due to the non-symmetric nature of $D(\dG/B_q)$, one does not expect an action of the symmetric category $D(\tN)$ per-se, but of some braided degeneration of $D(\tN)$ where the braided structure is dictated by the $R$-matrix for the quantum group.  In addition to these generic points, there is a third more specific point to consider.  Namely, in realizing the above outline one also runs into a--somewhat familiar--formality problem for derived morphisms over the quantum group (see Conjecture \ref{conj:formality} below).  
\par

The point of this section is to explain the above comments in detail, and also to provide a precise relation between these comments and the results of \cite{arkhipovbezrukavnikovginzburg04,bezrukavnikovlachowska07} which we've recalled above.  We also discuss subsequent connections between support theory and geometric  representation theory in Section \ref{sect:conclusion} below.

\subsection{Infinity categories}
\label{sect:infinity}

We describe how the materials of Part I of the text, as well as Section \ref{sect:perf_shf}, should lift to the infinity categorical setting.  Formally speaking, all of the materials of this section are conjectural.
\par

We consider again the half-quantum flag variety $\dG/B_q$ of Sections \ref{sect:G/Bq}.  Let $\QCoh_{\dg}(\dG/\dB)$ and $\QCoh_{\dg}(\dG/B_q)$ denote presentable, monoidal, stable infinity categories whose homotopy categories recover the unbounded derived categories $D(\dG/\dB)$ and $D(\dG/B_q)$, respectively.  We assume that the entire presentation of Section \ref{sect:D_Enh} also lifts to an infinity categorical level.  In particular, we have the central pullback functor $\zeta^\ast:\QCoh_{\dg}(\dG/\dB)\to \QCoh_{\dg}(\dG/B_q)$ which provides an action of $\QCoh_{dg}(\dG/\dB)$ on $\QCoh_{dg}(\dG/B_q)$.
\par

As stated above, the pullback $\zeta^\ast$ provides a symmetric action
\[
\star:\QCoh_{\dg}(\dG/\dB)\ot_k \QCoh_{\dg}(\dG/B_q)\to \QCoh_{\dg}(\dG/\dB)
\]
with commutes with colimits, and taking adjoints provides inner-Homs $\sRHom_{\dG/B_q}$.  These inner-Homs provide an enhancement $\QCoh_{\dg}^{\Enh}(\dG/B_q)$ of $\QCoh_{\dg}(\dG/B_q)$ in the symmetric monoidal infinity category of sheaves on the flag variety.  In particular, we have an identification of monoidal infinity categories
\[
\operatorname{R}\Gamma(\dG/\dB,\QCoh^{\Enh}_{\dg}(\dG/B_q))\cong \QCoh_{\dg}(\dG/B_q),
\]
by the same reasoning employed in the proofs of Theorem \ref{thm:enhA} and Proposition \ref{prop:D_Enh} above.

\subsection{A formality conjecture}
\label{sect:formality}

Let us continue in the infinity categorical setting outlined above.  By a higher version of Deligne's conjecture \cite[\S 5.3]{lurieHA} the derived endomorphisms of the unit $\msc{R}(G_q)=\sRHom_{\dG/B_q}(\1,\1)$ in $\QCoh^{\Enh}_{\dg}(\dG/B_q)$ should admit a canonical $E_3$-algebra structure, via the tensor product, opposite tensor product, and composition (see \cite[Theorem 5.1.2.2]{lurieHA}).  For $M$ and $N$ in $\QCoh_{\dg}(\dG/B_q)$ the binatural tensor actions and composition action of $\msc{R}(G_q)$ on $\sRHom_{\dG/B_q}(M,N)$ should give the morphisms $\sRHom_{\dG/B_q}(M,N)$ the structure of an $E_2$-module over $\msc{R}(G_q)$, and for $M$ and $N$ in the (full) subcategory of quantum group representations $\QCoh_{\dg}(\dG/G_q)$ this $E_2$-structure should lift to an $E_3$-structure.
\par

Now, at the infinity categorical level we \emph{can} speak of the symmetric monoidal category $\msc{R}(G_q)\text{-Mod}$ of $\msc{R}(G_q)$-modules in $\QCoh_{\dg}(\dG/\dB)$, with product given by $\ot=\ot_{\msc{R}(G_q)}$ (cf.\ Section \ref{sect:RGq_act}).  The category $\QCoh^{\Enh}_{\dg}(\dG/B_q)$ is then seen to be enriched, further, in the monoidal category of $\msc{R}(G_q)$-modules.  We let $\msc{X}$ denote the $E_3$-scheme over $\dG/\dB$ associated to $\msc{R}(G_q)$, so that $\QCoh_{\dg}(\msc{X})\cong \msc{R}(G_q)\text{-Mod}$ (see \cite{francis08}).

\begin{remark}
For one direct algebraic account of the $E_2$-structure on $\msc{R}(G_q)$ one can see \cite[\S 3.2]{schweigertwoike}.
\end{remark}

At this point we have our category $\QCoh^{\Enh}_{\dg}(\dG/B_q)$ which is now enriched in the monoidal category $\QCoh_{\dg}(\msc{X})$ of sheaves on our $E_3$-scheme $\msc{X}$, and we essentially ``go backwards" from this enrichment to obtain a tensor action.  Specifically, one expects the enhanced morphisms $\sRHom_{\dG/B_q}(M,-)$, considered as sheaves over $\msc{X}$, to commute with limits and thus have left adjoints.  These left adjoints provide a tensor categorical action
\[
\star:\QCoh_{\dg}(\msc{X})\ot_k\QCoh_{\dg}(\dG/B_q)\to \QCoh_{\dg}(\dG/B_q)
\]
which extends the action of $\QCoh_{\dg}(\dG/\dB)$.  Via the action on the unit, $\star$ restricts to, and is defined by, a central tensor embedding
\[
\xi^\ast:\QCoh_{\dg}(\msc{X})\to \QCoh_{\dg}(\dG/B_q).
\]
In this way we expect that $\QCoh_{\dg}(\dG/B_q)$ admits the natural structure of a sheaf of tensor categories over $\msc{X}$.

\begin{conjecture}[Formality]\label{conj:formality}
The $E_3$-scheme $\msc{X}$ is $E_2$-formal, so that we have a monoidal equivalence
\[
\QCoh_{\dg}(\tN)\overset{\sim}\to \QCoh_{\dg}(\msc{X}).
\]
In this way, $\QCoh_{\dg}(\dG/B_q)$ has the structure of a sheaf of tensor categories over (a twisted version of) the Springer resolution.
\end{conjecture}

To elaborate, the $E_3$-structure on $\msc{X}$ corresponds to a higher braided structure on the category of sheaves $\QCoh(\msc{X})$, and is reflected in an induced $2$-shifted Poisson structure $\omega$ on the cohomology $\tN=H^\ast(\msc{X})$ \cite[Theorem 9.1]{robertswillerton10} \cite[Theorem 3.5.4]{calaqueetal17}.  So our conjecture is that there is some braided perturbation $\QCoh_{\dg}^{\omega}(\tN)$ of the symmetric monoidal infinity category $\QCoh_{\dg}(\tN)$ under which $\QCoh_{\dg}(\dG/B_q)$ has the structure of a sheaf of tensor categories over $\tN$, considered now as an $E_3$-scheme.
\par

The main point of Conjecture \ref{conj:formality} is not so much the number of $E$'s which appear in the statement, but that the type of formality we are suggesting provides an equivalence $\QCoh_{\dg}(\tN)\simeq\QCoh_{\dg}(\msc{X})$ of \emph{monoidal} infinity categories, not just stable infinity categories.  We note however, that for many representation theoretic applications of Conjecture \ref{conj:formality}, $E_1$-formality will suffice.

\subsection{An accounting of formality}

The Formality Conjecture \ref{conj:formality} can be reduced to the existence of an $E_2$-formality
\begin{equation}\label{eq:2989}
\O(\mfk{n})=\Sym(\mfk{n}^\ast)\overset{\sim}\to \REnd_{\uqB}(\1,\1)
\end{equation}
for the derived endomorphisms over the quantum Borel, calculated in the symmetric monoidal infinity category $\Rep_{\dg}(\dB)$ of $\dB$-representations.  Such a formality would imply formality of the product
\[
\O_{\dG}\ot_k \Sym(\mfk{n}^\ast)\overset{\sim}\to \sRHom_{\dG/\su}(\1,\1)
\]
in the category of $\dB$-equivariant sheaves on $\dG$, and hence the desired formality $p_\ast\O_{\tN}\overset{\sim}\to \sRHom_{\dG/B_q}(\1,\1)$ over $\dG/\dB$, by descent.
\par

If we consider only $E_1$-formality for $\REnd_{\uqB}(\1,\1)$, which corresponds to considering a non-monoidal equivalence of infinity categories $\QCoh_{\dg}(\tN)\overset{\sim}\to \QCoh(\msc{X})$, such $E_1$-formality is \emph{almost} stated explicitly in the original work of Arkhipov-Bezrukavnikov-Ginzburg \cite{arkhipovbezrukavnikovginzburg04}.

\begin{theorem}[{\cite[Theorem 3.7.5]{arkhipovbezrukavnikovginzburg04}}]
Let $t:\Rep_{\dg}(B)\to \Rep_{\dg}(\mfk{b})$ be the forgetful functor.  There is an equivalence of $E_1$-algebras $t\Sym(\mfk{n}^\ast)\overset{\sim}\to t\REnd_{\uqB}(\1,\1)$.
\end{theorem}

\begin{remark}
What's actually proved in \cite{arkhipovbezrukavnikovginzburg04} is a slightly stronger result which does address certain rationality properties of the given $E_1$-formality map (see e.g.\ \cite[Proposition 3.9.2]{arkhipovbezrukavnikovginzburg04}).  Indeed, we are tempted to suggest that a careful reading of \cite{arkhipovbezrukavnikovginzburg04} will simply yield $E_1$-formality of the algebra $\msc{R}(G_q)$.  As another point, one can deduce from \cite{bezrukavnikovlachowska07} that the global sections
\[
\operatorname{R}\Gamma(\dG/\dB,\sRHom_{\dG/B_q}(\1,\1))=\RHom_{\FK{G}_q}(\1,\1)
\]
are also $E_1$-formal.
\end{remark}

\subsection{GRT versus tensor enhancements}
\label{sect:small_GRT}

We let $\mathbf{DG}(\FK{G}_q)$ denote the infinity category of unbounded dg sheaves for the small quantum group, and let $\operatorname{DG}(\FK{G}_q)$ denote the full monoidal infinity subcategory of coherent dg sheaves.  We recall again that, by definition, $\FK{G}_q=\QCoh(\dG/G_q)$ and we have the monoidal equivalence $1^\ast:\QCoh(\dG/G_q)\overset{\sim}\to \Rep \uqG$ of Theorem \ref{thm:ag}.  So the homotopy categories for $\mathbf{DG}(\FK{G}_q)$ and $\operatorname{DG}(\FK{G}_q)$ are equivalent to the unbounded derived category of small quantum group representation, and the bounded derived category of finite-dimensional representations respectively.

\begin{conjecture}\label{conj:GRT}
Suppose that the Formality Conjecture \ref{conj:formality} holds.  Then the (non-monoidal) pushforward functor
\[
\xi_\ast:=\sRHom_{\dG/B_q}(\1,-):\QCoh_{\dg}(\dG/B_q)\to \QCoh_{\dg}(\tN)
\]
restricts to an equivalence on the principal block of coherent dg sheaves for the small quantum group
\[
\xi_\ast|_{\operatorname{block}_0}:\operatorname{DG}(\FK{G}_q)_0\overset{\sim}\longrightarrow \Coh_{\dg}(\tN),
\]
via the Kempf embedding (Theorem \ref{thm:Kempf}).  This equivalence can be identified with the equivalence of \cite{arkhipovbezrukavnikovginzburg04,bezrukavnikovlachowska07}.
\end{conjecture}

We remark that one only needs to establish $E_1$-formality of the algebra $\msc{R}(G_q)$ for Conjecture \ref{conj:GRT} to make sense, since the statement is only concerned with the objects $\sRHom_{\dG/B_q}(\1,M)$ as sheaves over $\msc{X}$ or $\tN$.

\begin{remark}
Of course, $\xi_\ast$ also induces a functor from $\mathbf{DG}(\FK{G}_q)_0$ to $\QCoh_{\dg}(\tN)$, but this functor should not be an equivalence, simply because $\xi_\ast$ will not commute with infinite sums.  It seems reasonable to expect that $\xi_\ast$ induces an equivalence from Ind-coherent sheaves $\operatorname{IndCoh}_{\dg}(\dG/G_q)_0$ to $\QCoh(\tN)$ however.
\end{remark}

The above reframing of \cite{arkhipovbezrukavnikovginzburg04} is, in a sense, relatively tame.  Specifically, it is already explained in \cite[\S 1.4]{arkhipovbezrukavnikovginzburg04} that their equivalence is essentially the functor
\[
F^G:V\mapsto \text{descent of }\O_{\dG}\ot_k\RHom_{\uqB}(k,V)=\sRHom_{\dG/B_q}(\1,E_V).
\]
When compared with \cite{bezrukavnikovlachowska07} there is a somewhat more substantial gap between our suggestion above and the construction of \cite[\S 2.4]{bezrukavnikovlachowska07}, since Bezrukavnikov and Lachowska approach their equivalence through work of Backelin and Kremnizer \cite{backelinkremnizer08} which concerns the De-Concini Kac quantum group.\footnote{As far as we understand, there are some technical issues with the work \cite{backelinkremnizer08}.  See \cite[Remark 5.4]{tanisaki14}.  These issues should be clarified in type $A$ in the relatively near future. See the introduction to \cite{tanisaki21}.}

\subsection{Conclusions}
\label{sect:conclusion}

We have suggested above that the equivalences of \cite{arkhipovbezrukavnikovginzburg04,bezrukavnikovlachowska07} should be reflections of a more basic interaction between quantum groups and sheaves on the Springer resolution.  In this more fundamental setting one \emph{is} able to address fusion of representations, via the tensor structure on the half-quantum flag variety.  Also, as we explained in Part II of this work, support theory is yet another reflection of this geometric framing for quantum group representations.  In this way, the Formality Conjecture \ref{conj:formality} offers an explicit means of integrating support theory and geometric representation theory under a single uniform framework:

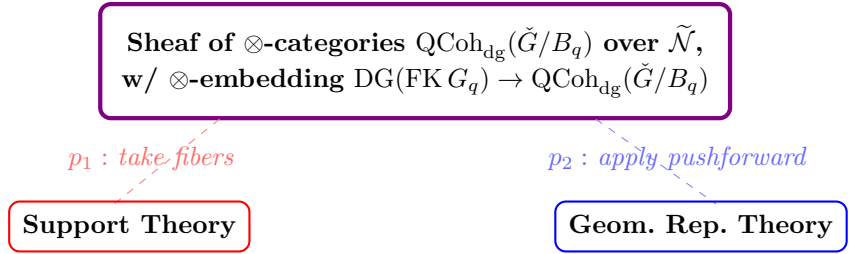
\[
\begin{tikzpicture}
\node[align=center,draw=violet, ultra thick, rounded corners, inner sep=8pt] at (3,5) {\bf Sheaf of $\ot$-categories $\QCoh_{\dg}(\dG/B_q)$ over $\tN$,\\ \bf w/ $\ot$-embedding $\operatorname{DG}(\FK{G}_q)\to \QCoh_{\dg}(\dG/B_q)$};
\draw[color=red!50, dashed]  (.4,4.23) -- (-1,3.1);
\draw[color=blue!50, dashed] (5.4,4.23) -- (7,3.1);

\node[align=center,draw=red, thick, inner sep=5pt, rounded corners] at (-.8,2.8) {\bf Support Theory};
\node[align=center,color=red!60] at (-.5,3.7) {$p_1:$\ \emph{take fibers}};
\node[align=center,color=blue!60] at (6.5,3.7) {$p_2:$\ \emph{apply pushforward}};
\node[align=center,draw=blue, thick, inner sep=5pt, rounded corners] at (6.8,2.8) {\bf Geom.\ Rep.\ Theory};
\end{tikzpicture}
\]

\section{Remarks on log-TQFTs and the derived modular functor}
\label{sect:log}

Below we make, what is intended to be, a casual comment about logarithmic topological quantum field theory.  We are particularly interested in the derived modular invariant of Schweigert and Woike \cite{schweigertwoike21}.  In short, we explain how our enhancement of the tensor category $\FK{G}_q$ may provide a means of resolving certain singularities which appear in the modular functor associated to the small quantum group, as described in \cite{schweigertwoike21}.  (See Observation \ref{obs:sing} below.)
\par

For the sake of expediency, we don't review any aspects of topological quantum field theory, but point the reader to the texts
\cite{bakalovkirillov01,turaev16}, and to the introduction of \cite{gukovetal21} for specific discussions of the logarithmic (a.k.a.\ non-semisimple) setting.

\subsection{A preliminary comment on even order $q$}
\label{sect:even_orderq}

As the reader may be aware, the particular quantum parameters $q$ which are of greatest interest in mathematical physics are of even order.  In the earlier sections of this text we have required $q$ to be of large \emph{odd} order.
\par

There are no essential problems which occur at even order $q$, save for some unknown properties of cohomology and the incompleteness of the analysis of \cite{negron21} at admissible but not strongly-admissible lattices.  The second point is, from our perspective, not a significant obstacle.  For the first point, one has to deal with a relatively minor issue concerning cohomology for the small quantum Borel.  Specifically, arguments of Ginzburg and Kumar \cite{ginzburgkumar93} imply a calculation of cohomology
\[
\Ext_{\uqB}(k,k)=\Sym(\mfk{w})
\]
at all large even order $q$, where $\mfk{w}$ is a representation for the Borel $\dB$ of Langland's dual type which is concentrated in cohomological degree 2.  (It is important, in this case, that one uses the quasi-Hopf algebra from \cite{arkhipovgaitsgory03,creutziggainutdinovrunkel20,gainutdinovlentnerohrmann,negron21} and not the more immediate construction of the small quantum group as a subalgebra in $U_q(\mfk{g})$.)
\par

One expects the representation $\mfk{w}$ to be the linear dual of the nilpotent subalgebra $\check{\mfk{n}}$ for $\dB$.  Supposing this fact, one simply replaces $\tN$ with the Springer resolution $\tN^\vee$ for the Langlands dual group in our analysis, and all of the results of this text will hold without further alteration at $q$ of large even order.
\par

\begin{remark}
If we simply accept $\mfk{w}$ as an unknown quantity, one can still replace the Springer resolution and nilpotent cone with the mystery space $\tN(\mfk{w}):=\dG\times_{\dB}\mfk{w}$ and its affinization $\mcl{N}(\mfk{w}):=\tN(\mfk{w})_{\rm aff}$, and all of the results will again hold at large even order $q$.
\end{remark}

\begin{remark}
For $\operatorname{SL}(2)$ at even order $q$, $\Ext^2_{u_q(B)}(k,k)$ is $1$-dimensional and so is determined by its associated weight.  In this case one sees immediately that $\Ext_{u_q(B)}(k,k)=\Sym(\check{\mfk{n}})$ and $\tN(\mfk{w})$ is the Springer resolution in type $A$.  So for $\operatorname{SL}(2)$ there are no barriers at even order $q$, and the analyses of this paper hold without alteration in this special case.
\end{remark}

\subsection{Derived modular functor}

Let $\msc{C}$ be a (generally non-semisimple) modular tensor category.  We consider the category $\msc{C}\text{-}\Surf$ of compact oriented surfaces with oriented boundary, and boundary labels in $\msc{C}$, as described in \cite[\S 3.1]{schweigertwoike21}.  Morphisms in $\msc{C}\text{-}\Surf$ are generated by mapping classes, i.e.\ isoclasses of orientation preserving diffeomorphisms, and sewing morphisms $s:\Sigma'\to \Sigma$.  Such $s$ simply chooses some oppositely oriented pairs of boundary components $\phi_i:S^1\to \partial\Sigma'$ with compatible labels, then identifies these pairs to form $\Sigma$.  We have the corresponding central extension $\msc{C}\text{-}\Surf^c\to\msc{C}\text{-}\Surf$ of $\msc{C}\text{-}\Surf$ obtained by appending a central generator to $\End_{\Surf}(\Sigma)$ which commutes with mapping classes and behaves multiplicatively with respect to disjoint union.  See \cite[\S 3.1]{schweigertwoike21} \cite[Definition 3.10]{fuchsschweigert17} for details.
\par

In \cite{schweigertwoike21} Schweigert and Woike construct a homological surface invariant for $\msc{C}$, which takes the form of a symmetric monoidal functor
\[
F_\msc{C}:\msc{C}\text{-}\Surf^c \to Vect_{\dg}.
\]
Given an expression $\msc{C}\cong \rep(u)$ for a Hopf algebra $u$, this invariant takes values $F_\msc{C}(\Sigma_g)\cong \1\ot^{\operatorname{L}}_u\operatorname{Ad}^{\ot g}$ on the genus $g$ closed surface $\Sigma_g$, where $\operatorname{Ad}=u^{ad}$ is the adjoint representation.  As a consequence, one observes mapping class group actions on homology for $\msc{C}$ (see also the preceding works \cite{lentneretal18,lentneretal}).
\par

When $\msc{C}$ is semisimple, the functor $F_\msc{C}$ of \cite{schweigertwoike21} reduces to an object which is dual to the usual modular functor described in \cite{bakalovkirillov01,turaev16}.  It is both linearly dual, in the sense that the values $F_\msc{C}(\Sigma)^\ast$ recover the expected values in the semisimple case \cite[Remark 3.12]{schweigertwoike21}, and also dual in the sense that it is constructed using the adjoint operation $\ot_u$ to the $\Hom$ functors employed in \cite{bakalovkirillov01,turaev16}.
\par

We suppose for this section that that there is a dual/adjoint modular functor $F^\msc{C}:\msc{C}\text{-}\Surf^c\to Vect_{\dg}$ whose values on higher genus surfaces, and standard spheres, are given by the expected derived Hom spaces
\[
\begin{array}{l}
F^\msc{C}(\Sigma_g)\cong \RHom_{\msc{C}}(\1,\operatorname{Ad}^{\ot g}),\vspace{1mm}\\
F^\msc{C}(S^2(V_1,\dots,V_r;W_1,\dots,W_s))\cong\RHom_{\msc{C}}(V_1\ot\dots\ot V_r,W_1\ot\dots\ot W_r),
\end{array}
\]
and whose sewing maps are determined by composition of morphisms (cf.\ \cite[\S V.2]{turaev16}).  If one is only concerned with gluing isomorphisms and modular group actions--as in \cite[Definition 5.1.13]{bakalovkirillov01} for example--one can take $F^\msc{C}$ directly to be the linear dual $(F_\msc{C})^\ast$, as remarked at \cite[(1.8)]{schweigertwoike}.

\begin{remark}
One can compare with the discussion of duality for state spaces, i.e.\ the values of our modular functor on surfaces, and related grading issues from \cite[\S 2.6.3]{creutzigdimoftegarnergeer}.  From the perspective of \cite{creutzigdimoftegarnergeer} the homological and cohomological flavors of the invariant $F_\msc{C}$ vs.\ $F^\msc{C}$ arise from a certain binary choice of orientation.
\end{remark}

\begin{definition}
For the quantum group $\FK{G}_q$ we take $\Surf^c(q)=\FK{G}_q\text{-}\Surf^c$ and let $F_q$ denote the cohomological-type modular functor $F_q=F^{\FK{G}_q}:\Surf^c_q\to Vect_{\dg}$ described above.
\end{definition}

\subsection{Singularities in the modular surface invariant}

In paralleling the semisimple setting, one might think of the invariant $F_q$, aspirationally, as a $2$-dimensional slice of a derived Reshetikhin-Turaev invariant associated to the small quantum group.  So we are interested in extending the given invariant $F_q$ to dimension three, in some possibly intricate way, to produce a once, or twice-extended $3$-dimensional logartithmic TQFT.
\par

If one proceeds na\"ively, then one considers a possible extension of $F_q$ to a functor
\[
\dot{F}_q:\operatorname{Bord}_{3,2,1}^c(q)\to  Vect_{\dg}
\]
from a marked bordism category in which $\Surf^c(q)$ resides in dimensions $2\pm \epsilon$.  One can see, however, that no such na\"ive extension $\dot{F}_q$ exists.  (This point is well-known, see for example \cite[Remark 4.16]{schweigertwoike}.)  Namely, such an extension would imply that all values $F_q(\Sigma)$ on surfaces $\Sigma$ without boundary are dualizable objects in $Vect_{\dg}$.
\par

To spell this out more clearly, if we cut $S^1$ along the $y$-axis $S^1_{-}\cup S^1_{+}\to S^1$ then we obtain coevaluation and evaluation morphisms for $F_q(\Sigma)$ via the maps
\[
\dot{F}_q(\Sigma\times S^1_-):k\to F_q(\Sigma)\ot_k F_q(-\Sigma),\ \ \dot{F}_q(\Sigma\times S^1_+):F_q(-\Sigma)\ot_k F_q(\Sigma)\to k.
\]
However, the values $F_q(\Sigma)$ are immediately seen to be non-dualizable since they are complexes with unbounded cohomology in general.
\par

From a slightly less na\"ive perspective, one can ask for dualizability of the objects $F_q(\Sigma)$ relative to a global action of the codimension one sphere $F_q(S^2)$, and in this way attempt to extract certain, slightly exotic, $3$-dimensional information from the invariant $F_q$.  Indeed, from a $3$-dimensional perspective one expects to obtain natural action maps
\[
F_q(S^2)\ot_k F_q(\Sigma)\to F_q(\Sigma)
\]
by puncturing, then placing an appropriately oriented sphere, in the identity $\dot{F}_q(\Sigma\times [-1,1])$.  The most obvious candidate for this global $F_q(S^2)$-action on $F_q(\Sigma)$ would be the global action of $F_q(S^2)\cong \O(\mcl{N})$ on $\RHom_{\FK{G}_q}$ induced by the braided tensor structure of $\FK{G}_q$.  The modular functor for $\FK{G}_q$ can then be reframed as the corresponding system of functors
\begin{equation}\label{eq:3083}
\xymatrix{
 & \Coh_{\dg}(\mcl{N})\ar[dr]^{\operatorname{R}\Gamma}\\
\Surf^c(q)\ar@{-->}[ur]|{\rm conj\ \exists?}\ar[rr]_{F_q} & & Vect_{\dg}.
}
\end{equation}
Here by $\Coh_{\dg}(\mcl{N})$ we mean the expected monoidal category of sheaves with a quantized symmetry, as in Section \ref{sect:formality}, and we note that the lift in \eqref{eq:3083} is \emph{not} supposed to be monoidal.  Indeed, this lift is subordinate to the monoidal functor $F_q$.
\par

With this new framing in mind, one asks for dualizability of $F_q(\Sigma)$ not as an object in $Vect_{\dg}$, but as an object in $\Coh_{\dg}(\mcl{N})$ and pursues an extension of the form
\begin{equation}\label{eq:3093}
\ddot{F}_q:\operatorname{Bord}^c_{3,2,1}(q)\rightsquigarrow \{\text{Morphisms in }\Coh_{\dg}(\mcl{N}),\ Vect_{\dg},\ \dots
\end{equation}
whose values on the products $\Sigma\times S^1$, for a surface $\Sigma$ without boundary, are traces
\[
\ddot{F}_q(\Sigma\times S^1)=\operatorname{Tr}_{\O_{\mcl{N}}}(F_q(\Sigma))
\]
in $\Coh_{\dg}(\mcl{N})$.  (See Remark \ref{rem:BZ} below.)
\par

However, one can see that such an extension \eqref{eq:3093} is still obstructed due to the singular nature of the nilpotent cone.  Namely, the adjoint representation $\operatorname{Ad}=\uqG^{ad}$ always has non-trivial projective summands \cite{ostrik97} \cite[Corollary 4.5]{lachowskaqi21}, so that $\RHom_{\FK{G}_q}(\1,\operatorname{Ad})$ has an $m$-torsion summand, for $m=\O(\mcl{N})^{>0}$.  This torsion summand is non-dualizable since $\mcl{N}$ is singular at $\{0\}$, so that $F_q(T^2)$ is non-dualizable.

\begin{observation}[Singularities of $F_q$]\label{obs:sing}
The existence of a lift $F_q:\Surf^c(q)\to \Coh_{\dg}(\mcl{N})$ and corresponding extension $\ddot{F}_q$ to dimension $3$, as requested in \eqref{eq:3083} and \eqref{eq:3093}, is still obstructed.  Namely, the value $F_q(T^2)\cong \RHom_{\FK{G}_q}(\1,\operatorname{Ad})$ \cite[Theorem 3.6]{schweigertwoike21} on the torus is never dualizable as an object in $\Coh_{\dg}(\mcl{N})$
\end{observation}

The main point of this section is to suggest a world in which such singularity problems for a derived Reshetikhin-Turaev theory may be resolvable.

\begin{remark}
In terms of the calculation of the moduli of vacua for quantum $\operatorname{SL}(2)$ provided in \cite{costellocreutziggaitto19,gainutdinovetal06,gannonnegron}, the invariant $F_q$ looks like a particularly reasonable candidate for the ``logarithmic $3$-d" TQFT associated to quantum $\operatorname{SL}(2)$, and its corresponding (triplet) conformal field theory \cite{feigintipunin,gannonnegron}, when compared with the sibling constructions of \cite{geeretal09,costantinoetal14,derenzietal}.  Hence our particular interest in this object.  It seems likely, however, that ideas from both the derived $2$-dimensional school of thought \cite{schweigertwoike21} and the more geometric $3$-dimensional school of thought \cite{geeretal09,costantinoetal14,derenzietal} may be necessary in order to construct a (semi-)functional $3$-dimensional TQFT in the logarithmic setting. 
\end{remark}

\begin{remark}\label{rem:BZ}
We first came across the relative dualizability criterion for logarithmic theories, employed above, in talks of Ben-Zvi.  One can see for example \cite{benzvi}, where the idea is furthermore attributed to Walker.
\end{remark}

\subsection{Resolving singularities in $F_q$}
\label{sect:sub_fin}

Let us suggest a possible means of resolving the singular properties of $F_q$, as described in Observation \ref{obs:sing} above.  Very directly, one can pursue a sheaf-analog of the construction of \cite{schweigertwoike1,schweigertwoike21} to produce a modular functor
\begin{equation}\label{eq:3122}
\msc{F}_q:\Surf^c(q)\to \QCoh_{\dg}(\dG/\dB)
\end{equation}
which takes values in the category of sheaves over the flag variety, and whose affinization (a.k.a.\ global sections) recovers the functor $F_q$.  On closed surfaces and standard spheres one expects such a localization of $F_q$ to admit non-canonical identification
\[
\begin{array}{l}
\msc{F}_q(\Sigma_g)\cong \sRHom_{\FK{G}_q}(\1,\operatorname{Ad}^{\ot g}),\vspace{1mm}\\
\msc{F}_q(S^2(M_1,\dots,M_r;N_1,\dots,N_s))\cong\sRHom_{\FK{G}_q}(M_1\ot\dots\ot M_r,N_1\ot\dots\ot N_r).
\end{array}
\]
\par

To suggest some explicit method of construction here, since all of the sheaves $\sHom_{\FK{G}_q}(P,P')$ are flat over $\dG/\dB$ at projective $P$ and $P'$ (Lemma \ref{lem:1399}), the excision properties for cohomology defined via the homotopy (co)end construction should still hold \cite[\S 3.1]{schweigertwoike} \cite[Remark 2.12]{schweigertwoike1}, so that we can recover cohomology via a homotopy end construction of $\msc{F}_q$ as in \cite[Eq. 2.1, Theorem 3.6]{schweigertwoike21}.  One then recovers the modular functor $F_q$ with linear values as the global sections
\[
F_q=\operatorname{R}\Gamma(\dG/\dB,\msc{F}_q),
\]
via Corollary \ref{cor:372}.  Given such a localization $\msc{F}_q$ of $F_q$, the values $\msc{F}_q(\Sigma)$ should then be calculated via derived morphism spaces
\[
\msc{F}_q(\Sigma)\cong \sRHom_{\FK{G}_q}(M,N\ot \operatorname{Ad}^{\ot g})
\]
which admit an action of the algebra $\msc{R}(G_q)=\sRHom_{\FK{G}_q}(\1,\1)$, as in \cite[Theorem 3.6]{schweigertwoike}.  The cohomology $H^\ast(\msc{F}_q(\Sigma))$ will then be an object in $\Coh(\tN)^{\mbb{G}_m}$ by Proposition \ref{prop:coherent}.
\par

Finally, if we accept Conjecture \ref{conj:formality}, then the object $\msc{F}_q(\Sigma)$ itself will be realized as a coherent dg sheaf over the Springer resolution, and the picture \eqref{eq:3083} now appears as the system of relations
\begin{equation}\label{eq:3150}
\xymatrix{
 & \Coh_{\dg}(\tN)\ar[dr]^{p_\ast}\\
\Surf^c(q)\ar@{-->}[ur]^{\tilde{\msc{F}}_q}\ar[rr]_{\msc{F}_q} & & \QCoh_{\dg}(\dG/\dB)
}
\end{equation}
which characterizes a more sensitive version of our logarithmic modular functor associated to the quantum group.  In this more sensitive construction the (at this point mythical) values of the lift $\tilde{\msc{F}}_q$ will be dualizable, since $\tN$ is smooth, so that such a localization desingularizes the original invariant $F_q$.  As suggested in Section \ref{sect:small_GRT}, and in the works \cite{arkhipovbezrukavnikovginzburg04,bezrukavnikov06}, lifting to the Springer resolution should also have some concrete computational advantages as well.

\begin{remark}\label{rem:boundary}
Though the suggestions of this subsection \emph{may} be sound from a purely mathematically perspective, one should still ask for a physical interpretation of the kind of $\dG/\dB$-localization we are suggesting.  While a definitive conclusion has not reached in this regard, we can provide some commentary.
\par

From one perspective, one considers the TQFT associated to the small quantum group as derived from a topological twist of a corresponding $3$-dimensional supersymmetric quantum field theory $T_{G,\kappa}$ for $G$ at a particular level $\kappa$.\footnote{We are being liberal in our use of the label ``TQFT" here.  As far as we understand, the partition functions (numerical manifold invariants) obtained from such twisted field theories--if such things even exists--are not, at present, computationally accessible via the associated QFT.}   This is the approach taken in \cite{creutzigdimoftegarnergeer}, for example.  Theory $T_{G,k}$ has Coulomb branch $\mathcal{M}_C=\mathcal{N}$ which gives rise to the value $F_q(S^2)=\O(\mathcal{N})$.  From the perspective of theory $T_{G,\kappa}$, one resolves the nilpotent cone $\tN\to \mathcal{N}$ by ``turning on" associated real mass parameters.  See for example \cite{bullimoredimoftegaiotto17,bullimoreferrarikim19}.  We do not know how these mass parameters persist through topological twisting however, or how they might contribute to the corresponding TQFT.
\par

For another vantage point--that of the associated conformal field theories--one simply has the sheaf of VOAs $\msc{V}_l(G)=\dG\times_{\dB} V_{\sqrt{l}Q}$ over $\dG/\dB$ which has global sections identified with the corresponding logarithmic $\mathbb{W}$-algebra $\Gamma(\dG/\dB,\msc{V}_l(G))=W_l(G)$ for $G$ at $q$ \cite{feigintipunin,sugimoto21}.  It is conjectured that these logarithmic $\mathbb{W}$-algebras have module categories equivalent to those of the corresponding small quantum group $\FK{G}_q$ at a $2l$-th root of $1$ \cite{feigintipunin,lentner21}, and this conjecture has been verified for $\operatorname{SL}(2)$ \cite{creutziglentnerrupert,gannonnegron}.  So the VOA $\msc{V}_l(G)$ localizes the quantum group over the flag variety in a rather direct sense, although the relationship with the categorical localization discussed in this text is not entirely clear.  These points of confusion around the flag variety are actually quite interesting from our perspective, and may very well be worth investigating in their own right.
\end{remark}

\appendix

\section{Inner-Hom nonsense for Parts I \& II}
\label{sect:A}

\subsection{Proof outline for Proposition \ref{prop:enriched}}

\begin{proof}[Outline for Proposition \ref{prop:enriched}]
If we take $\msc{F}=\sHom_{\dG/B_q}(M,N)$, composition is alternatively defined by the $\QCoh(\dG/\dB)$-linearity $\msc{F}\star \sHom(L,M)\to \sHom(L,\msc{F}\star M)$ composed with evaluation $ev:\msc{F}\star M\to N$ in the second coordinate.  One uses these two descriptions to deduce associativity of composition.
\par

Associativity for the monoidal structure follows from the fact that both maps
\[
\sHom(M_1,N_2)\ot \sHom(M_2,N_2)\ot \sHom(M_3,N_2)\rightrightarrows \sHom(M_1\ot M_2\ot M_2, N_1\ot N_2\ot N_3)
\]
are adjoint to the map
\[
\begin{array}{l}
\sHom(M_1,N_2)\ot \sHom(M_2,N_2)\ot \sHom(M_3,N_2)\star (M_1\ot M_2\ot M_3)\vspace{2mm}\\
\hspace{1cm}\overset{symm}\longrightarrow (\sHom(M_1,N_2)\star M_1)\ot (\sHom(M_2,N_2)\star M_2)\ot (\sHom(M_3,N_2)\star M_3)\vspace{2mm}\\
\hspace{2cm}\overset{comp}\longrightarrow N_1\ot N_2\ot N_3.
\end{array}
\]
Compatibilities with composition appears as an equality
\begin{equation}\label{eq:1268}
(g_1\ot g_2)\circ (f_1\ot f_2)=(g_1\circ f_1)\ot (g_2\circ f_2):\msc{G}_1\ot \msc{F}_1\ot \msc{G}_2\ot\msc{F}_2\to \sHom(L_1\ot L_2,N_1\ot N_2)
\end{equation}
where $f_i$ and $g_i$ are ``generalized sections", i.e.\ maps
\[
f_i:\msc{F}_i\to \sHom(L_i,M_i),\ \ g_i:\msc{G}_i\to \sHom(M_i,N_i).
\]
To equate these two sections \eqref{eq:1268} one must equate the corresponding morphisms
\[
(\msc{G}_1\ot \msc{F}_1\ot \msc{G}_2\ot \msc{F}_2)\star(L_1\ot L_2)\to N_1\ot N_2,
\]
which involve various applications of the half-braiding for $\QCoh(\dG/\dB)$ acting on $\QCoh(\dG/B_q)$.  One represents these two morphisms via string diagrams and observes the desired equality via naturality of the half-braiding for the $\QCoh(\dG/\dB)$-action.
\end{proof}

\subsection{Proof of Theorem \ref{thm:enhA}}

\begin{lemma}
The adjunction isomorphism
\[
\Hom_{\dG/B_q}(M,N)\to \Hom_{\dG/\dB}(\O_{\dG/\dB},\sHom(M,N))=\Gamma(\dG/\dB,\sHom(M,N))
\]
is precisely the global sections of the natural map \eqref{eq:978}.
\end{lemma}

\begin{proof}
The adjunction map sends a morphism $f:M\to N$ to the unique map $\O_{\dG/\dB}\to \sHom(M,N)$ for which the composite
\[
M=\O_{\dG/\dB}\star M\to \sHom(M,N)\star M\overset{ev}\to N
\]
is the morphism $f$.  Let us call this section $f:\O_{\dG/\dB}\to \sHom(M,N)$.  By considering the fact that we have the surjective map of sheaves
\[
\oplus_{f\in \Hom(M,N)}\O_{\dG/\dB}\star M\overset{\oplus f\ot id}\longrightarrow \Hom(M,N)\ot_k M
\]
we see that the above property implies that the uniquely associated map
\[
\Hom(M,N)\ot_k\O_{\dG/\dB}\to \sHom(M,N)
\]
whose global sections are the adjunction isomorphism fits into the diagram \eqref{eq:982}, and is therefore equal to the morphism \eqref{eq:978}.
\end{proof}

We now prove our theorem.

\begin{proof}[Proof of Theorem \ref{thm:enhA}]
Compatibility with evaluation \eqref{eq:982} implies that the restrictions of the composition and tensor structure on $\sHom_{\dG/B_q}$ along the inclusion
\[
a:\Hom_{\dG/B_q}(M,N)\subset \Hom_{\dG/B_q}(M,N)\ot_k\O_{\dG/B_q}\to \sHom_{\dG/B_q}(M,N)
\]
provided by adjunction recovers the composition and tensor structure maps for $\Hom_{\dG/B_q}(M,N)$.  (Here $\Hom_{\dG/B_q}(M,N)$ denotes the constant sheaf.)  For composition for example we understand, via \eqref{eq:982} and the manner in which composition and evaluation are related for $\Hom_{\dG/B_q}$, that the map
\begin{equation}\label{eq:1022}
\Hom(M,N)\ot_k\Hom(L,M)\overset{\circ}\longrightarrow \Hom(L,N)\overset{a}\to \sHom(L,N)
\end{equation}
is the unique one so that the composite
\[
\Hom(M,N)\ot_k\Hom(L,M)\ot_k L\to \sHom(L,N)\ot L\overset{ev}\to N
\]
is just the square of the $k$-linear evaluation map
\[
\Hom(M,N)\ot_k\Hom(L,M)\ot_k L\to \Hom(M,N)\ot_k M\to N
\]
But by \eqref{eq:982} this second map is equal to the composite
\[
\Hom(M,N)\ot_k\Hom(L,M)\ot_k L\overset{a\ot a\ot id}\longrightarrow \sHom(M,N)\ot \sHom(L,M)\ot L\overset{ev^2}\to N
\]
Hence \eqref{eq:1022} is equal to the map
\[
\Hom(M,N)\ot_k\Hom(L,M)\overset{a\ot a}\to \sHom(M,N)\ot \sHom(L,M)\overset{\circ}\to \sHom(L,N),
\]
which just says that restricting along the adjunction map $a$ recovers composition for $\Hom_{\dG/B_q}$ via the global sections of composition for $\sHom_{\dG/B_q}$.  Compatibility with evaluation also implies that the aforementioned map between $\Hom$ spaces respects the $\dG$-action.
\end{proof}

\subsection{Proof of Proposition \ref{prop:1660}}
\label{sect:1660}

We consider the category $\Rep\mcl{M}_q$ of simultaneous, compatible $B_q$ and $\uqG$-representations (cf.\ \cite[\S 3.12]{andersenparadowski95}), and the corresponding category $\QCoh(\dG/\mcl{M}_q)$ of mixed equivariant sheaves.  This category is braided via the $R$-matrix for the quantum group, and admits a central embedding from $\QCoh(\dG/\dB)$ via quantum Frobenius.  The forgetful functor $\QCoh(\dG/\mcl{M}_q)\to \QCoh(\dG/B_q)$ is central and $\QCoh(\dG/\dB)$-linear, and the Kempf embedding from $\FK{G}_q$ factors
\[
\FK{G}_q\to \QCoh(\dG/\mcl{M}_q)\to \QCoh(\dG/B_q)
\]
through a braided embedding $\FK{G}_q\to \QCoh(\dG/\mcl{M}_q)$.
\par

The important point of $\QCoh(\dG/\mcl{M}_q)$, or rather the derived category $D(\dG/\mcl{M}_q)$ is that it admits a braiding extending the braidings on $D(\dG/\dB)$ and $D(\FK{G}_q)$, simultaneously, and contains all products $\msc{F}\star M$ for $\msc{F}$ in $D(\dG/B_q)$ and $M$ in $D(\FK{G}_q)$.  We note that the evaluation map
\[
ev:\sRHom_{\dG/B_q}(M,N)\star M\to M
\]
for $M$ and $N$ in $D(\dG/\mcl{M}_q)$ is only a morphism in $D(\dG/B_q)$, and needn't lift along the central functor $D(\dG/\mcl{M}_q)\to D(\dG/B_q)$.  However, via centrality we can braid the evaluation against any object $L$ in $D(\dG/\mcl{M}_q)$ to get
\[
\operatorname{braid}_{L,N}(id_L\ot ev)=(ev\ot id_L)\operatorname{braid}_{L,\msc{R}\star M}.
\]

\begin{proof}[Proof of Proposition \ref{prop:1660}]
Take $\msc{R}=\msc{R}(G_q)$.  We have the two string diagrams
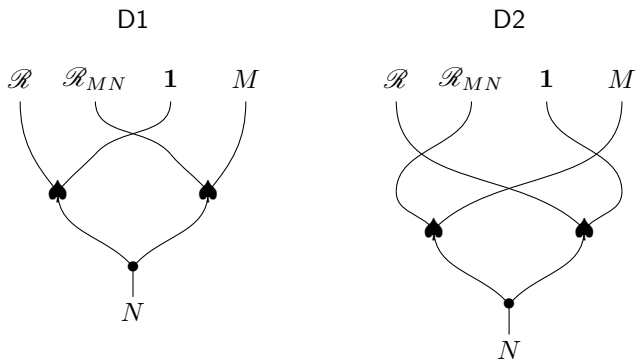
\[
\begin{tikzpicture}
\node at (-2.5,5.8) {$\mathsf{D1}$};
\node at (2.5,5.8) {$\mathsf{D2}$};
\node at (-4,5) {$\msc{R}$};
\node at (-3,5) {$\msc{R}_{MN}$};
\node at (-2,5) {$\1$};
\node at (-1,5) {$M$};
\node at (-3.5,3.5) {$\spadesuit$};
\node at (-1.5,3.5) {$\spadesuit$};
\node at (-2.5,2.5) {$\bullet$};
\node at (-2.5,1.9) {$N$};
\draw (-2,4.7) to [out=-90, in=50] (-3,4);
\draw (-3,4) to [out=-135, in=50] (-3.5,3.5);
\draw (-4,4.7) to [out=-90, in=130] (-3.5,3.5);
\draw (-3,4.7) to [out=-90, in=130] (-2,4);
\draw (-2,4) to [out=-50, in=130] (-1.5,3.5);
\draw (-1,4.7) to [out=-90, in=50] (-1.5,3.5);
\draw (-1.5,3.5) to [out=-90, in=50] (-2.5,2.5);
\draw (-3.5,3.5) to [out=-90, in=130] (-2.5,2.5);
\draw (-2.5,2.5) to [out=-90, in=90] (-2.5,2.1);
\node at (1,5) {$\msc{R}$};
\node at (2,5) {$\msc{R}_{MN}$};
\node at (3,5) {$\1$};
\node at (4,5) {$M$};
\node at (3.5,3) {$\spadesuit$};
\node at (1.5,3) {$\spadesuit$};
\node at (2.5,2) {$\bullet$};
\node at (2.5,1.4) {$N$};
\draw (1,4.7) to [out=-90, in=135] (3.5,3);
\draw (2,4.7) to [out=-90, in=90] (1,3.5);
\draw (1,3.5) to [out=-90, in=130] (1.5,3);
\draw (3,4.7) to [out=-90, in=90] (4,3.5);
\draw (4,3.5) to [out=-90, in=50] (3.5,3);
\draw (4,4.7) to [out=-90, in=45] (1.5,3);
\draw (3.5,3) to [out=-90, in=50] (2.5,2);
\draw (1.5,3) to [out=-90, in=130] (2.5,2);
\draw (2.5,2) to [out=-90, in=90] (2.5,1.6);
\end{tikzpicture}
\]
which represent the top morphism in \eqref{eq:diagram} and the bottom arc in \eqref{eq:diagram} respectively, where a node labeled $\spadesuit$ is an evaluation morphism and a node labeled $\bullet$ is the unit structure.  We cross strands via the braiding, and note that at least one object at each crossing is (M\"uger) central, so that the direction of crossing is irrelevant.  We claim that diagram $\mathsf{D1}$ is equivalent to the diagram $\mathsf{D2}$, in the sense that their associated morphisms are equal in $D(\dG/B_q)$.
\par

The diagram $\mathsf{D2}$ can be manipulated via the central structure on the functor $D(\dG/\mcl{M}_q)\to D(\dG/B_q)$ to get
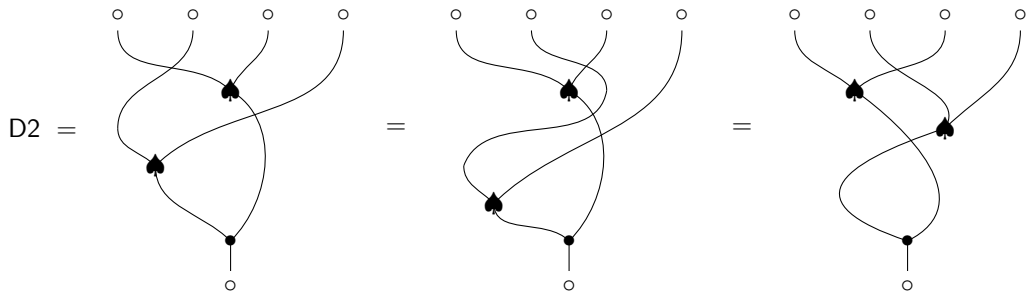
\[
\begin{tikzpicture}
\node at (-4,5) {$\circ$};
\node at (-3,5) {$\circ$};
\node at (-2,5) {$\circ$};
\node at (-1,5) {$\circ$};
\node at (-3.5,3) {$\spadesuit$};
\node at (-2.5,4) {$\spadesuit$};
\node at (-2.5,2) {$\bullet$};
\node at (-2.5,1.4) {$\circ$};
\draw (-4,4.8) to [out=-90, in=135] (-2.5,4);
\draw (-3,4.8) to [out=-90, in=90] (-4,3.5);
\draw (-4,3.5) to [out=-90, in=135] (-3.5,3);
\draw (-2,4.8) to [out=-90, in=70] (-2.5,4);
\draw (-2.5,4) to [out=-30, in=45] (-2.5,2);
\draw (-1,4.8) to [out=-90, in=45] (-3.5,3);
\draw (-3.5,3) to [out=-90, in=135] (-2.5,2);
\draw (-2.5,2) to [out=-90, in=90] (-2.5,1.6);
\node at (.5,5) {$\circ$};
\node at (1.5,5) {$\circ$};
\node at (2.5,5) {$\circ$};
\node at (3.5,5) {$\circ$};
\node at (2,1.4) {$\circ$};
\node at (2,4) {$\spadesuit$};
\node at (1,2.5) {$\spadesuit$};
\node at (2,2) {$\bullet$};
\draw (.5,4.8) to [out=-90, in=135] (2,4);
\draw (2.5,4.8) to [out=-90, in=70] (2,4);
\draw (2,4) to [out=-30, in=45] (2,2);
\draw (1.5,4.8) to [out=-90, in=90] (2.5,4);
\draw (2.5,4) to [out=-100, in=70] (.6,3);
\draw (.6,3) to [out=-90, in=135] (1,2.5);
\draw (3.5,4.8) to [out=-90, in=45] (1,2.5);
\draw (1,2.5) to [out=-90, in=135] (2,2);
\draw (2,2) to [out=-90, in=90] (2,1.6);
\node at (5,5) {$\circ$};
\node at (6,5) {$\circ$};
\node at (7,5) {$\circ$};
\node at (8,5) {$\circ$};
\node at (6.5,2) {$\bullet$};
\node at (5.8,4) {$\spadesuit$};
\node at (7,3.5) {$\spadesuit$};
\node at (6.5,1.4) {$\circ$};
\draw (5,4.8) to [out=-90, in=135] (5.8,4);
\draw (7,4.8) to [out=-90, in=45] (5.8,4);
\draw (5.8,4) .. controls (6.5,3.5) and (7.5, 2.5) .. (6.5,2);
\draw (6,4.8) to [out=-90, in=60] (7,3.5);
\draw (8,4.8) to [out=-90, in=45] (7,3.5);
\draw (7,3.5) .. controls (5.5,3) and (5, 2.5) .. (6.5,2);
\draw (6.5,2) to (6.5,1.6);
\node at (-.3,3.5) {$=$};
\node at (4.3,3.5) {$=$};
\node at (-5,3.5) {$\mathsf{D2}\ =$};
\end{tikzpicture}
\]
Since
\[
\operatorname{unit}_{N,\1}\circ \operatorname{braid}_{\1,N}=\operatorname{unit}_{\1,N}:\1\ot N\to N.
\]
the final diagram is in equivalent to the diagram $\mathsf{D1}$ above.  This verifies commutativity of the diagram \eqref{eq:diagram}.
\end{proof}

\bibliographystyle{abbrv}

\end{document}